\documentclass[a4paper,10pt]{article}
\usepackage{graphicx}
\usepackage{amsmath}
\usepackage{amssymb}
\usepackage{amsfonts}
\usepackage{amsthm}
\usepackage[all]{xy}

\newtheorem{theorem}{Theorem}[section]
\newtheorem{lemma}[theorem]{Lemma}
\newtheorem{proposition}[theorem]{Proposition}
\newtheorem{corollary}[theorem]{Corollary}

\theoremstyle{definition}
\newtheorem{definition}{Definition}[theorem]
\newtheorem{example}{Example}[definition]
\newtheorem{convention}{Convention}[definition]
\newtheorem{remark}{Remark}[definition]

\DeclareMathAlphabet{\mathpzc}{OT1}{pzc}{m}{it}

\begin{document}

\begin{center}

{ \LARGE \bfseries Algebras over $\infty$-operads}\\[0.5cm]
\textsc{ \large  Gijs Heuts } \\ [2cm]

\end{center}

\author{ Gijs Heuts \\ Harvard University \\ gheuts@math.harvard.edu}

\begin{abstract}
We develop a notion of an \emph{algebra over an $\infty$-operad $X$ with values in $\infty$-categories} which is completely intrinsic to the formalism of dendroidal sets. Its definition involves the notion of a \emph{coCartesian fibration of dendroidal sets} and extends Lurie's definition of a coCartesian fibration of simplicial sets. We show how, for a dendroidal set X, the coCartesian fibrations over $X$ fit together to form an $\infty$-category $\mathbf{coCart}(X)$. Using a generalization of the Grothendieck construction, we prove that $\mathbf{coCart}(X)$ is equivalent to the $\infty$-category of algebras in $\infty$-categories over the simplicial operad $\mathrm{hc}\tau_d(X)$ associated to $X$. This equivalence can be restricted to give an equivalence between algebras taking values in $\infty$-groupoids (or equivalently, spaces) and the $\infty$-category of so-called \emph{left fibrations} over $X$.
\end{abstract}

\tableofcontents

\newpage

\section*{Introduction}

Since their conception in the 70's through the works of Boardman and Vogt \cite{boardmanvogt} and May \cite{may} operads have been an important tool for studying algebraic structures in homotopy theory. The reason for this is that topological operads are capable of describing algebraic structures up to coherent higher homotopy. The classical examples are the little $k$-cubes operads $E_k$ which capture the algebraic structure of $k$-fold loop spaces. \par 
In recent years homotopy-theoretic methods have been infused into other areas of mathematics by the works of, amongst many others, Lurie \cite{dagv}\cite{htt}\cite{higheralgebra}  and To\"en and Vezzosi \cite{hag1}\cite{hag2}, in particular resulting in the subject of derived algebraic geometry. This infusion relies heavily on the methods of higher category theory, specifically that of so-called \emph{$\infty$-categories}. These provide a generalization of category theory in which one investigates structures up to coherent higher homotopy. As the reader might well know, one way of building such a formalism is by using topological categories, but it turns out that in practice one would often like to relax the strictness of composition of morphisms that exists in these categories. Joyal \cite{joyal} and Lurie \cite{htt} have developed a more flexible approach to higher category theory using the language of simplicial sets. When one tries to consider algebraic structures in this setting, the question of defining \emph{$\infty$-operads} surfaces naturally. In particular, since a category is a very specific kind of operad (namely one with only unary operations), the theory of $\infty$-operads should in fact subsume the theory of $\infty$-categories.
\par 
There is an obvious approach to such a theory, namely that of topological operads. However, as mentioned above, such a theory is too strict for many purposes. Two alternatives have been developed: one by Lurie \cite{higheralgebra}, the other by Moerdijk and Weiss \cite{moerdijkweiss2} through the use of \emph{dendroidal sets}. These dendroidal sets generalize simplicial sets in a natural way and produce an efficient framework for studying $\infty$-operads and hence algebraic structures up to coherent higher homotopy. A necessity for any such theory is a description of algebras over an $\infty$-operad. This paper suggests an approach to such a definition in a way that is completely intrinsic to the formalism of dendroidal sets. \par 
We take our inspiration from the classical Grothendieck construction. Given a category $C$ and a pseudofunctor
\begin{equation*}
\mathcal{F}: C \longrightarrow \mathbf{Cat}
\end{equation*}
to the category of (small) categories, we can associate to this a new category, called the \emph{Grothendieck construction on $\mathcal{F}$}, as follows. We define a category $\int_C \mathcal{F}$ whose objects are given by pairs $(c, x)$, where $c$ is an object of $C$ and $x$ is an object in the category $\mathcal{F}(c)$. A morphism $(c,x) \rightarrow (d,y)$ is a pair $(f, \phi)$ where $f: c \rightarrow d$ is a morphism of $C$ and $\phi$ is a morphism $(\mathcal{F}f)(x) \rightarrow y$ in $\mathcal{F}(d)$. There is an obvious projection functor
\begin{equation*}
\pi_{\mathcal{F}}: \int_C \mathcal{F} \longrightarrow C
\end{equation*}
This projection has several special properties, which make it into something called a \emph{cofibered category}. It turns out that the Grothendieck construction establishes an equivalence between the theory of pseudo-functors from $C$ to $\mathbf{Cat}$ and cofibered categories over $C$. This construction is of great use in the study of stacks, where it originated. \par 
Lurie extensively studies the generalization of this construction to the setting of $\infty$-categories in \cite{htt}. The higher-categorical analogue of a cofibered category is a \emph{coCartesian fibration} of simplicial sets. Lurie then establishes an equivalence between the theory of such coCartesian fibrations and the theory of functors from an $\infty$-category $C$ to the $\infty$-category $\mathbf{Cat}_\infty$ of small $\infty$-categories. \par 
The generalization of a functor to the setting of operads is the concept of an algebra. Since the theory of topological operads is equivalent to the study of $\infty$-operads using dendroidal sets \cite{dsetsvssimploper}, we might as well study algebras over an $\infty$-operad $S$ by looking at algebras over the operad $\mathrm{hc}\tau_d(S)$, the topological operad associated to $S$ through the just mentioned equivalence. (This is to be made precise later on.) However, it turns out there is a more efficient and flexible approach to such algebras. We propose a definition of a \emph{coCartesian fibration of dendroidal sets}. We show how these fibrations fit together into an $\infty$-category
\begin{equation*}
\mathbf{coCart}(S)
\end{equation*} 
of coCartesian fibrations over $S$. We construct an analog of the Grothendieck construction in this setting, which we refer to as the \emph{$\infty$-operadic Grothendieck construction}. The main result of this text is Theorem \ref{thm:Grothendieckconstruction}, which, formulated somewhat imprecisely, says:
\begin{theorem}
The $\infty$-operadic Grothendieck construction gives an equivalence of $\infty$-categories
\begin{equation*}
\mathbf{coCart}(S) \simeq \mathrm{Alg}_{\mathrm{hc}\tau_d(S)}(\mathbf{Cat}_\infty)
\end{equation*}
where the category on the right denotes the category of $\mathrm{hc}\tau_d(S)$-algebras in $\infty$-categories.
\end{theorem}
The reader familiar with chapter 3 of Lurie's book \cite{htt} will see that our strategy of proof is very much inspired by his. \par 
The plan of this text is as follows:
\begin{itemize}
\item The first section discusses a generalization of the classical Grothendieck construction to the theory of operads in $\mathbf{Sets}$. This generalization is straightforward and will not be a surprise to anyone familiar with the Grothendieck construction.
\item The second section discusses algebras over an $\infty$-operad in $\infty$-groupoids, or spaces, using the concept of left fibrations. This section is expository, proofs are provided in later sections. 
\item The third section discusses the definition of a coCartesian fibration of dendroidal sets and gives its basic properties.
\item The fourth section is aimed at producing an $\infty$-category of coCartesian fibrations. To do this, we pass through the world of simplicial model categories. To be precise: for any dendroidal set $S$ we define a combinatorial simplicial model category of \emph{marked dendroidal sets over $S$}. The $\infty$-category $\mathbf{coCart}(S)$ will be the homotopy-coherent nerve of the full simplicial subcategory of this model category on the fibrant-cofibrant objects.
\item The fifth chapter discusses the $\infty$-operadic Grothendieck construction and establishes our main result, Theorem \ref{thm:Grothendieckconstruction}.
\item The sixth chapter discusses the naturality properties of the $\infty$-category $\mathbf{coCart}(S)$ with respect to morphisms $S \longrightarrow T$ and the definition of symmetric monoidal $\infty$-categories in terms of dendroidal sets. It also provides proofs for the results of Section 2 and mentions how these can be used to obtain an infinite loop space machine for $\infty$-operads.
\end{itemize} 

\subsection*{Acknowledgements}
Most of the work presented here was carried out during the writing of my thesis required for a Master's degree at Utrecht University. I would like to thank Urs Schreiber, Dave Carchedi and Vladimir Hinich for useful conversations relating to the subject matter of this text. I would also like to thank Ittay Weiss, who was the second reader of my thesis. Most of all, thanks are due to Ieke Moerdijk; without his countless ideas and suggestions this work would not have been possible.

\section*{Preliminaries}
Throughout this text the word operad will always mean a \emph{symmetric coloured operad}. An operad $P$ in a given closed symmetric monoidal category $\mathbf{\mathcal{E}}$ with tensor unit $I$ can be described by a set of colours $C$ and for any tuple of colours $(c_1, \ldots, c_n, c)$ an object
\begin{equation*}
P(c_1, \ldots, c_n; c)
\end{equation*} 
of $\mathbf{\mathcal{E}}$, which is to be thought of as the object of operations of $P$ with inputs $c_1, \ldots, c_n$ and output $c$. Furthermore, for $c \in C$, we should have a \emph{unit}
\begin{equation*}
I \longrightarrow P(c; c)
\end{equation*}
and we should have \emph{compositions}
\begin{equation*}
P(c_1, \ldots, c_n; c) \otimes P(d_1^1, \ldots, d_1^{j_1}; c_1) \otimes \ldots \otimes P(d_n^1, \ldots, d_n^{j_n}; c_n) \longrightarrow P(d_1^1, \ldots, d_n^{j_n}; c)
\end{equation*}
Finally, permutations $\sigma \in \Sigma_n$ should act by transformations
\begin{equation*}
\sigma^*: P(c_1, \ldots, c_n; c) \longrightarrow P(c_{\sigma(1)}, \ldots, c_{\sigma(n)}; c)
\end{equation*}
All of these data are required to satisfy various well-known associativity, unit and equivariance axioms. In the particular case where $\mathbf{\mathcal{E}} = \mathbf{Sets}$ and all but the sets of operations of $P$ corresponding to operations with one input are empty, we retrieve the definition of a category. In case we set $\mathbf{\mathcal{E}}$ equal to the category $\mathbf{Top}$ of (compactly generated) spaces or the category $\mathbf{sSets}$ of simplicial sets, we obtain definitions of topological resp. simplicial operads. These restrict to definitions of topological and simplicial categories by allowing only operads with unary operations, i.e. operations with exactly one input. We will denote the category of operads in $\mathbf{\mathcal{E}}$ by $\mathbf{Oper}_{\mathbf{\mathcal{E}}}$. In case $\mathbf{\mathcal{E}}$ is the category of simplicial sets we will deviate from this convention and write $\mathbf{sOper}$. Apologies for the inconsistency, but it is convenient. \par 
Any symmetric monoidal category $\mathbf{\mathcal{C}}$ enriched over $\mathbf{\mathcal{E}}$ gives rise to an operad $\underline{\mathbf{\mathcal{C}}}$ by
\begin{equation*}
\underline{\mathbf{\mathcal{C}}}(c_1, \ldots, c_n; c) := \mathbf{\mathcal{C}}(c_1 \otimes \cdots \otimes c_n, c)
\end{equation*}
Given an operad $P$ in $\mathbf{\mathcal{E}}$ an algebra over $P$ in $\mathbf{\mathcal{C}}$ is then simply a morphism of operads
\begin{equation*}
P \longrightarrow \underline{\mathbf{\mathcal{C}}}
\end{equation*}
We briefly review the basics of dendroidal sets, of which an extensive treatment can be found in \cite{dendroidalsets} and \cite{moerdijkweiss2}, and fix our notation. The category $\mathbf{\Omega}$ is defined to be the category of finite rooted trees. These are trees equipped with a distinguished outer vertex called the \emph{output} and a (possibly empty) set of outer vertices not containing the output called \emph{inputs}. When drawing trees, we will always omit out- and input vertices from the picture. Recall that each such rooted tree $T$ defines a $\mathbf{Sets}$-operad $\Omega(T)$, the free operad generated by $T$, which is coloured by the edges of $T$. A morphism of trees $S \longrightarrow T$ is defined to be a morphism of operads $\Omega(S) \longrightarrow \Omega(T)$. The category of dendroidal sets is defined to be the category of presheaves on $\mathbf{\Omega}$:
\begin{equation*}
\mathbf{dSets} := \mathbf{Sets}^{\mathbf{\Omega}^{\mathrm{op}}}
\end{equation*} 
The dendroidal set represented by a tree $T$ will be denoted by $\Omega[T]$. We will denote the set of $T$-dendrices of a dendroidal set $X$ by $X_T$. There is an embedding of the simplex category into the category of finite rooted trees, denoted
\begin{equation*}
i: \mathbf{\Delta} \longrightarrow \mathbf{\Omega}
\end{equation*}
defined by sending $[n]$ to the linear tree $L_n$ with $n+1$ edges and $n$ inner vertices. By left Kan extension this induces an adjunction
\[
\xymatrix@R=40pt@C=40pt{
i_!: \mathbf{sSets} \ar@<.5ex>[r] & \mathbf{dSets}: i^*\ar@<.5ex>[l]
}
\]
As is the case in the simplex category, any morphism in $\Omega$ may be factorized into \emph{face and degeneracy maps}. Relations between these maps extending the well-known relations in the simplex category are described in $\cite{dendroidalsets}$. An inner face of $T$ contracting an inner edge $e$ will be denoted $\partial_e\Omega[T]$, an outer face chopping off a corolla with vertex $v$ is denoted $\partial_v\Omega[T]$. We will sometimes abuse notation and write $\partial_e T$ and $\partial_v T$ for the trees in $\Omega$ corresponding to the shapes of these faces. \par 
We will call a tree with one vertex and $n$ leaves an \emph{$n$-corolla}. The tree with no vertices and only a single edge will be denoted by $\eta$. We will often blur the distinction between $\Omega[\eta]$ and $\eta$ and write $\eta$ for the former, or $\eta_c$ if we want to be explicit about the fact that the unique edge of $\Omega[\eta]$ has colour $c$. Note that we have an isomorphism
\begin{equation*}
\mathbf{dSets}/\eta \simeq \mathbf{sSets}
\end{equation*} 
A dendroidal set $X$ is said to be an \emph{$\infty$-operad} if it has the extension property with respect to all inner horn inclusions of trees. If $X$ is an $\infty$-operad then the simplicial set $i^*(X)$ is an $\infty$-category and a 1-corolla of $X$ is called an equivalence if the induced 1-simplex of $i^*(X)$ is an equivalence. \par 
The functor
\begin{equation*}
\mathbf{\Omega} \longrightarrow \mathbf{Oper_{Sets}}: T \longmapsto \Omega(T)
\end{equation*}
defines, by left Kan extension, an adjunction
\[
\xymatrix@R=40pt@C=40pt{
\tau_d: \mathbf{dSets} \ar@<.5ex>[r] & \mathbf{Oper_{Sets}}: N_d\ar@<.5ex>[l]
}
\]
The right adjoint $N_d$ is called the \emph{dendroidal nerve}. Recall that the category $\mathbf{Oper_{Sets}}$ carries a tensor product $\otimes_{BV}$ called the \emph{Boardman-Vogt tensor product}. For two representables $\Omega[S], \Omega[T] \in \mathbf{dSets}$ their tensor product is defined by
\begin{equation*}
\Omega[S] \otimes \Omega[T] := N_d(\Omega(S) \otimes_{BV} \Omega(T))
\end{equation*}
We extend this definition by colimits to all of $\mathbf{dSets}$. By general arguments the functor $- \otimes X$ has a right adjoint $\mathbf{Hom}_\mathbf{dSets}(X, -)$, making $\mathbf{dSets}$ into a closed symmetric monoidal category. \par 
The category of dendroidal sets is closely related to the category $\mathbf{sOper}$ of simplicial operads. The Boardman-Vogt $W$-construction with respect to the interval $\Delta^1 \in \mathbf{sSets}$ (see \cite{bergermoerdijk2} and \cite{dendroidalsets} for a detailed description) yields a functor
\begin{equation*}
\mathbf{\Omega} \longrightarrow \mathbf{sOper}: T \longmapsto W(\Omega(T))
\end{equation*}
By left Kan extension this gives an adjunction
\[
\xymatrix@R=40pt@C=40pt{
\mathrm{hc}\tau_d: \mathbf{dSets} \ar@<.5ex>[r] & \mathbf{sOper}: \mathrm{hc}N_d\ar@<.5ex>[l]
}
\]
The right adjoint $\mathrm{hc}N_d$ is called the \emph{homotopy coherent dendroidal nerve}. Let us describe the simplicial operad $W(\Omega(T))$ explicitly. Given colours $c_1, \ldots, c_n, c$ of $T$, we describe the space of operations $W(\Omega(T))(c_1, \ldots, c_n; c)$. Suppose there exists a subtree $S$ of $T$ such that the leaves of $S$ are $c_1, \ldots, c_n$ and its root is $c$. If it exists, such an $S$ is unique. Define
\begin{equation*}
W(\Omega(T))(c_1, \ldots, c_n; c) := (\Delta^1)^{I(S)}
\end{equation*}
where $I(S)$ denotes the set of inner edges of the tree $S$. If this set is empty the right-hand side is understood to be the point $\Delta^0$. If there exists no $S$ matching the description above, we let this space of operations be empty. Composition is defined by grafting trees, assigning length 1 to newly arising inner edges (i.e. the edges along which the grafting occurs).  
\par 
We will without explicit mention use basic facts from the theory of model categories, which can for example be found in \cite{hirschhorn} and \cite{hovey}. We will say a class of morphisms is \emph{weakly saturated} if it is closed under retracts, pushouts and transfinite compositions. \par 
A monomorphism $f: X \longrightarrow Y$ of dendroidal sets is said to be \emph{normal} if, for any tree $T \in \mathbf{\Omega}$ and any $\alpha \in Y_T$ which is not in the image of $f$, the stabilizer $\mathrm{Aut}(T)_{\alpha}$ is trivial. A dendroidal set $X$ is called $\emph{normal}$ if the unique map $\emptyset \longrightarrow X$ is normal. The normal monomorphisms are the weakly saturated class generated by the boundary inclusions of trees $\partial\Omega[T] \subseteq \Omega[T]$. By Quillen's small object argument any map of dendroidal sets may be factored as a normal monomorphism followed by a morphism having the right lifting property with respect to all normal monomorphisms, which we refer to as a \emph{trivial fibration}. In particular, factoring a map $\emptyset \longrightarrow X$ in this way, we obtain a normal dendroidal set $X_{(n)}$ which we call a \emph{normalization} of $X$. An easy to prove and very useful fact is that any dendroidal set admitting a map to a normal dendroidal set is itself normal. \par 
A map $f: X \longrightarrow Y$ of dendroidal sets is called an $\emph{operadic equivalence}$ if there exists normalizations $X_{(n)}$ and $Y_{(n)}$ and a map $f_{(n)}: X_{(n)} \longrightarrow Y_{(n)}$ making the obvious diagram commute such that for any $\infty$-operad $Z$ the induced map
\begin{equation*}
i^* \mathbf{Hom_{dSets}}(Y_{(n)}, Z) \longrightarrow i^* \mathbf{Hom_{dSets}}(X_{(n)}, Z)
\end{equation*}
is a categorical equivalence of simplicial sets, i.e. an equivalence in the Joyal model structure. The following was established by Cisinski and Moerdijk in \cite{moerdijkcisinski}:
\begin{theorem}
There exists a combinatorial model structure on the category of dendroidal sets in which the cofibrations are the normal monomorphisms and the weak equivalences are the operadic equivalences. The fibrant objects of this model structure are precisely the $\infty$-operads. By slicing over $\eta$ we obtain a model structure on $\mathbf{sSets}$ which coincides with the Joyal model structure.
\end{theorem}

We will refer to this model structure as the \emph{Cisinski-Moerdijk model structure}. \par 
We close this section with a short remark on terminology. The use of the term \emph{cofibered category}, as in the introduction, clashes badly with the use of the terms cofibration and cofibrant, which will occur all over this text. We will therefore stick with the slightly awkward alternative \emph{opfibered category} and also \emph{opfibered operad}. This problem will disappear after the first chapter, when we switch to the term \emph{coCartesian fibration}.

\newpage

\section{A Grothendieck construction for operads}

Operads can be regarded as a generalization of categories in which morphisms are allowed to have multiple inputs. For this reason, they are sometimes referred to as \emph{symmetric multicategories}. It should not be a surprise that many concepts from category theory have a corresponding generalization to operads. In this chapter we take the Grothendieck construction, which establishes an equivalence between $\mathbf{Cat}$-valued pseudofunctors on a fixed category $C$ and opfibered categories over $C$, and establish a version of it applying to (weak) algebras over operads. Then we investigate its left adjoint and show that by a suitable restriction of the codomain the Grothendieck construction becomes an equivalence. The left adjoint will in later chapters be generalized to the setting of $\infty$-operads to establish a suitable Grothendieck construction and associated equivalence there. The material in this first chapter will not come as a surprise and is very similar to the categorical case. For this reason (and for the sake of brevity) details of proofs are often not provided. Starting from chapter 2 we will rigorously develop the generalization of all these concepts in the setting of $\infty$-operads. \par

\subsection{The Grothendieck construction}

Let $\underline{\mathbf{Cat}}$ be the operad induced by the category $\mathbf{Cat}$ under the Cartesian product of categories, i.e. $\underline{\mathbf{Cat}}$ has as sets of operations
\begin{equation*}
\underline{\mathbf{Cat}}(C_1, \ldots, C_n; C) = \mathbf{Fun}(C_1 \times \cdots \times C_n, C)
\end{equation*}
If $S$ is an operad in $\mathbf{Sets}$, we will refer to a morphism of operads $S \longrightarrow \underline{\mathbf{Cat}}$ as an \emph{$S$-algebra in $\mathbf{Cat}$}. \par
Given an algebra $\mathcal{F}: S \longrightarrow \underline{\mathbf{Cat}}$ we would like to invoke a generalization of the Grothendieck construction in ordinary category theory to obtain an operad $\int_S\mathcal{F}$ over $S$ associated to $\mathcal{F}$. This can be done as follows. Define the colours of $\int_S\mathcal{F}$ by
\begin{equation*}
\mathrm{col}\Bigl(\int_S\mathcal{F}\Bigr) = \coprod_{s \in \mathrm{col}(S)} \mathrm{ob}(\mathcal{F}(s))
\end{equation*}
There is an obvious projection $\pi_{\mathcal{F}}: \mathrm{col}(\int_S\mathcal{F}) \longrightarrow \mathrm{col}(S)$. Given colours $b_1, \ldots, b_n, b$ of $\int_S\mathcal{F}$, an operation $(b_1, \ldots, b_n) \longrightarrow b$ is a pair $(\sigma, f)$, where
\begin{equation*}
\sigma: (p(b_1), \ldots, p(b_n)) \longrightarrow p(b)
\end{equation*}
is an operation of $S$ and
\begin{equation*}
f: \mathcal{F}(\sigma)(b_1, \ldots, b_n) \longrightarrow b
\end{equation*}
is a morphism in the category $\mathcal{F}(p(b))$. One easily verifies that the obvious composition law makes $\int_S\mathcal{F}$ into an operad. Clearly $\pi_{\mathcal{F}}$ extends to a morphism of operads $\pi_{\mathcal{F}}: \int_S\mathcal{F} \longrightarrow S$. \par
As in ordinary category theory, it will turn out that we also need to consider `weak' algebras $\mathcal{F}: S \longrightarrow \underline{\mathbf{Cat}}$. To do this, we regard $\underline{\mathbf{Cat}}$ as a \emph{bioperad}. The definition of a bioperad is the direct generalization of the definition of a bicategory to the setting of operads. Phrased differently, we can think of bioperads as `weak operads in groupoids'. One can now also regard $S$ as a bioperad with only trivial 2-cells. A \emph{weak $S$-algebra with values in $\mathbf{Cat}$} is a morphism of bioperads $\mathcal{F}: S \longrightarrow \underline{\mathbf{Cat}}$. Note that in the special case where $S$ is a category, such a weak algebra is usually referred to as a pseudofunctor. \par
Now suppose given such a weak algebra $\mathcal{F}$. We construct $\int_S\mathcal{F}$ in exactly the same way as above. The only thing that requires extra work is the verification that $\int_S\mathcal{F}$ is really an operad. This is a bit tedious, but entirely analogous to the similar verification in ordinary category theory. We refer the reader to \cite{vistoli} for details of the categorical case. \par
The rest of this section will be devoted to a construction which will be of use in everything that follows. Suppose we are given a colour $s$ of our operad $S$. We would like to construct the `algebra $\underline{s}(-)$ corepresented by $s$' which should generalize the corresponding notion from ordinary category theory. This algebra can then be thought of as the `free $S$-algebra at $s$ on one generator'. We define it by
\begin{equation*}
\underline{s}(-): S \longrightarrow \mathbf{Sets}: c \mapsto \coprod_{n \in \mathbb{Z}_{\geq 0}} S(\underbrace{s, \ldots, s}_{\mathrm{n \, times}};c) / \Sigma_n
\end{equation*}
The following is the generalization of the (co)Yoneda lemma to operads in $\mathbf{Sets}$:

\begin{lemma}
\label{lemma:multiyoneda}
Given an algebra $\mathcal{F}: S \longrightarrow \underline{\mathbf{Sets}}$ and a colour $s$ of $S$ there is a bijection
\begin{equation*}
\mathbf{Nat}(\underline{s}(-), \mathcal{F}) \simeq \mathcal{F}(s)
\end{equation*}
given by evaluation at $\mathrm{id}_s$.
\end{lemma}

Defining
\begin{equation*}
\underline{s}(-) \times_{\Sigma} X := \coprod_{n \in \mathbb{Z}_{\geq 0}} S(\underbrace{s, \ldots, s}_{\mathrm{n \, times}};-) \times_{\Sigma_n} X^n
\end{equation*}
this is the special case $X = \ast$ of the following result:

\begin{lemma}
\label{lemma:multiyoneda2}
Given an algebra $\mathcal{F}: S \longrightarrow \underline{\mathbf{Sets}}$, a colour $s$ of $S$ and a set $X$ there is a bijection
\begin{equation*}
\mathbf{Nat}(\underline{s}(-) \times_{\Sigma} X, \mathcal{F}) \simeq \mathbf{Sets}(X, \mathcal{F}(s))
\end{equation*}
given by evaluation at $\mathrm{id}_s$.
\end{lemma}

\subsection{The left adjoint to the Grothendieck construction}

\begin{definition}
Denote by $[S, \underline{\mathbf{Cat}}]_{weak}$ the 2-category which has as objects the weak $S$-algebras in $\mathbf{Cat}$, as 1-morphisms the (weak) natural transformations and as 2-morphisms the modifications of those.
\end{definition}

\begin{definition}
For a fixed operad $S$ we define a 2-category $\mathbf{Oper_{Sets}}/S$ with objects operad morphisms $X \longrightarrow S$, morphisms commutative triangles
\[
\xymatrix{
X \ar[rr]\ar[dr] & & Y \ar[dl] \\
& S &
}
\]
and 2-morphisms natural transformations of operad morphisms projecting to the trivial natural transformation of $\mathrm{id}_S$ to itself.
\end{definition}

It is straightforward to verify that the Grothendieck construction introduced in the previous section will yield a 2-functor
\begin{equation*}
\int_S: [S, \underline{\mathbf{Cat}}]_{weak} \longrightarrow \mathbf{Oper_{Sets}}/S
\end{equation*}
This functor turns out to have a left adjoint, which we will describe in this section. It is this left adjoint that will later be generalized to the setting of $\infty$-operads. \par
We will first generalize our setup of `corepresentable algebras' from the first section of this chapter. Suppose we are given an operad morphism $p: X \longrightarrow S$. Define the set
\begin{equation*}
C(X) := \{(x_1, \ldots, x_n) \, | \, n \geq 0, \, x_i \in \mathrm{col}(X)\}
\end{equation*}
Now suppose $s$ is a colour of $S$. We define a category $p/s$ which has as objects
\begin{equation*}
\mathrm{ob}(p/s) := \coprod_{n \geq 0}\Bigl(\coprod_{(x_1, \ldots, x_n) \in C(X)} S(p(x_1), \ldots, p(x_n) ; s)\Bigr)/\Sigma_n
\end{equation*}
Given two such objects $[(p(x_1), \ldots, p(x_n)) \rightarrow s]$ and $[(p(y_1), \ldots, p(y_m)) \rightarrow s]$ a morphism between them is an equivalence class represented by a tuple of $m$ operations $(\xi_1 ,\ldots , \xi_m)$ of $X$ such that $\xi_i$ has output $y_i$ and the sets of inputs of the different $\xi_i$'s are disjoint and their union equals $\{x_1, \ldots, x_n\}$. Furthermore, after applying $p$, these $\xi_i$'s should make the obvious diagram commute. The so-called \emph{straightening functor} is then described by
\begin{equation*}
St_S(p): S \longrightarrow \underline{\mathbf{Cat}}: s \mapsto p/s
\end{equation*}
This assignment is seen to extend to a 2-functor
\begin{equation*}
St_S: \mathbf{Oper_{Sets}}/S \longrightarrow [S, \underline{\mathbf{Cat}}]_{weak}
\end{equation*}
The definition of this functor might seem somewhat elaborate at first, but (described sketchily) it is nothing but the $S$-algebra constructed as follows. We demand that the category $St_S(p)(s)$ contain all colours of $X$ lying over $s$ and we want this category to have a morphism for each operation of $X$ with codomain lying over $s$. Taking the freely generated $S$-algebra with these `generators' and imposing the obvious relations due to $\Sigma$-equivariance we arrive at the straightening functor described above. 

\begin{remark}
Note that if the morphism $p: X \longrightarrow S$ is simply the inclusion of the trivial operad on a colour of $S$, this construction reproduces the corepresentable algebra of the beginning of this chapter. In this light, the following result can also be seen as a jazzed up version of the Yoneda lemma. It should not come as a surprise that the strategy of proof is similar.
\end{remark}

\begin{proposition}
There is an adjunction
\[
\xymatrix@R=40pt@C=40pt{
St_S: \mathbf{Oper_{Sets}}/S \ar@<.5ex>[r] & [S, \underline{\mathbf{Cat}}]_{weak}: \int_S\ar@<.5ex>[l]
}
\]
\end{proposition}
\begin{proof}
We have to exhibit a natural equivalence of categories
\begin{equation*}
\gamma: [S, \underline{\mathbf{Cat}}]_{weak}(St(p), \mathcal{F}) \longrightarrow \mathbf{Oper_{Sets}}/S\bigl(p, \int_S \mathcal{F}\bigr)
\end{equation*}
In order to avoid too much tedious verification, we will give a sketch of how this is done and leave the details to the reader. In fact, the proof is already apparent from our `generators and relations' description of the straightening functor above. A morphism of operads in $\mathbf{Oper_{Sets}}/S\bigl(p, \int_S \mathcal{F}\bigr)$ is uniquely determined by the following data:
\begin{itemize}
\item[(1)] The induced map on colours, i.e. for each colour $x$ of $X$ lying over some $s$ in $S$ an assigment of a colour in $\mathcal{F}(s)$
\item[(2)] For each operation $\xi$ of $X$ lying over some $\sigma$ of $S$ the assignment of an operation of $\int_S \mathcal{F}$ lying over $\sigma$ (having in- and outputs determined by (1)). By our definition of the latter operad, this means we are assigning to $\xi$ a morphism in $\mathcal{F}(p(x))$, where $x$ denotes the output of $\xi$.
\end{itemize}
We defined $St(p)$ as the freely generated $S$-algebra determined by precisely these data, so the proof will now follow from the usual Yoneda-style argument. $\Box$
\end{proof}

\subsection{Opfibered operads and the essential image of the Grothendieck construction}

In this section we will define opfibered operads and show how to obtain weak algebras from them. We then show how this gives a quasi-inverse to the Grothendieck construction. \par
Let $X$ and $S$ be operads in $\mathbf{Sets}$ and let $p: X \longrightarrow S$ be a morphism between them. We will first study the notion of a \emph{$p$-coCartesian operation of $X$}. The definition of such an operation can be given without using any terminology borrowed from forestry, but it is slightly more convenient (and useful for generalization in later sections) to formulate it in terms of dendroidal sets, i.e. in the language of trees. Let $T$ be the tree drawn below, where the vertex $v$ has $n$ inputs and the vertex $w$ has $m$ inputs:
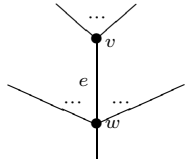
\[
\xymatrix@R=10pt@C=12pt{
&&&&&& \\
&&&*=0{\bullet}\ar@{-}[ul]\ar@{-}[ur]\ar@{}[u]|{\cdots}_(-.1)v&&&\\
&&&&&&\\
&&&*=0{\bullet}\ar@{-}\ar@{-}[ull]\ar@{-}[uu]^(.5)e\ar@{-}[urr]\ar@{}[ul]|{\cdots}\ar@{}[ur]|{\cdots}&&&\\
&&&\ar@{-}[u]_(1)w&&&\\
&&&&&&
}
\]
Recall that the horn $\Lambda^v \Omega[T]$ is the union of the corolla with vertex $v$ with the corolla obtained by contracting along the edge $e$.

\begin{definition}
Let $T$ be as above, where $m \geq 1$ is arbitrary. An operation $\xi: (x_1, \ldots, x_n) \longrightarrow x$ of $X$ is called \emph{$p$-coCartesian} if any diagram of the form
\[
\xymatrix{
C_n \ar[dr]^{\xi}\ar[d]_v & \\
\Lambda^v\Omega[T] \ar[d]\ar[r] & N_d(X) \ar[d]^{N_d(p)} \\
\Omega[T] \ar[r]\ar@{-->}[ur] & N_d(S)
}
\]
has a unique dotted lift as indicated. Here $v$ denotes (slightly abusively) the map sending the vertex of $C_n$ to the vertex $v$ of $T$.
\end{definition}

\begin{remark}
In the special case where $X$ and $S$ are categories (i.e. both have only unary operations), our notion of $p$-coCartesian operation coincides with the usual notion of $p$-coCartesian morphism.
\end{remark}

We can in fact also characterize coCartesian operations in a different fashion:

\begin{proposition}
Let $p: X \longrightarrow S$ be a morphism of operads. An operation $\xi: (x_1, \ldots, x_n) \longrightarrow x$ of $X$ is $p$-coCartesian if and only if, for all tuples of colours $(y_1, \ldots, y_k)$ of $X$ and every colour $z$ of $X$, the following diagram (for arbitrary $0 \leq i \leq k$) is a pullback square:
\[
\xymatrix{
X(\mathbf{y}; z) \ar[r]\ar[d] & X(\mathbf{y_x}; z)\ar[d] \\
S(\mathbf{p(y)}; p(z)) \ar[r] & S(\mathbf{p(y_x)}; p(z))
}
\]
Here we have used the following abbreviations:
\begin{eqnarray*}
\mathbf{y} & := & (y_1, \ldots, y_{i-1}, x, y_i, \ldots, y_k) \\
\mathbf{y_x} & := & (y_1, \ldots, y_{i-1}, x_1, \ldots, x_n, y_i, \ldots, y_k) \\
\mathbf{p(y)} & := & (p(y_1), \ldots, p(y_{i-1}), p(x), p(y_i), \ldots, p(y_k)) \\
\mathbf{p(y_x)} & := & (p(y_1), \ldots, p(y_{i-1}), p(x_1), \ldots, p(x_n), p(y_i), \ldots, p(y_k))
\end{eqnarray*}
The top arrow in the square is induced by precomposing with $\xi$.  
\end{proposition}

We collect the following basic facts about coCartesian operations.

\begin{proposition}\label{prop:coCartesian} Let $p: X \longrightarrow S$ be as before.
\begin{itemize}
\item[(1)] A $p$-coCartesian operation $x \longrightarrow y$ in $X$ is an isomorphism if and only if its image in $S$ is an isomorphism
\item[(2)] Given $p$-coCartesian operations $(x^i_1,\ldots,x^i_{k_i}) \longrightarrow x^i$, where $1 \leq i \leq n$ for some $n \in \mathbb{N}$, and a $p$-coCartesian operation $(x^1, \ldots, x^n) \longrightarrow y$, the composite
    \begin{equation*}
    (x^1_1,\ldots,x^1_{k_1},x^2_1,\ldots,x^2_{k_2}, \ldots, x^n_{k_n}) \longrightarrow y
    \end{equation*}
    is $p$-coCartesian. In the language of trees; grafting coCartesian corollas onto the leaves of a coCartesian corolla again produces a coCartesian corolla.
\item[(3)] If an $n$-ary operation $\xi$ is $p$-coCartesian, then so is $\sigma^*(\xi)$ for any $\sigma \in \Sigma_n$ 
\end{itemize}
\end{proposition}

\begin{definition}
Let $p: X \longrightarrow S$ be a map of operads. We will say \emph{$X$ is opfibered over $S$ (by $p$)} if for any operation $\sigma: (s_1, \ldots, s_n) \longrightarrow s$ of $S$ and colours $x_1, \ldots, x_n$ of $X$ such that $p(x_i) = s_i$, there exists a coCartesian operation $\xi: (x_1, \ldots, x_n) \longrightarrow x$ of $X$ which projects to $\sigma$. The phrase `by $p$' is usually omitted, which should not cause confusion.
\end{definition}

\begin{remark}
\label{remark:uniquelift}
Suppose $X$ is opfibered over $S$. Using Proposition \ref{prop:coCartesian} one sees that coCartesian lifts are unique up to unique vertical isomorphism once the domain of the lift is specified.
\end{remark}

\begin{definition}
If $p: X \longrightarrow S$ and $q: Y \longrightarrow S$ are operads opfibered over $S$, a \emph{morphism $p \longrightarrow q$ of operads opfibered over S} is a morphism of operads $X \longrightarrow Y$ compatible with the projections to $S$ which sends $p$-coCartesian operations of $X$ to $q$-coCartesian operations of $Y$. The (strict) 2-category with objects opfibered operads over $S$, morphisms as just specified and 2-morphisms natural transformations of operad morphisms projecting to the trivial natural transformation of $\mathrm{id}_S$ to itself is denoted $\mathbf{opFib}(S)$. There is an obvious inclusion
\begin{equation*}
\mathbf{opFib}(S) \longrightarrow \mathbf{Oper_{Sets}}/S
\end{equation*}
\end{definition}

\begin{remark}
The Grothendieck construction actually factors through $\mathbf{opFib}(S)$ and hence yields a 2-functor
\begin{equation*}
\int_S : [S, \underline{\mathbf{Cat}}]_{weak} \longrightarrow \mathbf{opFib}(S)
\end{equation*}
Indeed, let $\mathcal{F} \in [S, \underline{\mathbf{Cat}}]_{weak}$ and let $\sigma$ be an operation of $S$. Then the operations $(\sigma, \mathrm{id}_{\sigma(x_1, \ldots, x_n)}): (x_1, \ldots, x_n) \longrightarrow \sigma(x_1, \ldots, x_n)$ will be coCartesian lifts of $\sigma$ to $\int_S \mathcal{F}$.
\end{remark}

We will need the notion of a cleavage.

\begin{definition}
A \emph{cleavage} of an opfibered operad $p: X \longrightarrow S$ is a class $K$ of $p$-coCartesian morphisms of $X$ such that for any $\sigma: (s_1, \ldots, s_n) \longrightarrow s$ of $S$ and any tuple of colours $(x_1, \ldots, x_n)$ of $X$ mapping to $(s_1, \ldots, s_n)$, there is a unique morphism $(x_1, \ldots, x_n) \longrightarrow \sigma_!(x_1, \ldots, x_n)$ in $K$ which projects to $\sigma$. We also demand that $K$ be closed under the actions of the symmetric groups. In other words, a cleavage is simply a choice of $p$-coCartesian lift for each morphism of $S$ and specified lift of the domain.
\end{definition}

\begin{convention}
From now on we will tacitly assume that a choice of cleavage has been made for every opfibered operad.
\end{convention}

\begin{remark}
Note that a cleavage of an opfibered operad $p: X \longrightarrow S$ induces a unique factorization of any operation of $X$ into a $p$-coCartesian operation followed by a \emph{vertical morphism}, i.e. a morphism projecting to an identity morphism in $S$.
\end{remark}

Now we proceed to defining the pseudo-inverse to the Grothendieck construction
\begin{equation*}
\Phi: \mathbf{opFib}(S) \longrightarrow [S, \underline{\mathbf{Cat}}]_{weak}
\end{equation*}
alluded to in the beginning of this section. Suppose we are given an opfibered operad $p: X \longrightarrow S$. We begin by defining the weak algebra $\Phi(p)$ on colours. For $s \in S$, let $\eta_s \subseteq S$ denote the trivial operad on the colour $s$ and set
\begin{equation*}
\Phi(p)(s) := \eta_s \times_{S} X
\end{equation*}
Strictly speaking this is a bad definition, since the fiber on the right-hand side is only defined up to isomorphism. Hence for definiteness we identify this fiber with the category which has as objects all colours of $X$ projecting to $s$ and as morphisms the unary operations of $X$ projecting to the identity of $s$. We will denote this fiber by $X_s$ when no confusion can arise. \par 
We have fixed a cleavage of our opfibered operad. Suppose we have an operation $\sigma: (s_1, \ldots, s_n) \longrightarrow s$ of $S$. We define a functor
\begin{equation*}
\sigma_!: \prod_{i=1}^n \Phi(p)(s_i) \longrightarrow \Phi(p)(s)
\end{equation*}
which sends each object $(x_1, \ldots, x_n)$ to the object $\sigma_!(x_1, \ldots, x_n)$ defined by our cleavage, and each morphism $(f_1, \ldots, f_n)$ to the unique morphism $\sigma_!(f_1, \ldots, f_n)$ rendering the following diagram commutative:

\[
\xymatrix{
(x_1, \ldots, x_n) \ar[d]_{(f_1, \ldots, f_n)}\ar[r] & \sigma_!(x_1, \ldots, x_n) \ar@{-->}[d]^{\sigma_!(f_1, \ldots, f_n)} \\
(y_1, \ldots, y_n) \ar[r] & \sigma_!(x_1, \ldots, x_n)
}
\]

This arrow is indeed unique since the top horizontal arrow is $p$-coCartesian. We set
\begin{equation*}
\Phi(p)(\sigma) = \sigma_!
\end{equation*}
Similar to the categorical case this will not define a strict algebra $\Phi(p): S \longrightarrow \underline{\mathbf{Cat}}$. Indeed, suppose we are given morphisms $\sigma^i: (s^i_1,\ldots,s^i_{k_i}) \longrightarrow s^i$ in $S$, where $1 \leq i \leq n$ for some $n \in \mathbb{N}$, and a morphism $\tau: (s^1, \ldots, s^n) \longrightarrow s$, also in $S$. Then the functor $\tau_! \circ (\sigma^1_!, \ldots, \sigma^n_!)$ is not necessarily equal to $(\tau \circ (\sigma^1, \ldots, \sigma^n))_!$. However, using Proposition \ref{prop:coCartesian} we see that both functors are built using coCartesian morphisms lying over the same morphisms of $S$. Since coCartesian lifts with specified domain are unique up to unique isomorphism, there is a canonical natural isomorphism between these functors. Using these natural isomorphisms, $\Phi(p)$ acquires the structure of a weak algebra. \par 
It is now straightforward to extend the assignment $p \longrightarrow \Phi(p)$ to a 2-functor
\begin{equation*}
\Phi: \mathbf{opFib}(S) \longrightarrow [S, \underline{\mathbf{Cat}}]_{weak}
\end{equation*}

\begin{remark}
It may seem that our definition of $\Phi$ depends on our choice of cleavages. However, using the fact that coCartesian lifts (with specified domain) are unique up to unique vertical isomorphism, one sees that different choices will yield naturally isomorphic weak algebras and moreover these natural isomorphisms are uniquely determined. Hence the choice of cleavage is insubstantial.
\end{remark}

The main result of this section is:

\begin{theorem}
\label{thm:grothconstoperads}
There is an adjoint equivalence of bicategories
\[
\xymatrix@R=40pt@C=40pt{
\int_S : [S, \underline{\mathbf{Cat}}]_{weak}\ar@<.5ex>[r] & \mathbf{opFib}(S): \Phi\ar@<.5ex>[l]
}
\]
\end{theorem}
\begin{proof}
We will start by exhibiting the adjunction. Suppose we are given an opfibered operad $p: X \longrightarrow S$ and a weak multifunctor $\mathcal{F}: S \longrightarrow \underline{\mathbf{Cat}}$. We want to construct a (weakly natural) equivalence of categories
\begin{equation*}
\gamma: \mathbf{opFib}(S)\Bigl(\int_S \mathcal{F}, p\Bigr) \longrightarrow [S, \underline{\mathbf{Cat}}]_{weak}(\mathcal{F}, \Phi(p))
\end{equation*}
Suppose we are given a morphism $\phi: \int_S \mathcal{F} \longrightarrow p$ of opfibered operads. For each colour $s$ of $S$ this induces a morphism of the corresponding fibers, which we denote
\begin{equation*}
\phi_s: \mathcal{F}(s) \longrightarrow X_s
\end{equation*}
We would like these maps to constitute a weak natural transformation $\gamma(\phi)$. If we are given an operation $\alpha: (s_1, \ldots, s_n) \longrightarrow s$ of $S$, we have to investigate if the following square is commutative up to an invertible 2-cell:
\[
\xymatrix@C=50pt{
\mathcal{F}(s_1) \times \ldots \times \mathcal{F}(s_n) \ar[r]^{\phi_{s_1}\times\ldots\times\phi_{s_n}}\ar[d]_{\mathcal{F}(\alpha)} & X_{s_1} \times \ldots \times X_{s_n} \ar[d]^{\alpha_!} \\
\mathcal{F}(s) \ar[r]_{\phi_s} & X_s
}
\]
Consider an object $(x_1, \ldots, x_n)$ in the category in the top left corner. Then the object $\alpha(x_1, \ldots, x_n) \in \mathcal{F}(s)$ is the codomain of the coCartesian operation $(\alpha, \mathrm{id}_{(x_1, \ldots, x_n)})$ in $\int_S \mathcal{F}$. The operation $\phi\bigl((\alpha, \mathrm{id}_{(x_1, \ldots, x_n)})\bigr)$ will be a $p$-coCartesian operation of $X$, whose codomain is $\phi(\alpha(x_1, \ldots, x_n))$. Since coCartesian operations are unique up to unique vertical isomorphism, we get a uniquely determined isomorphism $\phi(\alpha(x_1, \ldots, x_n)) \longrightarrow \alpha_!(\phi_{s_1}(x_1),\ldots,\phi_{s_n}(x_n))$. These isomorphisms constitute the required 2-cell. \par
Now that we have defined $\gamma$ on objects, it is straightforward to define it on morphisms. Details are left to the reader. We will show $\gamma$ is an equivalence of categories by exhibiting an explicit pseudo-inverse
\begin{equation*}
\delta: [S, \underline{\mathbf{Cat}}]_{weak}(\mathcal{F}, \Phi(p)) \longrightarrow \mathbf{opFib}(S)\Bigl(\int_S \mathcal{F}, p\Bigr)
\end{equation*}
Suppose we are given a weak natural transformation $\theta: \mathcal{F} \longrightarrow \Phi(p)$. If $x$ is a colour of $\int_S \mathcal{F}$ we define
\begin{equation*}
\delta(\theta)(x) := \theta_{\pi_{\mathcal{F}}(x)}(x)
\end{equation*}
Suppose $(\sigma, \mathrm{id}_x): (x_1, \ldots, x_n) \longrightarrow x$ is an operation of $\int_S\mathcal{F}$, so that $\sigma: (s_1, \ldots, s_n) \longrightarrow s$ is an operation of $S$ (i.e. we have $\pi_\mathcal{F}(x_i) = s_i$ and similarly for $x$). The chosen cleavage of $p$ determines a $p$-coCartesian lift $\tilde \sigma$ of $\sigma$ to $X$ with domain $(\theta(x_1), \ldots, \theta(x_n))$. We have a 2-cell of $\theta$ as indicated in the following diagram:
\[
\xymatrix@C=50pt{
\mathcal{F}(s_1) \times \ldots \times \mathcal{F}(s_n) \ar[r]^{\theta_{s_1}\times\ldots\times\theta_{s_n}}\ar[d]_{\mathcal{F}(\sigma)} & X_{s_1} \times \ldots \times X_{s_n} \ar[d]^{\Phi(p)(\sigma)} \\
\mathcal{F}(s) \ar[r]_{\theta_s} & X_s
}
\]
This 2-cell provides us with an isomorphism $h_\sigma: \Phi(p)(\sigma)(\theta(x_1), \ldots, \theta(x_n)) \longrightarrow \theta(x)$ in the fiber $X_s$. We define
\begin{equation*}
\delta(\theta)\bigl((\sigma, \mathrm{id}_x)\bigr) := h_\sigma \circ \tilde\sigma
\end{equation*}
For a general morphism $(\sigma, f)$ of $\int_S\mathcal{F}$ we set
\begin{equation*}
\delta(\theta)\bigl((\sigma, f)\bigr) := \theta_x(f) \circ \delta(\theta)\bigl((\sigma, \mathrm{id}_x)\bigr)
\end{equation*}
One verifies explicitly that $\delta$ is pseudo-inverse to $\gamma$. \par
We now wish to show that our adjunction is indeed an adjoint equivalence. It is straightforward to check that the unit
\begin{equation*}
\epsilon: \mathrm{id}_{[S, \underline{\mathbf{Cat}}]_{weak}} \longrightarrow \Phi \circ \int_S
\end{equation*}
is a natural isomorphism of 2-functors. The counit
\begin{equation*}
\eta: \int_S \circ \, \Phi \longrightarrow \mathrm{id}_{\mathbf{opFib}(S)}
\end{equation*}
can be described as follows. On colours it is the identity and operations of the form $(\sigma, f)$ are mapped to $f \circ \tilde\sigma$, where the notation $\tilde \sigma$ has the same meaning as above. Given a cleavage of an opfibered operad, the factorization of an operation into a coCartesian operation followed by a vertical morphism exists and is unique. Hence the counit is full and faithful and we conclude that it is also a natural isomorphism of 2-functors. This completes the proof. $\Box$
\end{proof}

We end this section with a brief discussion of operads opfibered in groupoids.

\begin{definition}
Let $p: X \longrightarrow S$ be a map of operads. We say \emph{$X$ is opfibered in groupoids over $S$ by $p$} if $X$ is opfibered over $S$ by $p$ and every operation of $X$ is $p$-coCartesian. We denote by $\mathbf{opFib}_{\mathbf{Gpd}}(S)$ the full subcategory of $\mathbf{opFib}(S)$ on the operads opfibered in groupoids over $S$.
\end{definition}

Let us define the 2-category of weak $S$-algebras taking values in groupoids.

\begin{definition}
Denote by $[S, \mathbf{\underline{Gpd}}]_{weak}$ the full subcategory of $[S, \mathbf{\underline{Cat}}]_{weak}$ on the weak $S$-algebras $A$ such that $A(s)$ is a groupoid for every colour $s$ of $S$.
\end{definition}

One immediately verifies that for any $A \in [S, \mathbf{\underline{Gpd}}]_{weak}$ the Grothendieck construction $\int_S A$ is opfibered in groupoids over $S$. In fact, we have the following result:

\begin{theorem}
\label{thm:grothconstopergroupoids}
The restrictions of $\int_S$ and $\Phi$ give an adjoint equivalence of bicategories
\[
\xymatrix@R=40pt@C=40pt{
\int_S : [S, \underline{\mathbf{Gpd}}]_{weak}\ar@<.5ex>[r] & \mathbf{opFib}_{\mathbf{Gpd}}(S): \Phi\ar@<.5ex>[l]
}
\]
\end{theorem}

\newpage

\section{Left fibrations and topological algebras}
\label{section:leftfib}
In this section we will generalize Theorem \ref{thm:grothconstopergroupoids} to $\infty$-operads using dendroidal sets. We will define the notion of a \emph{left fibration of dendroidal sets}, which corresponds to an operad opfibered in groupoids. For a dendroidal set $S$ we will define an $\infty$-category $\mathbf{LFib}(S)$ of left fibrations over $S$. Using heuristics derived from the results of the previous section, we might expect that this $\infty$-category is in fact the $\infty$-category of $S$-algebras valued in $\infty$-groupoids or, equivalently, spaces. This expectation turns out to be valid; we will make these statements precise and formulate the main result of this section in Theorem \ref{thm:straighteningleftfib}. \par
Most results in this section will be stated without proofs in order to make the exposition more accessible. In the next section we will be dealing with the greater generality of algebras over $\infty$-operads valued in $\infty$-categories. Once we prove the relevant results there, it will not be hard to derive the results of this section from them. That is what we will do in section \ref{subsection:leftfibrations}. \par 

\subsection{Left fibrations and the covariant model structure}
\begin{definition}
Let $p: X \longrightarrow S$ be a map of dendroidal sets. Then $p$ is a \emph{left fibration} if the following conditions are satisfied:
\begin{itemize}
\item[(i)] $p$ is an inner fibration
\item[(ii)] For any corolla $\sigma$ of $S$ having inputs $\{s_1, \ldots, s_n\}$ (note that the set of inputs could be empty) and colors $\{x_1, \ldots, x_n\}$ of $X$ satisfying $p(x_i) = s_i$ for $1 \leq i \leq n$, there exists a corolla $\xi$ of $X$ with inputs $\{x_1, \ldots, x_n \}$ such that $p(\xi) = \sigma$
\item[(iii)] For any tree $T$ with at least two vertices and any leaf vertex $v$ of $T$, there exists a lift in any diagram of the form
\[
\xymatrix{
\Lambda^v[T] \ar[r]\ar[d] & X \ar[d]^p \\
\Omega[T] \ar[r]\ar@{-->}[ur] & S
}
\]
\end{itemize}
\end{definition}

\begin{remark}
\label{rmk:simplicialleftfib}
If $p: X \longrightarrow S$ is a left fibration, then the induced map $i^*p: i^*X \longrightarrow i^*S$ is a left fibration of simplicial sets, i.e. it has the right lifting property with respect to horn inclusions $\Lambda^n_i \longrightarrow \Delta^n$ for $0 \leq i < n$. In particular, if $i^*S$ is a Kan complex, then $i^*X$ is a Kan complex as well by a fundamental result of Joyal (see Proposition 1.2.5.1 of \cite{htt}). Also, if $q: K \longrightarrow L$ is a left fibration of simplicial sets, then $i_!(q)$ is a left fibration of dendroidal sets. 
\end{remark}

The following lemma tells us that our notion of left fibration indeed generalizes the notion of an operad opfibered in groupoids:

\begin{lemma}
\label{lemma:opfibgrpdsleftfib}
Let $p: X \longrightarrow S$ be a map of operads. Then $p$ exhibits $X$ as an operad opfibered in groupoids over $S$ if and only if $N_d(p): N_d(X) \longrightarrow N_d(S)$ is a left fibration of dendroidal sets.
\end{lemma}
\begin{proof}
Assume that $p: X \longrightarrow S$ makes $X$ into an operad opfibered in groupoids. Then $N_d(p)$ is automatically an inner fibration. Property (ii) is satisfied, because for any operation of $S$ and prescribed lifts of its inputs we can pick a coCartesian operation of $X$ lying over it. Let $T$ be a tree with exactly two vertices, whose leaf vertex we denote by $v$. We need a lift in any diagram as described in (iii). This lift exists by the fact that every operation of $X$ is $p$-coCartesian. If $T$ has 3 vertices, one uses the uniqueness of coCartesian lifts (up to vertical isomorphism) to establish the existence of the necessary lifts. If $T$ has 4 or more vertices, lifts in diagrams of the form given above automatically exist since the nerve of an operad is 2-coskeletal. \par 
Conversely, suppose $N_d(p)$ is a left fibration. Given an operation of $S$ with inputs $s_1, \ldots, s_n$ and colours $x_1, \ldots, x_n$ of $X$ such that $p(x_i) = s_i$, property (ii) tells us that we can find a lift of this operation to $X$ with inputs $x_1, \ldots, x_n$. To check that every operation of $X$ is $p$-coCartesian, we use property (iii) for trees with two vertices to get the desired lifts and use trees with three vertices to assure uniqueness of those lifts. The reader is invited to spell out the details. $\Box$
\end{proof}

Let us make the following elementary observations:

\begin{proposition}
\label{prop:leftfibcomppullback}
A composition of left fibrations is a left fibration. Suppose we are given a pullback square
\[
\xymatrix{
X' \ar[d]\ar[r] & X \ar[d] \\
S' \ar[r] & S
}
\]
in which the right vertical map is a left fibration. Then the left vertical map is a left fibration as well.
\end{proposition}

Our goal in this section is to describe a model structure on the category $\mathbf{dSets}/S$ in which the fibrant objects are precisely the left fibrations with codomain $S$. We will first endow $\mathbf{dSets}/S$ with the structure of a simplicial category.

\begin{definition}
Given maps $X \longrightarrow S$ and $Y \longrightarrow S$, we define the simplicial set $\mathrm{Map}_S(X, Y)$ as follows: 
\begin{equation*}
\mathrm{Map}_S(X, Y)_n := \mathbf{dSets}/S(X \otimes i_!(\Delta^n), Y) 
\end{equation*}
where the map $X \otimes i_!(\Delta^n) \longrightarrow S$ is obtained by composing the projection map $X \otimes i_!(\Delta^n) \longrightarrow X$ with the map $X \longrightarrow S$.
\end{definition}

Note that $\mathrm{Map}_S(X, Y)$ satisfies the following universal property: for any simplicial set $K$ there is an isomorphism
\begin{equation*}
\mathbf{sSets}(K, \mathrm{Map}_S(X, Y)) \simeq \mathbf{dSets}/S(X \otimes i_!(K), Y) 
\end{equation*}
This isomorphism is natural in $K$. \par 
The mapping objects $\mathrm{Map}_S(X,Y)$ make $\mathbf{dSets}/S$ into a simpicial category. Also bear in mind that $\mathbf{dSets}/S$ is tensored and cotensored over $\mathbf{sSets}$, a fact we already used in the definition above.

\begin{definition}
We will call a map in $\mathbf{dSets}/S$ a \emph{covariant cofibration} if its underlying map of dendroidal sets is a cofibration. We will call a map $f: X \longrightarrow Y$ in $\mathbf{dSets}/S$ a \emph{covariant equivalence} if for any left fibration $Z \longrightarrow S$ and normalizations $X_{(n)}$ and $Y_{(n)}$ of $X$ resp. $Y$ the induced map
\begin{equation*}
\mathrm{Map}_S(Y_{(n)}, Z) \longrightarrow \mathrm{Map}_S(X_{(n)}, Z)
\end{equation*}
is a weak homotopy equivalence of simplicial sets.
\end{definition}

\begin{remark}
The condition that there exist normalizations $X_{(n)}$ and $Y_{(n)}$ of $X$ resp. $Y$ such that the induced map
\begin{equation*}
\mathrm{Map}_S(Y_{(n)}, Z) \longrightarrow \mathrm{Map}_S(X_{(n)}, Z)
\end{equation*}
is a weak homotopy equivalence is equivalent to the condition that for any choice of normalizations  $X_{(n)}$ and $Y_{(n)}$ the stated map is a weak homotopy equivalence.
\end{remark}

The proof of the following result will be given in section \ref{subsection:leftfibrations}.

\begin{theorem}
\label{thm:covmodelstruct}
There exists a model structure on $\mathbf{dSets}/S$ in which the cofibrations (resp. weak equivalences) are the covariant cofibrations (resp. covariant weak equivalences). This model structure is combinatorial, left proper and simplicial. In this model structure the fibrant objects are precisely the left fibrations over $S$.
\end{theorem}

We will refer to this model structure as the \emph{covariant model structure}. \par 
If we are given a morphism $f: S \longrightarrow T$ of dendroidal sets, this induces an adjunction
\[
\xymatrix@R=40pt@C=40pt{
f_!: \mathbf{dSets}/S \ar@<.5ex>[r] &  \mathbf{dSets}/T: f^* \ar@<.5ex>[l]
}
\]
An obvious question to ask is whether the covariant model structure behaves well with respect to this adjunction. The answer is yes, which we will also prove in section \ref{subsection:leftfibrations}:
\begin{proposition}
\label{prop:naturalitycovmodstruct}
Suppose we are given a morphism $f: S \longrightarrow T$ of dendroidal sets. Then the adjunction $(f_!, f^*)$ is a Quillen adjunction, which is a Quillen equivalence if $f$ is an operadic equivalence.
\end{proposition}

Before we state the following result, let us introduce some notation. If $p: X \longrightarrow S$ is a map of dendroidal sets and $s$ is a colour of $S$, we will often write
\begin{equation*}
X_s := \{s\} \times_S X
\end{equation*}
for the fiber of $p$ over $s$. Using Remark \ref{rmk:simplicialleftfib} and Proposition \ref{prop:leftfibcomppullback} we see that the fibers of a left fibration are always Kan complexes. \par 
It turns out that the weak equivalences between fibrant objects of $\mathbf{dSets}/S$ are easily characterized:

\begin{proposition}
Suppose we are given a diagram
\[
\xymatrix{
X \ar[rd]_p\ar[rr]^f & & Y \ar[dl]^q \\
& S & 
}
\]
in which $p$ and $q$ are left fibrations. Then the following are equivalent:
\begin{itemize}
\item[(i)] The map $f$ is a covariant equivalence
\item[(ii)] The map $f$ is an operadic equivalence
\item[(iii)] For every color $s \in S$ the induced map of fibers $f_s: X_s \longrightarrow Y_s$ is a weak homotopy equivalence of Kan complexes
\end{itemize} 
\end{proposition}

\subsection{The straightening functor}
In this section we will describe the relation between the category $\mathbf{dSets}/S$ and the category of algebras over the simplicial operad $\mathrm{hc}\tau_d(S)$ under the assumption that $S$ is cofibrant in the Cisinski-Moerdijk model structure, i.e. normal. The results of Berger and Moerdijk \cite{bergermoerdijk1}\cite{bergermoerdijk2} in particular yield the following:
\begin{theorem}
If $S$ is normal, so that $\mathrm{hc}\tau_d(S)$ is cofibrant, there exists a left proper simplicial model structure on the simplicial category $\mathrm{Alg}_{\mathrm{hc}\tau_d(S)}(\mathbf{sSets})$ of simplicial $\mathrm{hc}\tau_d(S)$-algebras in which a map of algebras is a weak equivalence (resp. a fibration) if and only if it is a pointwise weak equivalence (resp. a pointwise fibration).
\end{theorem}

We will now define the so-called \emph{straightening functor}
\begin{equation*}
St_S: \mathbf{dSets}/S \longrightarrow \mathrm{Alg}_{\mathrm{hc}\tau_d(S)}(\mathbf{sSets})
\end{equation*}
Note that we can also describe $\mathbf{dSets}/S$ as a presheaf category; indeed, we have
\begin{equation*}
\mathbf{dSets}/S \simeq \mathbf{Sets}^{(\int_{\mathbf{\Omega}}S)^{\mathrm{op}}}
\end{equation*} 
where $\int_{\mathbf{\Omega}}S$ is the category of elements of $S$. From this we conclude that $\mathbf{dSets}/S$ is generated under colimits by objects of the form $\Omega[T] \longrightarrow S$ for $T \in \mathbf{\Omega}$. Since the straightening functor is supposed to be a left adjoint, it will suffice to construct it on these generators and then extend its definition by a left Kan extension. \par 
First, consider the special case where $S = \Omega[T]$ and $p$ is the identity map of $\Omega[T]$. For any color $c$ of $T$ let $T/c$ denote the subtree of $T$ which consists of $c$ and `everything above $c$'. 

\begin{example}
If $T$ is the tree
\[
\xymatrix@R=10pt@C=12pt{
&&&&&&\\
&*=0{\bullet}\ar@{-}[u]&&*=0{\bullet}&&&\\
&&*=0{\bullet}\ar@{-}[ul]\ar@{-}[ur]&&*=0{\bullet}\ar@{-}[u]&&\\
&&&*=0{\bullet}\ar@{-}\ar@{-}[ul]^{c}\ar@{-}[ur]&&&\\
&&&\ar@{-}[u]&&&\\
&&&&&&
}
\]
then $T/c$ is the tree
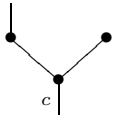
\[
\xymatrix@R=10pt@C=12pt{
&&&&&&\\
&*=0{\bullet}\ar@{-}[u]&&*=0{\bullet}&&&\\
&&*=0{\bullet}\ar@{-}[ul]\ar@{-}[ur]&&&\\
&&\ar@{-}[u]^c&&&&\\
&&&&&&
}
\]
\end{example} 

Define the cube
\begin{equation*}
\Delta[T/c] := (\Delta^1)^{\mathrm{col}(T/c)\backslash \{c\}}
\end{equation*}
where the product on the right is understood to be $\Delta^0$ if the set occurring in the exponent is empty.
The $\mathrm{hc}\tau_d(\Omega[T])$-algebra $St_{\Omega[T]}(\mathrm{id}_{\Omega[T]})$ is given by
\begin{equation*}
St_{\Omega[T]}(\mathrm{id}_{\Omega[T]})(c) := \Delta[T/c]
\end{equation*} 
The structure maps
\[
\xymatrix{
\mathrm{hc}\tau_d(\Omega[T])(c_1, \ldots, c_n; c) \times St_{\Omega[T]}(\mathrm{id}_{\Omega[T]})(c_1) \times \cdots \times St_{\Omega[T]}(\mathrm{id}_{\Omega[T]})(c_n) \ar[d] \\ St_{\Omega[T]}(\mathrm{id}_{\Omega[T]})(c)
}
\]
are given by grafting trees, assigning length 1 to the newly arising inner edges $c_1, \ldots, c_n$. 
\par 

\begin{remark}
When comparing this definition to the ones given in chapter \ref{section:grothconstinftyoperads}, one will see that at several places there we use the opposites of the simplicial sets used here. In the present setting we do not have to be very nitpicky about the orientations of our simplices; a simplicial set and its opposite have isomorphic geometric realizations and are hence equivalent. However, when we wish to model algebras in $\infty$-categories the orientation of our simplices does matter, reflecting the fact that a category is in general not equivalent to its opposite.
\end{remark}

Now let $S$ be any dendroidal set and $p: \Omega[T] \longrightarrow S$ a map. We get a map $\mathrm{hc}\tau_d(p)$ of simplicial operads, which induces an adjunction
\[
\xymatrix@R=40pt@C=40pt{
\mathrm{hc}\tau_d(p)_!: \mathrm{Alg}_{\mathrm{hc}\tau_d(\Omega[T])}(\mathbf{sSets}) \ar@<.5ex>[r] &  \mathrm{Alg}_{\mathrm{hc}\tau_d(S)}(\mathbf{sSets}): \mathrm{hc}\tau_d(p)^* \ar@<.5ex>[l]
}
\]
We set
\begin{equation*}
St_S(p) := \mathrm{hc}\tau_d(p)_!(St_{\Omega[T]}(\mathrm{id}_{\Omega[T]}))
\end{equation*}
Functoriality in $p$ is given as follows. Suppose we are given maps
\[
\xymatrix{
\Omega[R] \ar[r]^f & \Omega[T] \ar[r]^p & S
}
\]
If $f$ is a face map of $T$, the map $St_S(f)$ is described by the inclusions
\begin{equation*}
\Delta[S/c] \simeq \Delta[S/c] \times \{0\}^{\mathrm{col}(T/f(c))\backslash f(\mathrm{col}(S/c))} \longrightarrow \Delta[T/c]
\end{equation*}
for colours $c$ of $R$. If $f$ is a degeneracy, it is clear how to define $St_S(f)$. \par 
Having defined the functor $St_S$ on all maps of the form $\Omega[T] \longrightarrow S$, we take a left Kan extension of $St_S$ to all of $\mathbf{dSets}/S$ to obtain a functor
\begin{equation*}
St_S: \mathbf{dSets}/S \longrightarrow \mathrm{Alg}_{\mathrm{hc}\tau_d(S)}(\mathbf{sSets}): X \longmapsto \varinjlim_{\Omega[T] \rightarrow X} St_S(\Omega[T] \rightarrow X)
\end{equation*}
Since $St_S$ preserves colimits, the adjoint functor theorem provides us with a right adjoint to the straightening functor. We call this right adjoint the \emph{unstraightening functor} and denote it $Un_S$. We will later see that the unstraightening functor can be promoted to a simplicial functor. \par 
The following is the main result of this chapter and will be proven in section \ref{subsection:leftfibrations}.

\begin{theorem}
\label{thm:straighteningleftfib}
Let $S$ be a normal dendroidal set. Then the adjunction
\[
\xymatrix@R=40pt@C=40pt{
St_S: \mathbf{dSets}/S \ar@<.5ex>[r] & \mathrm{Alg}_{\mathrm{hc}\tau_d(S)}(\mathbf{sSets}): Un_S \ar@<.5ex>[l]
}
\]
is a Quillen equivalence.
\end{theorem}

Recall that if $\mathbf{C}$ is a (simplicial) model category, we use the notation $\mathbf{C}^\circ$ to denote its full (simplicial) subcategory on its fibrant-cofibrant objects. 

\begin{definition}
For any dendroidal set $S$ we define $\mathbf{LFib}(S)$, the $\infty$-category of left fibrations over $S$, by
\begin{equation*}
\mathbf{LFib}(S) := \mathrm{hc}N\bigl((\mathbf{dSets}/S)^\circ\bigr)
\end{equation*}
where $\mathbf{dSets}/S$ is equipped with the covariant model structure.
\end{definition}

Theorem \ref{thm:straighteningleftfib} and Proposition \ref{prop:naturalitycovmodstruct} allow us to interpret $\mathbf{LFib}(S)$ as the $\infty$-category of $S$-algebras in spaces.

\newpage

\section{coCartesian fibrations}
\subsection{coCartesian corollas}
\label{subsection:coCartcorollas}

Let $p: X \longrightarrow S$ be a morphism of dendroidal sets. Suppose $T$ is a tree with at least two vertices, having a leaf corolla with $n$ inputs, whose vertex we denote by $v$:
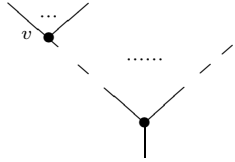
\[
\xymatrix@R=10pt@C=12pt{
&&&&&& \\
&*=0{\bullet}\ar@{-}[ul]\ar@{-}[ur]^(-.2)v\ar@{}[u]|{\cdots}&&&&&\\
&&\ar@{--}[ul]&\ar@{}[u]|{\cdots\cdots}&\ar@{--}[ur]&&\\
&&&*=0{\bullet}\ar@{-}\ar@{-}[ul]\ar@{-}[ur]&&&\\
&&&\ar@{-}[u]&&&\\
&&&&&&
}
\]
Recall that the horn $\Lambda^v[T]$ is the union of all the faces of $\Omega[T]$, except for the face obtained by `chopping off' the vertex $v$ and its input leaves. Now suppose $\alpha \in X_{C_n}$ is an $n$-corolla of $X$. Such an $\alpha$ is called \emph{$p$-coCartesian} if any diagram of the form
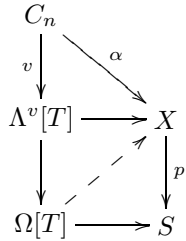
\[
\xymatrix{
C_n \ar[dr]^{\alpha}\ar[d]_v & \\
\Lambda^v[T] \ar[d]\ar[r] & X \ar[d]^p \\
\Omega[T] \ar[r]\ar@{-->}[ur] & S
}
\]
has a dotted lift as indicated, for any tree $T$ matching the description given above. By slight abuse of notation we have denoted the inclusion of the leaf corolla by $v$. \par 
We note the following basic properties of coCartesian corollas:
\begin{proposition}
\begin{itemize}
\item[(1)] If a corolla $\xi \in X_{C_n}$ is $p$-coCartesian, then so is $\sigma^*\xi$ for any $\sigma \in \Sigma_n$
\item[(2)] If $X$ and $S$ are $\infty$-operads, $p$ is an inner fibration and $\xi$ is a $p$-coCartesian 1-corolla such that $p(\xi)$ is an equivalence in $S$, then $\xi$ is an equivalence in $X$ 
\item[(3)] If $p: X \longrightarrow Y$ and $q: Y \longrightarrow Z$ are maps of dendroidal sets and $\xi \in X_{C_n}$ is such that $\xi$ is $p$-coCartesian and $p(\xi)$ is $q$-coCartesian, then $\xi$ is $p \circ q$-coCartesian
\end{itemize}
\end{proposition}
\begin{proof}
Proofs of (1) and (3) are straightforward. The second property only concerns the underlying $\infty$-categories of $X$ and $S$. This property is already known and can for example be found as Proposition 2.4.1.5 of \cite{htt}. $\Box$
\end{proof}

As is the case for ordinary operads, coCartesian corollas are closed under composition, provided the base $S$ is an $\infty$-operad and $p$ is an inner fibration:

\begin{proposition}
\label{prop:composition}
Let $p: X \longrightarrow S$ be an inner fibration of dendroidal sets and assume $S$ is an $\infty$-operad. Let $V$ be a tree with two vertices; denote its root corolla by $w$, the other one by $v$. We denote the inner edge of $V$ by $e$ and its root edge by $r$. Now assume we are given a map $\Omega[V] \longrightarrow X$ which sends the two corollas with vertices $v$ and $w$ to $p$-coCartesian corollas of $X$. Then the image of $\partial_e \Omega[V]$ is also a coCartesian corolla of $X$.
\end{proposition}
\begin{proof}
Suppose $T$ is a tree with at least two vertices, having a leaf corolla with vertex $u$, and that we are given a diagram
\[
\xymatrix{
C_n \ar[dr]^{\partial_e \Omega[V]}\ar[d]_u & \\
\Lambda^u[T] \ar[d]\ar[r] & X \ar[d]^p \\
\Omega[T] \ar[r] & S
}
\]
We wish to produce a lift $\Omega[T] \longrightarrow X$ in this diagram. First, construct a new tree $\tilde T$ from $T$  by replacing the corolla with vertex $u$ by the tree $V$. Use Lemma \ref{lemma:innerextension} and the fact that $S$ is inner Kan to guarantee the existence of a map $\Omega[\tilde T] \longrightarrow S$ compatible with the given maps from $\Omega[T]$ and $\Omega[V]$ into $S$. \par
Now note that any face $f$ of $\Omega[T]$, except for the face $\partial_u \Omega[T]$, will induce a corresponding face of $\Omega[\tilde T]$, which we denote by $\tilde f$. Denote the set of these faces by $F(\tilde T)$. Observe that
\begin{equation*}
\Lambda^e[\tilde T] = \partial_v \Omega[\tilde T] \cup \partial_r \Omega[\tilde T] \cup F(\tilde T)
\end{equation*}
Our goal is to construct compatible maps from the three subsets indicated on the right-hand side into $X$. First assume $f$ is a face of $\Omega[T]$ different from $\partial_u \Omega[T]$. By applying Lemma \ref{lemma:innerextension} to the trees $f$ and $V$, we see that we get a map $F(\tilde T) \longrightarrow X$ compatible with the map $\Omega[\tilde T] \longrightarrow S$. \par
Now consider the horn $\Lambda^v[\partial_r\tilde T]$. It is given by the union of the face $\partial_e \partial_r \Omega[\tilde T]$ with the set of faces $\partial_r(F(\tilde T))$. By what we have just proved, we already have suitable maps from $\partial_r(F(\tilde T))$ into $X$. We also have $\partial_e \partial_r \Omega[\tilde T] = \partial_r \partial_e \Omega[\tilde T] = \partial_r \Omega[T]$, from which we already have a map into $X$ by assumption. Now, by the fact that the corolla with vertex $v$ is $p$-coCartesian, we get the following lift:
\[
\xymatrix{
\Lambda^v[\partial_r\tilde T] \ar[d]\ar[r] & X \ar[d]^p \\
\partial_r\Omega[\tilde T] \ar[r]\ar@{-->}[ur] & S
}
\]
Finally, consider the horn $\Lambda^w[\partial_v\tilde T]$. Observe that $\partial_r \partial_v \Omega[\tilde T] = \partial_v \partial_r \Omega[\tilde T]$, so on this face our map to $X$ is defined by what we just proved. The faces $\partial_v(F(\tilde T))$ make up the rest of this horn, so we see that our map is defined on all of $\Lambda^w[\partial_v\tilde T]$. Now using that the corolla with vertex $w$ is coCartesian, we find the lift
\[
\xymatrix{
\Lambda^w[\partial_v\tilde T] \ar[d]\ar[r] & X \ar[d]^p \\
\partial_v \Omega[\tilde T] \ar[r]\ar@{-->}[ur] & S
}
\]
We have now built a diagram
\[
\xymatrix{
\Lambda^e[\tilde T] \ar[d]\ar[r] & X \ar[d]^p \\
\Omega[\tilde T] \ar[r]\ar@{-->}[ur] & S
}
\]
in which the dotted lift exists by the fact that $p$ is an inner fibration. The dotted map gives us a map from $\partial_e \Omega[\tilde T] = \Omega[T]$ to $X$, which is a lift in our original diagram. $\Box$
\end{proof}

This proposition actually admits a converse as follows:

\begin{proposition}
\label{prop:composition2}
Let $p$ and $V$ be as before and again assume $S$ is an $\infty$-operad. Suppose we are given a map $\phi: \Omega[V] \longrightarrow X$ which sends the corolla with vertex $v$ and the corolla $\partial_e \Omega[V]$ to $p$-coCartesian corollas of $X$. Then the image of the corolla with vertex $w$ will again be a $p$-coCartesian corolla of $X$.
\end{proposition}
\begin{proof}
The strategy is similar to the proof of the previous proposition. Again, suppose $T$ is a tree with at least two vertices, having a leaf corolla with a vertex we will also call $w$, and that we are given a diagram
\[
\xymatrix{
C_n \ar[dr]^{\phi(w)}\ar[d]_w & \\
\Lambda^w[T] \ar[d]\ar[r] & X \ar[d]^p \\
\Omega[T] \ar[r] & S
}
\]
We wish to produce a lift $\Omega[T] \longrightarrow X$ in this diagram. Construct a new tree $\tilde T$ from $T$  by replacing the corolla with vertex $w$ by the tree $V$. Use Lemma \ref{lemma:innerextension} and the fact that $S$ is inner Kan to guarantee the existence of a map $\Omega[\tilde T] \longrightarrow S$ compatible with the given maps from $\Omega[T]$ and $\Omega[V]$ into $S$. \par
Again, any face $f$ of $\Omega[T]$ except for the face $\partial_w \Omega[T]$ will induce a corresponding face of $\Omega[\tilde T]$, denoted by $\tilde f$. We denote the set of these faces by $F(\tilde T)$. We will now consider the horn $\Lambda^v[\tilde T]$. Observe that
\begin{equation*}
\Lambda^v[\tilde T] = \partial_e \Omega[\tilde T] \cup \partial_r \Omega[\tilde T] \cup F(\tilde T)
\end{equation*}
Our goal is to construct compatible maps from the three subsets indicated on the right-hand side into $X$. First assume $f$ is a face of $\Omega[T]$ different from $\partial_v \Omega[T]$. By applying Lemma \ref{lemma:innerextension} to the trees $f$ and $V$, we see that we get a map $\tilde f \longrightarrow X$ compatible with the map $\Omega[\tilde T] \longrightarrow S$. Hence we can define $\phi$ on all of $F(\tilde T)$. \par
Now consider the horn $\Lambda^e[\partial_r \tilde T]$. It consists of the face $\partial_v \partial_r \Omega[\tilde T]$ together with the faces $\partial_r F(\tilde T)$. By what we just showed, we already have suitable maps from the latter faces into $X$. Also, we have $\partial_v \partial_r \Omega[\tilde T] = \partial_r \partial_v \Omega[\tilde T] = \partial_r \Omega[T]$, from which we already have a map to $X$ by assumption. Since $e$ is an inner edge of $\partial_r \tilde T$ and $p$ is an inner fibration, we get a lift as indicated in the following diagram:
\[
\xymatrix{
\Lambda^e[\partial_r \tilde T] \ar[d]\ar[r] & X \ar[d]^p \\
\partial_r \Omega[\tilde T] \ar[r]\ar@{-->}[ur] & S
}
\]
We denote the corolla obtained from the corollas $v$ and $w$ by contracting along the inner edge $e$ by $w \circ v$. We continue by considering the horn $\Lambda^{w \circ v}[\partial_e \tilde T]$. It consists of the faces $\partial_e F(\tilde T)$, on which our map to $X$ has already been defined, and the face $\partial_r \partial_e \Omega[\tilde T] = \partial_e \partial_r \Omega[\tilde T]$. Since we just extended our map to $\partial_r \Omega[\tilde T]$, it is in particular defined on this latter face. Hence we have a map $\Lambda^{w \circ v}[\partial_e \tilde T] \longrightarrow X$. By assumption, the image of $w \circ v$ is $p$-coCartesian, so we can find a lift in the following diagram:
\[
\xymatrix{
\Lambda^{w \circ v}[\partial_e \tilde T] \ar[d]\ar[r] & X \ar[d]^p \\
\partial_e \Omega[\tilde T] \ar[r]\ar@{-->}[ur] & S
}
\]
Putting everything together, we have built a diagram
\[
\xymatrix{
\Lambda^v[\tilde T] \ar[d]\ar[r] & X \ar[d]^p \\
\Omega[\tilde T] \ar[r]\ar@{-->}[ur] & S
}
\]
The lift exists since the image of $v$ is a $p$-coCartesian corolla by assumption. The restriction of this lift to $\partial_v \Omega[\tilde T] = \Omega[T]$ will serve as a lift in the original diagram, completing the proof. $\Box$
\end{proof}

As is to be expected, coCartesian lifts of corollas with specified inputs are unique up to a contractible space of choices. Suppose $X$ and $S$ are dendroidal sets and $p: X \longrightarrow S$ is a map between them. Let $\sigma$ be a corolla of $S$ with leaves $s_1, \ldots, s_k$ and root $s$ and let $x_1, \ldots, x_k \in X$ be such that $p(x_i) = s_i$ for $1 \leq i \leq k$. Then we define a simplicial set $\mathbf{coCart}_p(\sigma, \{x_1, \ldots, x_k\})$ as follows. We have inclusions
\begin{equation*}
\iota_n: [n] \longrightarrow C_k \star [n-1]
\end{equation*}
which send $\{0\}$ to the root of $C_k$ and $\{i\}$ to $\{i-1\}$ for $i \geq 1$. The join operation $\star$ above is given by putting a vertex at the root of $C_k$ and attaching the linear tree $[n-1]$ to that. We adopt the convention $[-1] = \emptyset$ and let $C_k \star \emptyset$ be just $C_n$. Using the maps $\iota_\bullet$ we obtain a simplicial set
\begin{equation*}
Q: [n] \mapsto \mathbf{dSets}(\Omega[C_k \star [n-1]], X)
\end{equation*}

Now let $\mathbf{coCart}_p(\sigma, \{x_1, \ldots, x_k\})$ be the simplicial subset of $Q$ consisting of maps $f$ satisfying the following conditions:
\begin{itemize}
\item[(1)] The inputs of $f(C_k)$ are $x_1, \ldots, x_k$ and $p \circ f (C_k) = \sigma$
\item[(2)] The corolla $f(C_k)$ is $p$-coCartesian
\item[(3)] The simplex $f(\iota_n([n]))$ is contained in the fiber $X_s := \eta_s \times_{S} X$ over $s$
\end{itemize}
It is clear that the vertices of $\mathbf{coCart}_p(\sigma, \{x_1, \ldots, x_k\})$ are exactly the coCartesian lifts of $\sigma$ with the specified leaves. Note that the embedding $\Delta^{\bullet} \hookrightarrow \Omega[C_k] \star \Delta^{\bullet-1}$ induces, by pullback, a map of simplicial sets
\begin{equation*}
\mathbf{coCart}_p(\sigma, \{x_1, \ldots, x_k\}) \longrightarrow \mathbf{dSets}(\Delta^\bullet, X) = i^*(X)
\end{equation*}
which will actually factor through the fiber $i^*(X_s)$. Note that, since $X_s$ allows a map to $\eta$, it can be identified with a simplicial set. We will therefore forget about the distinction between $i^*(X_s)$ and $X_s$ from now on.

\begin{proposition}
\label{prop:uniquecoCartlift}
If $p: X \longrightarrow S$  is a map of dendroidal sets, $\sigma$ is a corolla of $S$ as above and $x_1, \ldots, x_k \in X$ are such that $p(x_i) = s_i$ for $1 \leq i \leq k$, then $\mathbf{coCart}_p(\sigma, \{x_1, \ldots, x_k\})$ is a contractible Kan complex.
\end{proposition}
\begin{proof}
Suppose we are given a map $\partial \Delta^n \longrightarrow \mathbf{coCart}_p(\sigma, \{x_1, \ldots, x_k\})$. Denoting the vertex of $C_k$ by $v$, one easily sees that this data is equivalent to a map
\begin{equation*}
\Lambda^v[C_k \star [n-1]] \longrightarrow X
\end{equation*}
By the fact that $C_k$ is mapped to a $p$-coCartesian corolla, we can extend our map to a map $\Omega[C_k] \star \Delta^{n-1} \longrightarrow X$ satisfying the conditions listed above. Hence, we have found an extension of our original map to a map $\Delta^n \longrightarrow \mathbf{coCart}_p(\sigma, \{x_1, \ldots, x_k\})$. $\Box$
\end{proof}

\begin{corollary}
\label{cor:uniquecoCartlift}
Let $p: X \longrightarrow S$ be an inner fibration, and let $\sigma$ be as in the proposition. Let $\alpha, \beta \in \mathbf{coCart}_p(\sigma, \{x_1, \ldots, x_k\})$. Set $T = C_k \star [0]$ and denote the root of $C_k$ by $r$. Then there exists a map $\Omega[T] \longrightarrow X$ mapping $C_k$ to $\alpha$, the inner face of $T$ to $\beta$ and $\{r\} \star [0]$ to an equivalence in the $\infty$-category $X_{s}$. In other words, $\beta$ equals a composition of $\alpha$ with an equivalence.
\end{corollary}
\begin{proof}
The proposition tells us that $\mathbf{coCart}_p(\sigma, \{x_1, \ldots, x_k\})$ maps to a connected Kan complex in $X_s$. Hence this is a subcomplex of the maximal Kan complex of $X_s$, all of whose 1-simplices are equivalences. By connectedness there is a 1-simplex from $\alpha$ to $\beta$ in $\mathbf{coCart}_p(\sigma, \{x_1, \ldots, x_k\})$, which is necessarily an equivalence. $\Box$
\end{proof}

\subsection{coCartesian fibrations}
\begin{definition}
A morphism $p: X \longrightarrow S$ of dendroidal sets is a \emph{coCartesian fibration} if the following conditions are satisfied:
\begin{itemize}
\item[(1)] The map $p$ is an inner fibration of dendroidal sets
\item[(2)] For any corolla $\sigma$ of $S$ with input colours $s_1, \ldots, s_n$ and output $s$ and colours $x_1, \ldots, x_n$ of $X$ satisfying $p(x_i) = s_i$, there exists a $p$-coCartesian corolla $\xi$ of $X$ with input colours $x_1, \ldots, x_n$ such that $p(\xi) = \sigma$
\end{itemize}
\end{definition}

\begin{remark}
Note that a left fibration is precisely a coCartesian fibration $p$ in which every corolla is $p$-coCartesian.
\end{remark}

We note the following properties of coCartesian fibrations, which are easily proven:
\begin{proposition}
\label{prop:pullbackfibration}
A composition of coCartesian fibrations is again a coCartesian fibration. Furthermore, if we are given a pullback square
\[
\xymatrix{
X' \ar[d]\ar[r] & X \ar[d] \\
S' \ar[r] & S
}
\]
in which the right vertical map is a coCartesian fibration, then the left vertical map is so as well.
\end{proposition}

The following proposition describes the relationship between coCartesian fibrations and opfibered operads and partly justifies thinking of coCartesian fibrations as `opfibered $\infty$-operads'.

\begin{proposition}
Let  $X$ and $S$ be operads in $\mathbf{Sets}$ and let $p: X \longrightarrow S$ be a morphism between them. Then $p$ exhibits $X$ as being opfibered over $S$ if and only if the induced map $N_d(p): N_d(X) \longrightarrow N_d(S)$ is a coCartesian fibration.
\end{proposition}
\begin{proof}
The proof is entirely analogous to that of Lemma \ref{lemma:opfibgrpdsleftfib}. $\Box$ 
\end{proof}

\newpage


\section{The coCartesian model structure}
In order to describe a version of the operadic Grothendieck construction in the setting of $\infty$-operads, we'd like to construct an $\infty$-category of coCartesian fibrations of dendroidal sets over a fixed dendroidal set $S$. In order to do this, we will consider \emph{marked dendroidal sets}, which are dendroidal sets equipped with a certain subset of their set of corollas, whose elements we will refer to as \emph{marked corollas}. Given a dendroidal set $S$ we will construct a (combinatorial) simplicial model category $\mathbf{dSets^+}/S$ of \emph{marked dendroidal sets over $S$} in which the fibrant objects will exactly be the coCartesian fibrations over $S$ (with coCartesian corollas in $X$ marked). \par
The outline of the approach is identical to the one of Lurie in the third chapter of \cite{htt}, who does all of this in the context of $\infty$-categories and marked simplicial sets. The results in this text are a direct generalization of Lurie's and reproduce his results. The difference is that the greater generality of dendroidal sets makes the explicit combinatorics more involved. \par

\subsection{Marked dendroidal sets}
\noindent Given a dendroidal set $X$, we denote by $cor(X)$ its set of corollas, i.e.

\begin{equation*}
cor(X) = \coprod_{n \in \mathbb{Z}_{\geq 0}} X_{C_n}
\end{equation*}

The category of \emph{marked dendroidal sets}, denoted $\mathbf{dSets^+}$, has as objects pairs $(X, \mathcal{E}_X)$, where $X$ is a dendroidal set and $\mathcal{E}_X$ is a subset of $cor(X)$ containing all degenerate 1-corollas and being closed under the actions of the symmetric groups. corollas contained in $\mathcal{E}_X$ will be called \emph{marked}. A morphism $f: (X, \mathcal{E}_X) \longrightarrow (Y, \mathcal{E}_Y)$ is a morphism $X \longrightarrow Y$ mapping marked corollas to marked corollas. We will often abuse notation and just write $X$ when actually the marked dendroidal set $(X, \mathcal{E}_X)$ is meant. \par
We have an obvious forgetful functor $u: \mathbf{dSets^+} \longrightarrow \mathbf{dSets}$. This functor has a left adjoint $(-)^{\flat}$ and a right adjoint $(-)^{\sharp}$. For a dendroidal set $X$, we will have $X^\flat = (X, s_0(X_\eta))$ (only degenerate 1-corollas are marked) and $X^\sharp = (X, cor(X))$ (all corollas are marked). Similar functors can be defined in the context of marked simplicial sets and we will use the same symbols to denote them when there is no danger of confusion. The canonical inclusion $i: \mathbf{\Delta} \longrightarrow \mathbf{\Omega}$ induces a restriction functor $i^*: \mathbf{dSets} \longrightarrow \mathbf{sSets}$. By Kan extension this functor has a left adjoint $i_!$ which is easily seen to be full and faithful. Clearly there are induced functors $j^*: \mathbf{dSets^+} \longrightarrow \mathbf{sSets^+}$ and $j_!: \mathbf{sSets^+} \longrightarrow \mathbf{dSets^+}$ fitting into the following diagram: \\

\[
\xymatrix@R=40pt@C=40pt{
\mathbf{dSets^+} \ar[d]^{j^*}\ar[r]^u & \mathbf{dSets} \ar@<2ex>[l]^{(-)^\sharp} \ar@<-2ex>[l]_{(-)^\flat}  \ar[d]^{i^*} \\
\mathbf{sSets^+} \ar[r]^u \ar@<1ex>[u]^{j_!} & \mathbf{sSets}\ar@<1ex>[u]^{i_!} \ar@<2ex>[l]^{(-)^\sharp} \ar@<-2ex>[l]_{(-)^\flat}
}
\]

In this diagram, we have identities $u \circ j_! = i_! \circ u$, $(-)^{\sharp/\flat} \circ i_! = j_! \circ (-)^{\sharp/\flat}$ and similar ones involving $i^*$ and $j^*$. We also have $u \circ (-)^{\sharp/\flat} = \mathrm{id}$, $i^*i_! = \mathrm{id}$ and $j^*j_! = \mathrm{id}$.\par
The tensor product on $\mathbf{dSets}$ (see \cite{dendroidalsets}) can be used to define a tensor product on $\mathbf{dSets^+}$. Indeed, set
\begin{equation*}
(X, \mathcal{E}_X) \otimes (Y, \mathcal{E}_Y) = (X \otimes_{\mathbf{dSets}} Y, \mathcal{E}_{X \otimes Y})
\end{equation*}
where
\begin{equation*}
\mathcal{E}_{X \otimes Y} = cor(\overline{\mathcal{E}_X} \otimes_{\mathbf{dSets}} \overline{\mathcal{E}_Y})
\end{equation*}
Here $\overline{\mathcal{E}_X}$ denotes the smallest dendroidal subset of $X$ containing $\mathcal{E}_X$; the dendroidal set $\overline{\mathcal{E}_Y}$ is defined similarly. Unraveling the definition, one sees that a marked corolla of the tensor product is either a marked corolla of one of the factors (tensored with a colour of the other) or a corolla obtained by contracting all the inner edges of a dendrex of $X \otimes Y$ as depicted below, where the `black' corollas are copies of a marked corolla of $X$ and the `white' corolla is a marked corolla of $Y$:

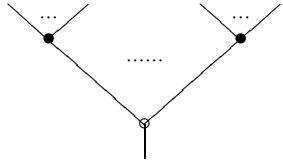
\[
\xymatrix@R=10pt@C=12pt{
&&&&&& \\
&*=0{\bullet}\ar@{-}[ul]\ar@{-}[ur]\ar@{}[u]|{\cdots}&&&&*=0{\bullet}\ar@{-}[ul]\ar@{-}[ur]\ar@{}[u]|{\cdots}&\\
&&&\ar@{}[u]|{\cdots\cdots}&&&\\
&&&*=0{\circ}\ar@{-}\ar@{-}[uull]\ar@{-}[uurr]&&&\\
&&&\ar@{-}[u]&&&\\
&&&&&&
}
\]

One verifies that the tensor product on $\mathbf{dSets^+}$ as just defined commutes with the formation of colimits in each variable separately.

\begin{remark}
The tensor product on $\mathbf{dSets^+}$ extends the tensor product on $\mathbf{sSets^+}$ defined by
\begin{equation*}
(X, \mathcal{E}_X) \otimes_{\mathbf{sSets^+}} (Y, \mathcal{E}_Y) = (X \times Y, \mathcal{E}_X \times \mathcal{E}_Y)
\end{equation*}
Indeed, it is known that the tensor product on $\mathbf{dSets}$ extends the usual product on $\mathbf{sSets}$ (see e.g. \cite{dendroidalsets}) and for two marked simplicial sets $X$ and $Y$ any element of $\mathcal{E}_X \times \mathcal{E}_Y$ can be described as an inner face of a 2-simplex whose two outer faces are a marked 1-simplex of $X$ and a marked 1-simplex of $Y$.
\end{remark}

Let $X$ and $Y$ be marked dendroidal sets. We define the internal hom-object $\mathbf{Hom}_{\mathbf{dSets}^+}(X,Y)$ by
\begin{eqnarray*}
\mathbf{Hom}_{\mathbf{dSets^+}}(X,Y)_T & = & \mathbf{dSets^+}(X \otimes \Omega[T]^\flat, Y) \\
\mathcal{E}_{\mathbf{Hom}_{\mathbf{dSets^+}}(X,Y)} & = & \coprod_{n \in \mathbb{Z}_{\geq 0}} \mathbf{dSets^+}(X \otimes \Omega[C_n]^\sharp, Y)
\end{eqnarray*}
One checks explicitly that $\mathbf{Hom}_{\mathbf{dSets}^+}(X,-)$ is right adjoint to $- \otimes X$, so that $\mathbf{dSets^+}$ acquires the structure of a closed symmetric monoidal category. \par
We also define
\begin{eqnarray*}
\mathrm{Map}^\flat(X,Y) & = & (u \circ j^*)(\mathbf{Hom}_{\mathbf{dSets^+}}(X,Y))
\end{eqnarray*}
and we let $\mathrm{Map}^\sharp(X,Y)$ be the simplicial subset of $\mathrm{Map}^\flat(X,Y)$ consisting of the simplices all of whose edges are marked in $j^*(\mathbf{Hom}_{\mathbf{dSets^+}}(X,Y))$. We can characterize these objects by universal mapping properties. Indeed, for $K$ a simplicial set we have natural isomorphisms
\begin{eqnarray*}
\mathbf{sSets}(K,\mathrm{Map}^\flat(X,Y)) & \simeq & \mathbf{dSets^+}((i_!(K))^\flat \otimes X, Y) \\
\mathbf{sSets}(K,\mathrm{Map}^\sharp(X,Y)) & \simeq & \mathbf{dSets^+}((i_!(K))^\sharp \otimes X, Y)
\end{eqnarray*}
For the second isomorphism we have used the fact that $(-)^\sharp: \mathbf{sSets} \longrightarrow \mathbf{sSets^+}$ is left adjoint to the functor taking an object $X \in \mathbf{sSets^+}$ to the simplicial set consisting of all simplices whose edges are marked in $X$. \par
Now observe that for any simplicial set $K$ there is a projection $i_!(K) \otimes u(X) \longrightarrow u(X)$ which on colours can be described as projection onto the second factor. (Note that this depends crucially on the fact that $i_!(K)$ has only unary corollas.) Hence, we can regard $i_!(K) \otimes u(X)$ (and similarly $i_!(K) \otimes u(Y)$) as living over $S$. Now define $\mathrm{Map}_S^\flat(X,Y)$ and $\mathrm{Map}_S^\sharp(X,Y)$ to be the simplicial subsets of $\mathrm{Map}^\flat(X,Y)$, respectively $\mathrm{Map}^\sharp(X,Y)$, which represent maps from $X$ to $Y$ compatible with the maps to $S$, i.e. which can be described by
\begin{eqnarray*}
\mathbf{sSets}(K,\mathrm{Map}_S^\flat(X,Y)) & \simeq & (\mathbf{dSets^+}/S^\sharp)((i_!(K))^\flat \otimes X, Y) \\
\mathbf{sSets}(K,\mathrm{Map}_S^\sharp(X,Y)) & \simeq & (\mathbf{dSets^+}/S^\sharp)((i_!(K))^\sharp \otimes X, Y)
\end{eqnarray*}
Both these mapping objects make the overcategory $\mathbf{dSets^+}/S^\sharp$ into a simplicial category. We will abuse notation and from now on write $\mathbf{dSets^+}/S$.

\begin{definition}
Suppose $p: X \longrightarrow S$ is a coCartesian fibration of dendroidal sets. We will denote by $X^\natural$ the marked dendroidal set whose marked corollas are precisely the coCartesian corollas of $X$.
\end{definition}

We have the following result:

\begin{proposition}
\label{prop:mappingobjects}
Let $X \longrightarrow S^\sharp$ be a map in $\mathbf{dSets^+}$ and let $Y \longrightarrow S$ be a coCartesian fibration. Suppose furthermore that the dendroidal set $u(X)$ is normal. Then $\mathrm{Map}_S^\flat(X,Y^\natural)$ is an $\infty$-category and $\mathrm{Map}_S^\sharp(X,Y^\natural)$ is the largest Kan complex contained in it.
\end{proposition}
\begin{proof}
Observe that $\mathrm{Map}_S^\flat(X,Y^\natural)$ can be identified with a full simplicial subset of $\bigl(i^*\mathbf{Hom}_{\mathbf{dSets}}(u(X),Y)\bigr)_S$, where the subscript $S$ is again meant to indicate maps compatible with the projections to $S$. Hence it suffices to show this latter simplicial set is an $\infty$-category. Let $0 < k < n$ and suppose we are given a diagram
\[
\xymatrix{
\Lambda^n_k \ar[r]\ar[d] & \bigl(i^*\mathbf{Hom}_{\mathbf{dSets}}(u(X),Y)\bigr)_S \\
\Delta^n \ar@{-->}[ur]
}
\]
By adjunction, finding the dotted lift in this diagram is equivalent to finding the dotted lift in
\[
\xymatrix{
i_!(\Lambda^n_k) \otimes u(X) \ar[r]\ar[d] & Y \ar[d] \\
i_!(\Delta^n) \otimes u(X) \ar@{-->}[ur]\ar[r] & S
}
\]
Since $u(X)$ is normal, the left vertical map is the smash product (or pushout-product) of the cofibration $\emptyset \longrightarrow u(X)$ and the inner anodyne extension $i_!(\Lambda^n_k) \longrightarrow i_!(\Delta^n)$. We know (see e.g. \cite{moerdijkweiss2}) that this is again inner anodyne. Since $Y \longrightarrow S$ is an inner fibration we conclude that the dotted lift exists. \par
The proof of the second claim will be given in the next section, after we have shown certain stability properties of marked anodynes. $\Box$
\end{proof}

\subsection{Marked anodyne morphisms}
In this section we will introduce \emph{marked anodyne} morphisms, which will serve as a convenient technical tool. Given a tree $T$, a \emph{leaf corolla} of $T$ is a corolla all of whose input edges are leaves of $T$.  We define the class of marked anodyne morphisms to be the smallest weakly saturated class of morphisms containing
\begin{itemize}
\item[(1)] All inner horn inclusions $\Lambda^e[T]^\flat \subseteq \Omega[T]^\flat$, where $e$ denotes an inner edge of $T$
\item[(2)] The inclusions
    \begin{equation*}
    (\Lambda^v[T], \mathcal{E}\cap cor(\Lambda^v[T])) \subseteq (\Omega[T], \mathcal{E})
    \end{equation*}
    where $T$ is a tree with at least two vertices, $v$ is the vertex of a leaf corolla and where $\mathcal{E}$ is the union of all degenerate 1-corollas of $T$ with the leaf corolla with vertex $v$
\item[($2^*$)] For all $n \geq 0$, the inclusion
    \begin{equation*}
    \coprod_{c \in \mathrm{in}(C_n)} \eta_c \subseteq \Omega[C_n]^\sharp
    \end{equation*}
    where $\mathrm{in}(C_n)$ denotes the set of input edges of $C_n$
\item[(3)] For a tree $T$ with two vertices $v$ and $w$, the inclusion
    \begin{equation*}
    (\Omega[T], \mathcal{E}) \subseteq \Omega[T]^\sharp
    \end{equation*}
    where $\mathcal{E}$ is the union of the degenerate 1-corollas of $T$ together with the two corollas whose vertices are $v$, resp. $w$
\item[(4)] For any simplicial Kan complex $K$, the inclusion $j_!(K^\flat) \subseteq j_!(K^\sharp)$
\end{itemize}

The first thing we want to do is to characterize maps having the right lifting property with respect to marked anodynes.

\begin{proposition}
\label{prop:lifting}
A map $p: X \longrightarrow Y$ in $\mathbf{dSets^+}$ has the right lifting property with respect to all marked anodynes if and only if
\begin{itemize}
\item[(a)] $u(p)$ is an inner fibration of dendroidal sets
\item[(b)] A corolla $\alpha$ of $X$ is marked if and only if $p(\alpha)$ is marked and $\alpha$ is $p$-coCartesian
\item[(c)] For every marked corolla $\beta$ of $Y$ with inputs $y_1, \ldots, y_n$ and colours $x_1, \ldots, x_n$ of $X$ such that $p(x_i) = y_i$ for $1 \leq i \leq n$, there exists a marked corolla $\alpha$ of $X$ with inputs $x_1, \ldots, x_n$ such that $p(\alpha) = \beta$
\end{itemize}
\end{proposition}
\begin{proof}
First assume $p$ has the right lifting property with respect to all marked anodynes. Then (a) is immediate from (1) and the fact that $(-)^\flat$ is left adjoint to $u$, and (c) is immediate from ($2^*$). Suppose now that $\alpha$ is an $n$-corolla of $X$ with root $r$. If $\alpha$ is marked then the right lifting property with respect to (2) implies that $\alpha$ is $p$-coCartesian. Conversely, suppose $p(\alpha)$ is marked and $\alpha$ is $p$-coCartesian. Use (c) to pick a marked $p$-coCartesian lift $\tilde \alpha$ of $p(\alpha)$ with the same inputs as $\alpha$. Now define a tree $T = C_n \star [0]$, where we denote the vertex of $C_n$ by $v$ and the root vertex by $w$, and let $\mathcal{E}_T$ be the set of degenerate corollas of $\Omega[T]$ together with the corollas with vertices $v$ and $w$. By Corollary \ref{cor:uniquecoCartlift} there is a map $(\Omega[T], \mathcal{E}_T) \longrightarrow X$ such that the corolla with vertex $v$ is mapped to $\tilde \alpha$, the corolla with vertex $w$ is mapped to an equivalence in the $\infty$-category $X_{p(r)}$ and the inner face $\partial_e \Omega[T]$ is mapped to $\alpha$. If we can show that the image of the 1-corolla with vertex $w$ is marked in $X$, we can deduce from (3) that $\alpha$ is marked. Now let $K$ be the maximal simplicial Kan complex of $u\circ j^*(X_{p(r)})$, whose 1-simplices are precisely the equivalences in $X_{p(r)}$. From (4) we deduce that every such equivalence is marked, hence also the image of the corolla with vertex $w$. \par
Conversely, assume that $p$ is a map satisfying (a), (b) and (c). The right lifting property with respect to (1), resp. ($2^*$), follows immediately from (a), resp. (c). The possibility of lifting with respect to the maps of (2) follows from (b). Considering (3) we can, by pulling back $p$ along the given map $\Omega[T]^\sharp \longrightarrow S$, reduce to the case where $S = \Omega[T]^\sharp$, which is an $\infty$-operad. We can then apply Proposition \ref{prop:composition}, which tells us coCartesian corollas are closed under composition, and conclude that $p$ has the right lifting property with respect to the maps of (3). Now suppose $K$ is a simplicial Kan complex. Again, we can reduce to the case where $S = j_!(K^\sharp)$. Since $p$ is an inner fibration, $X$ will be an $\infty$-operad. Giving a map $j_!(K^\flat) \longrightarrow X$ is the same as giving a map $K \longrightarrow u\circ j^*(X)$. Any such map will factor through the maximal Kan complex of $u\circ j^*(X)$ and hence it will map 1-simplices of $K$ to equivalences in $X$. From Theorem A.7 of \cite{moerdijkcisinski} we conclude that any equivalence is $p$-coCartesian and hence marked by (b). Therefore $p$ will have the right lifting property with respect to the maps of (4). $\Box$
\end{proof}

\begin{corollary}
\label{cor:liftwrtmarked}
A map $p: (X, \mathcal{E}_X) \longrightarrow Y^\sharp$ in $\mathbf{dSets^+}$ has the right lifting property with respect to all marked anodynes if and only if $u(p)$ is a coCartesian fibration and $(X, \mathcal{E}_X) = X^\natural$.
\end{corollary}

\begin{corollary}
\label{cor:markedanodyne3}
For a tree $T$ with two vertices $v$ and $w$ ($w$ denoting the root vertex), let $\mathcal{E}$ denote the union of the degenerate 1-corollas of $T$ together with the corolla $v$ and the corolla obtained by contracting the inner edge of $T$. Then the inclusion
\begin{equation*}
(\Omega[T], \mathcal{E}) \subseteq \Omega[T]^\sharp
\end{equation*}
is marked anodyne.
\end{corollary}
\begin{proof}
By Proposition \ref{prop:composition2} this inclusion has the left lifting property with respect to the maps satisfying conditions (a), (b) and (c) of the Proposition above. The conclusion follows from the fact that marked anodyne morphisms form a weakly saturated class. $\Box$
\end{proof}

\begin{definition}
A map $f: X \longrightarrow Y$ in $\mathbf{dSets^+}$ is called a \emph{cofibration} if the underlying map $u(f)$ is a normal monomorphism of dendroidal sets, i.e. a cofibration in the usual model structure on $\mathbf{dSets}$.
\end{definition}

\begin{remark}
Observe that the cofibrations in $\mathbf{dSets^+}$ are generated as a weakly saturated class by the boundary inclusions $\partial \Omega[T]^\flat \subseteq \Omega[T]^\flat$ and the inclusions $\Omega[C_n]^\flat \subseteq \Omega[C_n]^\sharp$ for $n \in \mathbb{Z}_{\geq 0}$.
\end{remark}

The following proposition will be crucial in what follows.

\begin{proposition}
\label{prop:smashproduct}
The class of marked anodyne morphisms is stable under smash products with arbitrary cofibrations, i.e. given a marked anodyne $f: A \longrightarrow B$ and a cofibration $g: X \longrightarrow Y$, the induced map
\begin{equation*}
(B \otimes X) \coprod_{A \otimes X} (A \otimes Y) \longrightarrow X \otimes Y
\end{equation*}
is again marked anodyne.
\end{proposition}
\begin{proof}
By standard arguments involving weakly saturated classes it suffices to check this on generating families of marked anodynes and cofibrations. Hence we may assume $f$ belongs to one of the classes (1), (2), ($2^*$), (3) or (4) and we may suppose $g$ equals either a boundary inclusion $\partial \Omega[T]^\flat \subseteq \Omega[T]^\flat$ or an inclusion $\Omega[C_k]^\flat \subseteq \Omega[C_k]^\sharp$ for some $k \in \mathbb{Z}_{\geq 0}$. We need to check ten cases. \par
\textbf{(1a)} Suppose $f$ is an inner horn inclusion $\Lambda^e[T]^\flat \subseteq \Omega[T]^\flat$ and $g$ is the inclusion $\partial \Omega[S]^\flat \subseteq \Omega[S]^\flat$. Denoting the smash product of $f$ and $g$ by $f \wedge g$, we have that $f \wedge g = (u(f) \wedge u(g))^\flat$. In $\mathbf{dSets}$ inner anodynes are stable under smash products with normal monomorphisms (see \cite{dendroidalsets}). Hence $u(f) \wedge u(g)$ is inner anodyne and we conclude that $f \wedge g$ belongs to the weakly saturated class generated by (1). \par
\textbf{(1b)} Let $f$ be the same as in (1a) but set $g$ equal to $\Omega[C_k]^\flat \subseteq \Omega[C_k]^\sharp$. Then $f \wedge g$ is an isomorphism. Indeed, $u(f \wedge g)$ is clearly an isomorphism, and since $\Lambda^e[T]$ already contains all colours of $\Omega[T]$ the map $f \wedge g$ is surjective on marked corollas. \par
\textbf{(2a)} We assume $f$ is a morphism from the family (2) and $g$ equals $\partial \Omega[S]^\flat \subseteq \Omega[S]^\flat$. Checking that the smash product is marked anodyne is quite involved in this case; the combinatorics are deferred to Lemma \ref{lemma:cylindersubdivision2}. \par
\textbf{(2b)} Assume $f$ is a morphism from the class (2) and $g$ equals $\Omega[C_k]^\flat \subseteq \Omega[C_k]^\sharp$. Then $f \wedge g$ is an isomorphism by an argument similar to the one at (1b): the map $u(f \wedge g)$ is an isomorphism and all marked corollas of the tensor product $\Omega[T]^\flat \otimes \Omega[C_k]^\sharp$ are already contained in one of the summands of the pushout which is the domain of $f \wedge g$. \par
\textbf{(}$\mathbf{2^*}$\textbf{a)} Take $f$ as in ($2^*$) and let $g$ be  $\partial \Omega[S]^\flat \subseteq \Omega[S]^\flat$. Again, the combinatorics are somewhat involved and can be found in the proof of Lemma \ref{lemma:cylindersubdivision}. \par
\textbf{(}$\mathbf{2^*}$\textbf{b)} Let $f$ be the map
\begin{equation*}
\coprod_{c \in \mathrm{in}(C_n)} \eta_c \subseteq \Omega[C_n]^\sharp
\end{equation*}
and set $g$ equal to $\Omega[C_k]^\flat \subseteq \Omega[C_k]^\sharp$. It is again clear that the underlying map $u(f \wedge g)$ is an isomorphism. First assume $n, k > 0$. Denote the marked corollas of the domain of $f \wedge g$ by $\mathcal{E}_0$ and identify this domain with $(\Omega[C_n] \otimes \Omega[C_k], \mathcal{E}_0)$. If we denote the root of $C_n$ by $r$, the only corollas of $\Omega[C_n] \otimes \Omega[C_k]$ that are not yet in $\mathcal{E}_0$ are $r \otimes C_k$ and the corolla obtained by contracting all the inner edges of any one of the two shuffles of $\Omega[C_n] \otimes \Omega[C_k]$ (these two contractions are the same by the Boardman-Vogt relation). Denote the union $\mathcal{E}_0$ with this latter corolla by $\mathcal{E}_1$. By looking at the shuffle obtained by grafting copies of $C_k$ on the leaves of $C_n$, we see that the map
\begin{equation*}
(\Omega[C_n] \otimes \Omega[C_k], \mathcal{E}_0) \longrightarrow (\Omega[C_n] \otimes \Omega[C_k], \mathcal{E}_1)
\end{equation*}
can be obtained as a composition of pushouts of maps as in (3) and hence is marked anodyne. Now considering the shuffle where we graft a copy of $C_n$ onto each leaf of $C_k$, we see that the map
\begin{equation*}
(\Omega[C_n] \otimes \Omega[C_k], \mathcal{E}_1) \longrightarrow \Omega[C_n]^\sharp \otimes \Omega[C_k]^\sharp
\end{equation*}
can be obtained as a composition of pushouts of maps as in Corollary \ref{cor:markedanodyne3}. The cases where we allow $n$ and/or $k$ to be $0$ are easier and left to the reader. \par
\textbf{(3a)} Let $f$ be a map as in (3) and let $g$ be the inclusion $\partial \Omega[S]^\flat \subseteq \Omega[S]^\flat$ for some tree $S$. If $S$ has no vertices, the boundary $\partial \Omega[S]$ is empty and the map $f \wedge g$ is isomorphic to $f$ itself. If $S$ has at least one vertex the map $f \wedge g$ is an isomorphism; indeed, $u(f \wedge g)$ is an isomorphism and all marked corollas of $\Omega[T]^\sharp \otimes \Omega[S]^\flat$ are already contained in $\Omega[T]^\sharp \otimes \partial \Omega[S]^\flat$. \par
\textbf{(3b)} Let $f$ be the same, but take $g$ to be $\Omega[C_k]^\flat \subseteq \Omega[C_k]^\sharp$. In the domain of $f \wedge g$ all corollas which are obtained as a corolla of one of the factors tensored with a colour of the other are marked. Since all corollas of the tensor product can be obtained as `compositions' of such corollas, we see that $f \wedge g$ can be obtained as a composition of pushouts of maps of the form (3). \par
\textbf{(4a)} Let $K$ be a simplicial Kan complex and let $f$ be the inclusion $j_!(K^\flat) \subseteq j_!(K^\sharp)$. Set $g$ equal to $\partial \Omega[S]^\flat \subseteq \Omega[S]^\flat$. If $S$ has no vertices then $f \wedge g$ is isomorphic to $f$. Otherwise $f \wedge g$ is an isomorphism, just as in \textbf{(3a)}. \par
\textbf{(4b)} Let $f$ be as in \textbf{(4a)} and let $g$ be $\Omega[C_k]^\flat \subseteq \Omega[C_k]^\sharp$. By the same argument as in \textbf{(3b)} the smash product $f \wedge g$ can be obtained as a (possibly transfinite) composition of pushouts of maps of the form (3). $\Box$
\end{proof}

\begin{corollary}
\label{cor:exponential}
Let $X \longrightarrow S$ be a coCartesian fibration of dendroidal sets and let $K$ be a normal dendroidal set. Then the induced map
\begin{equation*}
\mathbf{Hom}_\mathbf{dSets}(K, X) \longrightarrow \mathbf{Hom}_\mathbf{dSets}(K, S)
\end{equation*}
is a coCartesian fibration and $\mathbf{Hom}_\mathbf{dSets^+}(K^\flat, X^\natural) = \mathbf{Hom}_\mathbf{dSets}(K,X)^\natural$.
\end{corollary}
\begin{proof}
The proof is a standard adjunction argument. Let $A \longrightarrow B$ be a marked anodyne in $\mathbf{dSets^+}$ and suppose we are given a commutative diagram
\[
\xymatrix{
A \ar[d]\ar[r] & \mathbf{Hom}_\mathbf{dSets^+}(K^\flat, X^\natural) \ar[d] \\
B \ar[r] & \mathbf{Hom}_\mathbf{dSets^+}(K^\flat, S^\sharp)
}
\]
By adjunction, finding a lift in this diagram is equivalent to finding a lift in
\[
\xymatrix{
A \otimes K^\flat \ar[d]\ar[r] & X^\natural \ar[d] \\
B \otimes K^\flat \ar[r] & S^\sharp
}
\]
The left vertical map is the smash product of $A \longrightarrow B$ with the cofibration $\emptyset \longrightarrow K^\flat$ and hence marked anodyne. By Corollary \ref{cor:liftwrtmarked} the lift exists. We conclude that the right vertical map in our original diagram has the right lifting property with respect to all marked anodynes. Note that we can identify $\mathbf{Hom}_\mathbf{dSets^+}(K^\flat, S^\sharp)$ with $\mathbf{Hom}_\mathbf{dSets}(K, S)^\sharp$. Applying Corollary \ref{cor:liftwrtmarked} again yields the conclusion. $\Box$
\end{proof}

From the proposition we also get the following corollary, which is interesting in its own right.

\begin{corollary}
\label{cor:pointwiseequiv}
Let $X$ be an $\infty$-operad and let $K$ be a normal dendroidal set. Then a 1-corolla of $\mathbf{Hom}_\mathbf{dSets}(K, X)$ is an equivalence if and only if for each colour of $K$ (i.e. each element of $K_\eta$) the induced 1-corolla of $X$ is an equivalence. In other words, equivalences in $\mathbf{Hom}_\mathbf{dSets}(K, X)$ are exactly the pointwise equivalences.
\end{corollary}
\begin{proof}
Let $\mathcal{E}$ denote the set of equivalences of $X$. Note that the underlying dendroidal set of $\mathbf{Hom}_\mathbf{dSets^+}(K^\flat, (X, \mathcal{E}))$ is simply $\mathbf{Hom}_\mathbf{dSets}(K, X)$. The marked corollas of $\mathbf{Hom}_\mathbf{dSets^+}(K^\flat, (X, \mathcal{E}))$ correspond precisely to the maps $K \otimes \Omega[C_1] \longrightarrow X$ for which the induced 1-corolla of $X$ for each colour of $K$ is an equivalence. We wish to show these are precisely the equivalences of the $\infty$-operad $\mathbf{Hom}_\mathbf{dSets}(K,X)$. \par
Denote by $*$ the terminal object of $\mathbf{dSets}$, which is (isomorphic to) $N_d(\mathbf{Comm})$. Let $p: X \longrightarrow *$ denote the unique projection.  Theorem A.7 from the appendix of \cite{moerdijkcisinski} tells us that the $p$-coCartesian 1-corollas of $X$ are precisely the equivalences of $X$. One sees that the map $(X, \mathcal{E}) \longrightarrow *^\flat$ satisfies conditions (a), (b) and (c) of Proposition \ref{prop:lifting} and hence has the right lifting property with respect to all marked anodynes. By the same adjunction argument as before the induced map
\begin{equation*}
\bar p: \mathbf{Hom}_\mathbf{dSets^+}(K^\flat, (X,\mathcal{E})) \longrightarrow \mathbf{Hom}_\mathbf{dSets^+}(K^\flat, *^\flat) = *^\flat
\end{equation*}
will also have the right lifting property with respect to all marked anodynes. Hence the marked 1-corollas of $\mathbf{Hom}_\mathbf{dSets^+}(K^\flat, X^\natural)$ have to be the $\bar p$-coCartesian 1-corollas, i.e. (applying the Theorem again) the equivalences of $\mathbf{Hom}_\mathbf{dSets^+}(K^\flat, X^\natural)$. Combining this with our earlier observation about the marked corollas of this internal hom yields the conclusion. $\Box$
\end{proof}

We also get the following corollary, which was already claimed in Proposition \ref{prop:mappingobjects}.

\begin{corollary}
Let $(X, \mathcal{E}_X) \longrightarrow S^\sharp$ be a map in $\mathbf{dSets^+}$ and let $Y \longrightarrow S$ be a coCartesian fibration. Suppose furthermore that the dendroidal set $X$ is normal. Then $\mathrm{Map}_S^\sharp(X,Y^\natural)$ is the maximal Kan complex contained in the $\infty$-category $\mathrm{Map}_S^\flat(X,Y^\natural)$.
\end{corollary}
\begin{proof}
We improve upon the proof of Proposition \ref{prop:mappingobjects} as follows. The structural map $X \longrightarrow S^\sharp$ induces a map $\eta \longrightarrow \mathbf{Hom}_{\mathbf{dSets^+}}(X, S^\sharp)$. Observe that the marked simplicial set $(\mathrm{Map}_S^\flat(X,Y^\natural),\mathrm{Map}_S^\sharp(X,Y^\natural)_1)$ fits into a pullback diagram
\[
\xymatrix{
j_!(\mathrm{Map}_S^\flat(X,Y^\natural),\mathrm{Map}_S^\sharp(X,Y^\natural)_1) \ar[r]\ar[d] & \mathbf{Hom}_{\mathbf{dSets^+}}(X, Y^\natural) \ar[d] \\
\eta \ar[r] & \mathbf{Hom}_{\mathbf{dSets^+}}(X, S^\sharp)
}
\]
We have just seen that the right vertical map has the right lifting property with respect to all marked anodynes, so the left vertical map will have this lifting property as well. By Corollary \ref{cor:liftwrtmarked} the map $\mathrm{Map}_S^\flat(X,Y^\natural) \longrightarrow \Delta^0$ is a coCartesian fibration and the set of coCartesian 1-simplices of $\mathrm{Map}_S^\flat(X,Y^\natural)$ is $\mathrm{Map}_S^\sharp(X,Y^\natural)_1$. We know that coCartesian 1-simplices over the point are exactly equivalences (again, one can use Theorem A.7 of \cite{moerdijkcisinski} or consult \cite{htt}). In other words, $\mathrm{Map}_S^\sharp(X,Y^\natural)$ consists of the simplices of $\mathrm{Map}_S^\flat(X,Y^\natural)$ all of whose edges are equivalences. This is precisely the maximal Kan complex contained in the latter simplicial set. $\Box$
\end{proof}

We conclude this section with the following lemma, which will be useful later.

\begin{lemma}
\label{lemma:Kanfibration}
Suppose $X, Y \in \mathbf{dSets^+}/S$ and let $f: X \longrightarrow Y$ be a cofibration in $\mathbf{dSets^+}/S$. Let $Z \longrightarrow S$ be a coCartesian fibration and let $f^*: \mathrm{Map}_S^\sharp(Y, Z^\natural) \longrightarrow \mathrm{Map}_S^\sharp(X, Z^\natural)$ be the induced map. Then $f^*$ is a Kan fibration.
\end{lemma}
\begin{proof}
Let $g: A \subseteq B$ be a left anodyne inclusion of simplicial sets. We want to show that we can solve lifting problems of the form
\[
\xymatrix{
A \ar[d]\ar[r] & \mathrm{Map}_S^\sharp(Y, Z^\natural) \ar[d]^{f^*} \\
B \ar[r]\ar@{-->}[ur] & \mathrm{Map}_S^\sharp(X, Z^\natural)
}
\]
Indeed, the right vertical map will then be a left fibration over a Kan complex and hence a Kan fibration. By adjunction, this lifting problem is equivalent to the following one:
\[
\xymatrix{
i_!(A)^\sharp \otimes Y \coprod_{i_!(A)^\sharp \otimes X} i_!(B)^\sharp \otimes X \ar[r]\ar[d] & Z^\natural \ar[d] \\
i_!(B)^\sharp \otimes Y \ar[r]\ar@{-->}[ur] & S^\sharp
}
\]
One easily verifies that $i_!(g)^\sharp$ is marked anodyne. The left vertical map is the smash product of $i_!(g)^\sharp$ with the cofibration $f$ and so is marked anodyne as well. Therefore the right vertical map has the right lifting property with respect to this map. $\Box$
\end{proof}

\subsection{coCartesian equivalences}

Given a dendroidal set $X$, a \emph{normalization} of $X$ is a trivial fibration
\begin{equation*}
p: X_{(n)} \longrightarrow X
\end{equation*}
in the Cisinski-Moerdijk model structure on $\mathbf{dSets}$ such that $X_{(n)}$ is a normal dendroidal set. Note that a trivial fibration is in particular a coCartesian fibration in which every corolla of the domain is coCartesian. Now, if $(X, \mathcal{E}_X)$ is a marked dendroidal set, we will refer to the marked dendroidal set
\begin{equation*}
(X_{(n)}, p^{-1}(\mathcal{E}_X))
\end{equation*}
as a \emph{marked normalization} of $X$. \par
We fix once and for all a normalization of the terminal object $\ast$, i.e. a factorization of the unique map $\emptyset \longrightarrow \ast$ into a cofibration followed by a trivial fibration, and denote the resulting dendroidal set by $E_\infty$:
\[
\xymatrix@W=15pt{
\emptyset \ar@{>->}[r] & E_\infty \ar@{->>}^\sim[r] & \ast 
}
\]
Since any dendroidal set which admits a map to a normal dendroidal set is itself normal \cite{dendroidalsets}, $(X,\mathcal{E}_X) \times E_\infty^\sharp$ will be a marked normalization of $(X, \mathcal{E}_X)$. When needed, we will use this explicit normalization.

\begin{definition}
A morphism $X \longrightarrow Y$ in $\mathbf{dSets^+}/S$ is said to be a \emph{marked equivalence} if for any coCartesian fibration $Z \longrightarrow S$ there exists a choice of marked normalizations of $X$ and $Y$ fitting into a commutative square
\[
\xymatrix{
X_{(n)} \ar[r]\ar[d] & Y_{(n)} \ar[d] \\
X \ar[r] & Y
}
\]
such that the induced map
\begin{equation*}
\mathrm{Map}_S^\flat(Y_{(n)},Z^\natural) \longrightarrow \mathrm{Map}_S^\flat(X_{(n)},Z^\natural)
\end{equation*}
is an equivalence of $\infty$-categories.
\end{definition}

\begin{remark}
Actually, the property of being a marked equivalence implies the following: for \emph{any} choice of marked normalizations fitting into a commutative square as above, the induced map
\begin{equation*}
\mathrm{Map}_S^\flat(Y_{(n)},Z^\natural) \longrightarrow \mathrm{Map}_S^\flat(X_{(n)},Z^\natural)
\end{equation*}
is an equivalence of $\infty$-categories. \par 
Also, it will sometimes be convenient to reduce to the case where $S$ is normal. There is an obvious functor $\mathbf{dSets^+}/S \longrightarrow \mathbf{dSets^+}/S \times E_\infty$ given by taking the product with $E_\infty^\sharp$. Then a map $X \longrightarrow Y$ in $\mathbf{dSets^+}/S$ is a marked equivalence if and only if its image under this functor is a marked equivalence in $\mathbf{dSets^+}/S \times E_\infty$.
\end{remark}

This section will be devoted to several useful results about marked equivalences.

\begin{proposition}
\label{prop:markedanodyne}
Marked anodyne morphisms between objects in $\mathbf{dSets^+}/S$ whose underlying dendroidal sets are normal are marked equivalences.
\end{proposition}
\begin{proof}
Suppose $X$ and $Y$ are marked dendroidal sets over $S^\sharp$. Let $f: X \longrightarrow Y$ be marked anodyne, suppose $u(X)$ and $u(Y)$ are normal and let $g: A \subseteq B$ be an inclusion of simplicial sets. Since the smash product $f \wedge j_!(g^\flat)$ is again marked anodyne, we deduce that for any coCartesian fibration $Z \longrightarrow S$ the induced map $\mathrm{Map}_S^\flat(Y,Z^\natural) \longrightarrow \mathrm{Map}_S^\flat(X,Z^\natural)$ has the right lifting property with respect to $g$. Hence it is a trivial fibration of simplicial sets. $\Box$
\end{proof}

Recall the following result (Lemma 3.1.3.2 from \cite{htt}):

\begin{lemma}
Let $f: C \longrightarrow D$ be a functor between $\infty$-categories. Then the following are equivalent:
\begin{itemize}
\item[(1)] The functor $f$ is an equivalence of $\infty$-categories
\item[(2)] For every simplicial set $K$, the induced functor $C^K \longrightarrow D^K$ induces a homotopy equivalence between the maximal Kan complexes contained in $C^K$ and $D^K$
\end{itemize}
\end{lemma}

This lemma allows us to prove the following result:

\begin{proposition}
\label{prop:equivKan}
Let $f: X \longrightarrow Y$ be a morphism in $\mathbf{dSets^+}/S$. Then the following are equivalent:
\begin{itemize}
\item[(1)] $f$ is a marked equivalence
\item[(2)] For every coCartesian fibration $Z \longrightarrow S$, the induced map
\begin{equation*}
\mathrm{Map}_S^\sharp(Y_{(n)}, Z^\natural) \longrightarrow \mathrm{Map}_S^\sharp(X_{(n)}, Z^\natural)
\end{equation*}
is a homotopy equivalence of Kan complexes
\end{itemize}
\end{proposition}
\begin{proof}
Since $\mathrm{Map}_S^\sharp(X_{(n)}, Z^\natural)$ is the maximal Kan complex contained in $\mathrm{Map}_S^\flat(X_{(n)}, Z^\natural)$, it is immediate that (1) implies (2). Now assume (2) and let $Z \longrightarrow S$ be an arbitrary coCartesian fibration. We need to show that
\begin{equation*}
\mathrm{Map}_S^\flat(Y_{(n)}, Z^\natural) \longrightarrow \mathrm{Map}_S^\flat(X_{(n)}, Z^\natural)
\end{equation*}
is an equivalence of $\infty$-categories. By the previous lemma it suffices to show that the induced map
\begin{equation*}
\mathrm{Map}_S^\flat(Y_{(n)}, Z^\natural)^K \longrightarrow \mathrm{Map}_S^\flat(X_{(n)}, Z^\natural)^K
\end{equation*}
induces a homotopy equivalence on maximal Kan complexes, for any simplicial set $K$. Now, by applying multiple adjunctions, we get
\begin{equation*}
\mathrm{Map}_S^\flat(X_{(n)}, Z^\natural)^K \simeq \mathrm{Map}_S^\flat(X_{(n)}, \mathbf{Hom_{dSets^+}}(i_!(K)^\flat, Z^\natural))
\end{equation*}
Applying Corollary \ref{cor:exponential} we find an isomorphism
\begin{equation*}
\mathrm{Map}_S^\flat(X_{(n)}, Z^\natural)^K \simeq \mathrm{Map}_S^\flat(X_{(n)}, (\mathbf{Hom_{dSets}}(i_!(K), Z))^\natural)
\end{equation*}
which is natural in $X_{(n)}$. By assumption $f_{(n)}$ induces a homotopy equivalence
\begin{equation*}
\mathrm{Map}_S^\sharp(Y_{(n)}, (\mathbf{Hom_{dSets}}(i_!(K), Z))^\natural) \longrightarrow \mathrm{Map}_S^\sharp(X_{(n)}, (\mathbf{Hom_{dSets}}(i_!(K), Z))^\natural)
\end{equation*}
between the maximal Kan complexes of the $\infty$-categories appearing on the right-hand side of this isomorphism. The result follows. $\Box$
\end{proof}

The following result provides further evidence of the relation between coCartesian fibrations and algebras and will be important in establishing the coCartesian model structure.

\begin{proposition}
\label{prop:markedequivalence}
Let $X \longrightarrow S$ and $Y \longrightarrow S$ be coCartesian fibrations of dendroidal sets. Suppose $f: X \longrightarrow Y$ is a map compatible with the maps to $S$ and mapping coCartesian corollas to coCartesian corollas. Then the following are equivalent:
\begin{itemize}
\item[(1)] The map $f$ induces a marked equivalence $X^\natural \longrightarrow Y^\natural$ in $\mathbf{dSets^+}/S$
\item[(2)] The map $f$ is an operadic equivalence
\item[(3)] The map $f$ induces a categorical equivalence $X_s \longrightarrow Y_s$ for each colour $s$ of $S$
\end{itemize}
\end{proposition}

Before proving this proposition, we will need to introduce some new concepts. Suppose we are given a dendroidal set $X$ and colours $x_1, \ldots, x_k, x$ of $X$. We would like to have suitable model for `the space of operations' from $\{x_1, \ldots, x_k\}$ to $x$. One such model can of course be obtained by considering the simplicial operad $\mathrm{hc}\tau_d(X)$. However, we will introduce another model which is more convenient for our purposes. \par 
Recall from section \ref{subsection:coCartcorollas} the maps
\begin{equation*}
\iota_n: [n] \longrightarrow C_k \star [n-1]
\end{equation*}
and the simplicial set $Q$ defined by
\begin{equation*}
Q: [n] \mapsto \mathbf{dSets}(\Omega[C_k \star [n-1]], X)
\end{equation*}
Now let $X^L(x_1, \ldots, x_k; x)$ be the simplicial subset of $Q$ consisting of maps $f$ satisfying the following conditions:
\begin{itemize}
\item[(1)] The inputs of $f(C_k)$ are $x_1, \ldots, x_k$
\item[(2)] The simplex $f(\iota_n([n]))$ is contained in $\eta_x \subseteq X$, i.e. $f(\iota_n([n]))$ is a degenerate simplex at $x$
\end{itemize}
We first need to check that these simplicial sets behave reasonably under maps of dendroidal sets $X \longrightarrow Y$.

\begin{proposition}
Let $f: X \longrightarrow Y$ be an inner fibration of dendroidal sets and let $x_1, \ldots, x_k, x$ be colours of $X$. Then the induced map
\begin{equation*}
X^L(x_1, \ldots, x_k; x) \longrightarrow Y^L(f(x_1), \ldots, f(x_k); f(x))
\end{equation*} 
is a left fibration.
\end{proposition}
\begin{proof}
Let $0 \leq i < n$. One immediately sees that under the map $\iota_n$ the face $\partial_i\Delta^n$ corresponds to an inner face of $C_k \star [n-1]$, given by contracting the edge $\iota_n(i)$. Suppose we are given a diagram
\[
\xymatrix{
\Lambda^n_i \ar[r]\ar[d] & X^L(x_1, \ldots, x_k; x) \ar[d] \\
\Delta^n \ar[r] & Y^L(f(x_1), \ldots, f(x_k); f(x))
}
\]
Finding a lift in this diagram is equivalent to finding a lift in the corresponding diagram
\[
\xymatrix{
\Lambda^{\iota_n(i)}[C_k \star [n-1]] \ar[r]\ar[d] & X \ar[d]\\
\Omega[[C_k \star [n-1]] \ar[r] & Y
}
\]
This lift exists since $f$ is an inner fibration. $\Box$
\end{proof}

\begin{corollary}
If we assume $X$ to be an $\infty$-operad, then $X^L(x_1, \ldots, x_k; x)$ is a Kan complex. 
\end{corollary}
\begin{proof}
Denote the terminal object of $\mathbf{dSets}$ by $*$. All of its mapping spaces are isomorphic to the point $\Delta^0$. The map $X \longrightarrow *$ is an inner fibration, so 
\begin{equation*}
X^L(x_1, \ldots, x_k; x) \longrightarrow \Delta^0
\end{equation*}
is a left fibration. This is equivalent to this map being a Kan fibration (this holds for left fibrations over any Kan complex). The assertion follows. $\Box$ 
\end{proof}

\begin{corollary}
If $f: X \longrightarrow Y$ is an inner fibration of $\infty$-operads, the induced map 
\begin{equation*}
X^L(x_1, \ldots, x_k; x) \longrightarrow Y^L(f(x_1), \ldots, f(x_k); f(x))
\end{equation*} 
is a Kan fibration.
\end{corollary}
\begin{proof}
The simplicial set $Y^L(f(x_1), \ldots, f(x_k); f(x))$ is a Kan complex and the stated map is a left fibration. Again, this implies it is in fact a Kan fibration. $\Box$
\end{proof}

Of course, the most important property of our spaces of operations is the following.

\begin{proposition}
\label{prop:mappingspaceequiv}
A map of $\infty$-operads $f: X \longrightarrow Y$ is an operadic equivalence if and only if it is essentially surjective and the induced maps
\begin{equation*}
X^L(x_1, \ldots, x_k; x) \longrightarrow Y^L(f(x_1), \ldots, f(x_k); f(x))
\end{equation*}
are homotopy equivalences of simplicial sets for any tuple $(x_1, \ldots, x_k, x)$ of colours of $X$.
\end{proposition}
\begin{proof}
The proof is an adaptation of Theorem 3.11 of \cite{moerdijkcisinski2}. $\Box$
\end{proof}

We now start working towards the proof of Proposition \ref{prop:markedequivalence}.

\begin{proposition}
\label{prop:fibersequencemappingspace}
Suppose $f: X \longrightarrow Y$ is an inner fibration of $\infty$-operads and suppose we are given colours $x_1, \ldots, x_k, x$ of $X$. Furthermore, suppose we have a corolla $\sigma$ of $Y$ whose leaves are $f(x_1), \ldots, f(x_k)$ and whose root is $f(x)$. If there exists an $f$-coCartesian corolla $\tilde\sigma$ of $X$ whose leaves are $x_1, \ldots, x_k$ such that $f(\tilde\sigma) = \sigma$, then there is a homotopy fiber sequence
\begin{equation*}
X_{f(x)}^L(\tilde x;x) \longrightarrow X^L(x_1, \ldots, x_k; x) \longrightarrow Y^L(f(x_1), \ldots, f(x_k); f(x))
\end{equation*}
where $\tilde x$ denotes the root of $\tilde \sigma$ and $X_{f(x)}$ is the fiber of $f$ over $f(x)$.
\end{proposition}
\begin{proof}
Let $\phi$ denote the map
\begin{equation*}
X^L(x_1, \ldots, x_k; x) \longrightarrow Y^L(f(x_1), \ldots, f(x_k); f(x))
\end{equation*}
We already know $\phi$ is a Kan fibration, so it suffices to compute the homotopy type of the fiber $\phi^{-1}(\sigma)$. Now let us for the purposes of this proof define the `mapping space' $X^L(\tilde \sigma; x)$. This is the simplicial set whose $n$-simplices are maps
\begin{equation*}
\Omega[C_k \star [n]] \longrightarrow X
\end{equation*}
whose restrictions to $C_k$ and $[n]$ are $\tilde \sigma$ and the degenerate $n$-simplex at $x$ respectively. We denote by
\begin{equation*}
\partial_{\tilde \sigma}: X^L(\tilde \sigma; x) \longrightarrow X^L(\tilde x; x)
\end{equation*}
the map induced by chopping of the inputs of $\tilde \sigma$. We easily verify that this is in fact a trivial Kan fibration. Indeed, the data of a diagram
\[
\xymatrix{
\partial \Delta^n \ar[d]\ar[r] & X^L(\tilde \sigma; x) \ar[d] \\
\Delta^n \ar[r] & X^L(\tilde x; x)
}
\]
correspond precisely to a map
\begin{equation*}
\Lambda^{\tilde x}[C_k \star [n]] \longrightarrow X
\end{equation*}
which can be extended to $\Omega[C_k \star [n]]$ since $X$ is an $\infty$-operad. This extension also gives us the lift we need in the previous diagram. \par 
For simplicity of notation, we will write
\begin{equation*}
Z := X^L(\tilde \sigma; x) \times_{X^L(\tilde x; x)} X^L_{f(x)}(\tilde x; x)
\end{equation*} 
Suppose we are given an $n$-simplex of $Z$, corresponding to a $C_k \star [n]$-dendrex of $X$. By taking the inner face of this dendrex obtained by contracting the inner edge $\tilde x$, we obtain a map
\begin{equation*}
\partial_{\tilde x}: Z \longrightarrow \phi^{-1}(\sigma) 
\end{equation*}
We claim this map is also a trivial Kan fibration, which completes the proof. Indeed, consider a lifting problem
\[
\xymatrix{
\partial \Delta^n \ar[d]\ar[r] & Z \ar[d] \\
\Delta^n \ar[r]\ar@{-->}[ur] & \phi^{-1}(\sigma)
}
\]
Denoting the vertex of $C_k$ by $v$, this lifting problem is equivalent to the lifting problem
\[
\xymatrix{
\Omega[C_k]\ar[d]_v\ar[dr]^{\tilde \sigma} & \\
\Lambda^v[C_k \star [n]] \ar[r]\ar[d] & X \ar[d]^f\\
\Omega[C_k \star [n]] \ar@{-->}[ur]\ar[r] & Y
}
\]
This problem can be solved by the assumption that $\tilde \sigma$ is $f$-coCartesian. $\Box$
\end{proof}

\begin{proposition}
Suppose we are given a diagram
\[
\xymatrix{
X \ar[rr]^f\ar[dr]_{p\circ f} & & Y \ar[dl]^p \\
& S &
}
\]
where $X$, $Y$ and $S$ are $\infty$-operads and both $p$ and $p \circ f$ are coCartesian fibrations. Suppose that $f$ maps $p\circ f$-coCartesian corollas of $X$ to $p$-coCartesian corollas of $Y$. If for every colour $s$ of $S$ the induced map $X_s \longrightarrow Y_s$ is a categorical equivalence, then $f$ is a weak equivalence of $\infty$-operads.
\end{proposition}
\begin{proof}
Let $x_1, \ldots, x_n, x$ be colours of $X$ and let $\sigma$ be a corolla of $S$ with inputs $p(x_1), \ldots, p(x_n)$ and output $p(x)$. Let $\tilde \sigma$ be a $p \circ f$-coCartesian lift of $\sigma$ with inputs $x_1, \ldots, x_n$. We adopt the notations
\begin{eqnarray*}
Z_X & := & X^L(\tilde \sigma; x) \times_{X^L(\tilde x; x)} X^L_{(p \circ f)(x)}(\tilde x; x) \\
Z_Y & := & Y^L(f(\tilde \sigma); f(x)) \times_{Y^L(f(\tilde x); f(x))} Y^L_{p(x)}(f(\tilde x); f(x))
\end{eqnarray*}
where the right-hand sides are defined as in the previous proof. Recall that we have homotopy equivalences
\begin{eqnarray*}
Z_X & \longrightarrow & X^L_{(p \circ f)(x)}(\tilde x; x) \\
Z_Y & \longrightarrow & Y^L_{p(x)}(f(\tilde x); f(x))
\end{eqnarray*}
Using the fact that $f$ preserves coCartesian corollas we obtain a diagram
\[
\xymatrix{
Z_X \ar[r]\ar[d] & X^L(x_1, \ldots, x_k; x) \ar[r]\ar[d] & S^L(p(x_1), \ldots, p(x_k); p(x)) \ar@{=}[d] \\
Z_Y \ar[r] & Y^L(f(x_1), \ldots, f(x_k); f(x)) \ar[r] & S^L(p(x_1), \ldots, p(x_k); p(x))
}
\]
in which the rows are homotopy fiber sequences. The left vertical map is a homotopy equivalence by our assumption and the homotopy equivalences described above. Hence the middle vertical map is also a homotopy equivalence (this is basically an application of the five lemma). Thus $f$ is full and faithful. It is also essentially surjective by the assumption that it is fiberwise a categorical equivalence. Applying Proposition \ref{prop:mappingspaceequiv} it is thus an operadic equivalence. $\Box$
\end{proof}

For the purposes of the following proof, we introduce the notion of a \emph{strong homotopy}. Given two maps $f,\, g: X \longrightarrow Y$ of marked dendroidal sets over $S^\sharp$, we will say that they are strongly homotopic if there exists a lift indicated by the dotted arrow in the following diagram:
\[
\xymatrix{
X \coprod X \ar[d]\ar[r]^{f \coprod g} & Y \ar[d] \\
X \otimes i_!(\Delta^1)^\sharp \ar@{-->}[ur] \ar[r] & S^\sharp
}
\]
Here the left vertical map is induced by the two inclusions $\{0\} \rightarrow \Delta^1$ and $\{1\} \rightarrow \Delta^1$. Note that in the particular case that $X$ is normal, the map $u(Y) \longrightarrow S$ is a coCartesian fibration and $Y = u(Y)^\natural$, the maps $f$ and $g$ are strongly homotopic if and only if they are equivalent when viewed as objects of the $\infty$-category $\mathrm{Map}^\flat_S(X, Y)$.

\begin{proof}[Proof of Proposition \ref{prop:markedequivalence}]
We have just proven the equivalence between (2) and (3). For the rest of this proof, we are free to replace $X^\natural$, $Y^\natural$ and $S^\sharp$ by $X^\natural \times E_\infty^\sharp$, $Y^\natural \times E_\infty^\sharp$ and $(S \times E_\infty)^\sharp$ respectively. In particular, we are allowed to assume that $X$, $Y$ and $S$ are normal. \par 
By a standard argument, condition (1) of the proposition is equivalent to the existence of a strong homotopy inverse $Y \longrightarrow X$. One immediately sees that the existence of such a homotopy inverse implies (3). Thus, it suffices to prove that (2) implies (1). Our goal is to find a strong homotopy inverse to $f$. We will show how to obtain a strong left homotopy inverse to $f$; an analogous construction shows that $f$ will also admit a strong right homotopy inverse, completing the proof. \par 
Since $S$ is normal, it admits a normal skeletal filtration (see \cite{dendroidalsets}), i.e. $S$ can be obtained by a (possibly transfinite) composition of maps each of which is a pushout of a boundary inclusion $\partial \Omega[T] \longrightarrow \Omega[T]$ for some tree $T$. In other words, we can build $S$ by adjoining all its dendrices one by one (note the analogy with the usual method of building a simplicial set simplex by simplex). \par 
We will now construct a strong left homotopy inverse (say $g$) to $f$ by induction on a fixed normal skeletal filtration. That is, suppose $S_\alpha$ is a stage of the filtration for some ordinal $\alpha$ and suppose we have already constructed a map
\begin{equation*}
g_\alpha: Y^\natural \times_{S^\sharp} S_\alpha^\sharp \longrightarrow X^\natural \times_{S^\sharp} S_\alpha^\sharp
\end{equation*}
and a strong homotopy
\begin{equation*}
h_\alpha: (X^\natural \times_{S^\sharp} S_\alpha^\sharp) \otimes i_!(\Delta^1)^\sharp \longrightarrow X^\natural \times_{S^\sharp} S_\alpha^\sharp
\end{equation*}
from $g_\alpha \circ f$ to the identity. We wish to construct suitable maps $g_{\alpha+1}$ and $h_{\alpha+1}$, where the latter is a strong homotopy from $g_{\alpha+1} \circ f$ to the identity. Using this trick, we may reduce to the case where $S_{\alpha+1} = \Omega[T]$ and $S_\alpha = \partial \Omega[T]$. \par 
For convenience of notation, define
\begin{eqnarray*}
X_1 & := & X^\natural \times_{S^\sharp} \Omega[T]^\sharp \\
Y_1 & := & Y^\natural \times_{S^\sharp} \Omega[T]^\sharp \\
X_0 & := & X_1 \times_{\Omega[T]^\sharp} \partial \Omega[T]^\sharp \\
Y_0 & := & Y_1 \times_{\Omega[T]^\sharp} \partial \Omega[T]^\sharp \\
Z_0 & := & (X_1 \otimes \{0\}) \coprod_{X_0 \otimes \{0\}} (X_0 \otimes i_!(\Delta^1)^\sharp) \coprod_{X_0 \otimes \{1\}} (Y_0 \otimes \{1\}) \\
Z_1 & := & (X_1 \otimes i_!(\Delta^1)^\sharp) \coprod_{X_1 \otimes \{1\}} (Y_1 \otimes \{1\})
\end{eqnarray*}
Let us reformulate the problem. The data we are given defines a map
\begin{equation*}
\phi_0: Z_0 \longrightarrow X_1 
\end{equation*} 
and our goal is now to find a lift in the following diagram:
\[
\xymatrix{
Z_0 \ar[d]\ar[r] & X_1 \ar[d] \\
Z_1 \ar@{-->}[ur]\ar[r] & \Omega[T]^\sharp
}
\]
Note that if a lift in this diagram exists on the level of underlying dendroidal sets, the only thing we need to check in order for this map to respect the markings is that $\{x\} \otimes i_!(\Delta^1)$ is mapped to a marked 1-corolla of $X_1$ for each colour $x$ of $X_1$. Also, this condition is automatically satisfied if $\partial \Omega[T] \neq \emptyset$, i.e. if $T$ has at least one vertex. \par 
If $T = \eta$, we will have $X_0 = Y_0 = \emptyset$, $Z_0 = X_1 \otimes \{0\}$ and both $X_1$ and $Y_1$ will be $\infty$-categories in which the marked 1-simplices are exactly the equivalences. The lifting problem above now corresponds to finding a homotopy inverse $Y_1 \longrightarrow X_1$. This exists since $X_1 \longrightarrow Y_1$ is an equivalence of $\infty$-categories. \par 
Now consider a tree $T$ with at least one vertex. Any extension $Z_1 \longrightarrow X_1$ of $Z_0 \longrightarrow X_1$ will automatically be compatible with the maps to $\Omega[T]^\sharp$. Indeed, any map to $\Omega[T]^\sharp$ is completely determined by what it does on colours, and all the colours of $T$ are already contained in $\partial \Omega[T]$. By the comments made above, it suffices to construct this map on the level of underlying dendroidal sets. Since $X_1$ maps to $\Omega[T]^\sharp$ by an inner fibration, its underlying dendroidal set is an $\infty$-operad. Observing that the underlying map of $Z_0 \longrightarrow Z_1$ is a normal monomorphism between normal objects, we can apply Lemma \ref{lemma:extensionuptohomotopy} to deduce that we only need to construct the desired lift up to homotopy. Hence, we may replace the inclusion $u(Z_0 \longrightarrow Z_1)$ by the weakly equivalent inclusion $u(k)$ where $k$ is the map 
\begin{equation*}
(X_1 \otimes \{0\}) \coprod_{X_0 \otimes \{0\}} (X_0 \otimes i_!(\Delta^1)^\sharp) \longrightarrow X_1 \otimes i_!(\Delta^1)^\sharp
\end{equation*}
We observe that $k$ is the smash product of the cofibration $X_0 \subseteq X_1$ with the marked anodyne $\{0\} \subseteq i_!(\Delta^1)^\sharp$ and is therefore marked anodyne. Since $X_1 \longrightarrow \Omega[T]^\sharp$ has the right lifting property with respect to all marked anodynes the desired lift exists and we are done. $\Box$
\end{proof}

\subsection{The model structure}

The following result will establish a model structure on $\mathbf{dSets^+}/S$ which we will refer to as the \emph{coCartesian model structure}.

\begin{theorem}
\label{thm:coCartmodelstructure}
There exists a left proper combinatorial model structure on $\mathbf{dSets^+}/S$ in which the cofibrations are the maps whose underlying maps of dendroidal sets are cofibrations and the weak equivalences are the marked equivalences.
\end{theorem}
\begin{proof}
We will use Jeff Smith's machinery for combinatorial model categories, which is summarized in the Appendix. Using Proposition \ref{prop:combinmodelcat}, we reduce to checking the following things:
\begin{itemize}
\item[(1)] The class $W$ of weak equivalences is perfect
\item[(2)] Weak equivalences are stable under pushouts by cofibrations
\item[(3)] Any map having the right lifting property with respect to all cofibrations is a weak equivalence
\end{itemize}
We start by proving (1). Suppose $X \in \mathbf{dSets^+}/S$. Applying Quillen's small object argument, we can factor the map $X \longrightarrow S^\sharp$ into a marked anodyne followed by a map having the right lifting property with respect to all marked anodynes, i.e. we get a factorization
\begin{equation*}
X \longrightarrow Z_X^\natural \longrightarrow S^\sharp 
\end{equation*}
where $Z_X \longrightarrow S$ is a coCartesian fibration. Furthermore, this factorization can be constructed in a way that is functorial in $X$ and preserves filtered colimits. \par 
Now suppose we are given a morphism $f: X \longrightarrow Y$ in $\mathbf{dSets^+}/S$. We fix the normalizations $X_{(n)} =  X \times E_\infty^\sharp$ and $Y_{(n)} =  Y \times E_\infty^\sharp$. The map $f$ is a marked equivalence if and only if the induced map $f_{(n)}: X_{(n)} \longrightarrow Y_{(n)}$ is a marked equivalence. Now the factorization described above yields the following square:
\[
\xymatrix{
X_{(n)} \ar[r]^{f_{(n)}}\ar[d] & Y_{(n)} \ar[d] \\
Z_{X_{(n)}}^\natural \ar[r] & Z_{Y_{(n)}}^\natural
}
\]
Since marked anodynes are normal, both $Z_{X_{(n)}}$ and $Z_{Y_{(n)}}$ will be normal. Proposition \ref{prop:markedanodyne} implies that both vertical morphisms are marked equivalences. By applying the 2-out-of-3 property of marked equivalences to the diagram we see that $f_{(n)}$ is a marked equivalence if and only if the bottom arrow is a marked equivalence. By \ref{prop:markedequivalence} this is the case precisely if the underlying map $Z_{X_{(n)}} \longrightarrow Z_{Y_{(n)}}$ is an equivalence in the Cisinski-Moerdijk model structure on $\mathbf{dSets}$. Since products commute with filtered colimits our normalization will preserve filtered colimits and so does our factorization, as remarked above. We can now apply Lemma \ref{lemma:perfectclass} from the Appendix and the fact that the class of weak equivalences in the Cisinski-Moerdijk model structure is perfect to conclude that our class $W$ of marked equivalences is perfect. \par 
(2) Let $f: X \longrightarrow Y$ be a marked equivalence and let $g: X \longrightarrow V$ be a cofibration. Consider the pushout diagram
\[
\xymatrix{
X_{(n)} \ar[r]\ar[d] & Y_{(n)} \ar[d] \\
V_{(n)} \ar[r] & (V \coprod_X Y)_{(n)}
}
\]
For an arbitrary coCartesian fibration $Z \longrightarrow S$ this will give us a pullback square
\[
\xymatrix{
\mathrm{Map}^\sharp_S((V \coprod_X Y)_{(n)}, Z^\natural) \ar[r]\ar[d] & \mathrm{Map}^\sharp_S(Y_{(n)}, Z^\natural) \ar[d] \\
\mathrm{Map}^\sharp_S(V_{(n)}, Z^\natural) \ar[r] & \mathrm{Map}^\sharp_S(X_{(n)}, Z^\natural)
}
\]
The right vertical map is a homotopy equivalence by the assumption that $f$ is a marked equivalence and the bottom map is a Kan fibration by Lemma \ref{lemma:Kanfibration}. Since the Quillen model structure on $\mathbf{sSets}$ is right proper, the right vertical map will be a homotopy equivalence as well. By Proposition \ref{prop:equivKan} we conclude that $V \longrightarrow V \coprod_X Y$ is a marked equivalence. \par 
(3) Suppose we are given a morphism $f: X \longrightarrow Y$ in $\mathbf{dSets^+}/S$ which has the right lifting property with respect to all cofibrations. Again setting $X_{(n)} =  X \times E_\infty^\sharp$ and $Y_{(n)} =  Y \times E_\infty^\sharp$ we see that the induced map $f_{(n)}: X_{(n)} \longrightarrow Y_{(n)}$ also has the right lifting property with respect to all cofibrations. This means that the underlying map of dendroidal sets is a trivial fibration (in the Cisinski-Moerdijk model structure) and that a corolla of $X_{(n)}$ is marked if and only if its image under $f_{(n)}$ is marked. We obtain a map $g_{(n)}: Y_{(n)} \longrightarrow X_{(n)}$ by lifting in the following diagram of underlying dendroidal sets:
\[
\xymatrix{
\emptyset \ar[d]\ar[r] & u(X_{(n)}) \ar@{->>}[d] \\
u(Y_{(n)}) \ar@{-->}^{g_{(n)}}[ur] \ar@{=}[r] & u(Y_{(n)})
}
\]
This lift is immediately seen to respect markings. To show that $g_{(n)}$ is a strong homotopy inverse to $f_{(n)}$, we take a lift as in the diagram below:
\[
\xymatrix@C=90pt{
u(X_{(n)}) \coprod u(X_{(n)}) \ar[r]^{\mathrm{id}_{X_{(n)}} \coprod g_{(n)} \circ f_{(n)}} \ar[d] & u(X_{(n)}) \ar@{->>}[d] \\
u(X_{(n)}) \otimes i_!(\Delta^1) \ar@{-->}[ur]\ar[r] & u(Y_{(n)})
}
\]
Again one easily verifies that this lift respects markings in such a way that we obtain a map
\begin{equation*}
X_{(n)} \otimes i_!(\Delta^1)^\sharp \longrightarrow X_{(n)}
\end{equation*}
This shows $g_{(n)}$ is indeed strong homotopy inverse to $f_{(n)}$, from which we conclude that $f_{(n)}$ and hence also $f$ is a marked equivalence. $\Box$
\end{proof}

\begin{remark}
Note that we get the following improvement of Lemma \ref{prop:markedanodyne}: \emph{every} marked anodyne morphism is a trivial cofibration.
\end{remark}

This model structure enjoys several pleasant properties, which are stated in the following propositions.

\begin{proposition}
\label{prop:fibrantobjects}
The fibrant objects of $\mathbf{dSets^+}/S$ are precisely the objects isomorphic to some $Z^\natural \longrightarrow S^\sharp$, where $Z \longrightarrow S$ is a coCartesian fibration.
\end{proposition}
\begin{proof}
Suppose $X \longrightarrow S^\sharp$ is fibrant. Then this map has the right lifting property with respect to trivial cofibrations, in particular with respect to marked anodynes. Proposition \ref{prop:lifting} implies that $u(X) \longrightarrow S$ is a coCartesian fibration and that $X = u(X)^\natural$. \par 
Now suppose we are given a coCartesian fibration $f: Z \longrightarrow S$. First, let $g: A \longrightarrow B$ be a trivial cofibration between normal objects. We aim to find a lift $\phi$ in the diagram
\[
\xymatrix{
A \ar[d]_g \ar[r] & Z^\natural \ar[d] \\ 
B \ar[r]\ar@{-->}[ur]^\phi & S^\sharp
}
\] 
Since $g$ is trivial, the map
\begin{equation*}
\mathrm{Map}^\sharp_S(B, Z^\natural) \longrightarrow \mathrm{Map}^\sharp_S(A, Z^\natural)
\end{equation*} 
is a trivial Kan fibration. In particular it is surjective on vertices. It follows that there exists a map $\phi$ serving as a lift in our diagram. \par 
The only thing left to show is that the class of trivial cofibrations is in fact the weak saturation of the class of trivial cofibrations between normal objects. This is done in the proof of Proposition 8.4.2 in \cite{dendroidalsets}. We include the argument here for completeness. Suppose $g: A \longrightarrow B$ is a trivial cofibration in $\mathbf{dSets^+}/S$ (we leave the maps to $S^\sharp$ implicit throughout this argument), where $A$ and $B$ are not required to be normal. Take a normalization $B'$ of $B$ and construct the pullback
\[
\xymatrix{
A' \ar[r]^{g'}\ar[d] & B' \ar[d] \\
A \ar[r]_g & B
}
\]
so that $A'$ is also a normalization of $A$. Take a pushout to obtain a diagram
\[
\xymatrix{
A' \ar[r]^{g'}\ar[d] & B' \ar[d]\ar[ddr] & \\
A \ar[r]^v\ar[drr]_g & C \ar[dr]_w & \\
& & B
}
\]
in which the square is actually still a pullback. Now suppose $w$ has the right lifting property with respect to cofibrations. Then we can find a section $s: B \longrightarrow C$ such that $w \circ s = \mathrm{id}_B$. This exhibits $g$ as a retract of $v$:
\[
\xymatrix{
A \ar@{=}[r]\ar[d]_g & A \ar@{=}[r]\ar[d]_v & A\ar[d]_g \\
B \ar[r]_s & C \ar[r]_w & B
}
\]
Since $v$ is a pushout of $g'$, a trivial cofibration between normal objects, the result follows. \par
It remains to show that $w$ has the desired property. Consider a cofibration $U \longrightarrow V$. Since we may assume this cofibration is either the inclusion $(\partial \Omega[T])^\flat \subseteq (\Omega[T])^\flat$ or the inclusion $(\Omega[C_1])^\flat \subseteq (\Omega[C_1])^\sharp$ we can in particular assume $U$ is normal. Consider a diagram of the form
\[
\xymatrix{
& B' \ar[d] \\
U \ar[d]\ar[r] & C \ar[d]^w \\
V \ar@{-->}[ur]\ar[r] & B
}
\]
We are supposed to find the dotted lift indicated. Now suppose we can find a lift of $U \longrightarrow C$ to a map $U \longrightarrow B'$. Using the explicit normalization $B' = B \times E_\infty^\sharp$ one sees that the map $B' \longrightarrow B$ has the right lifting property with respect to cofibrations, so we can find a map $V \longrightarrow B'$ rendering the diagram commutative. Composing with $B' \longrightarrow C$ does the trick. \par
We need to construct the lift $U \longrightarrow B'$. For this, we pull back the pushout square we had along $U \longrightarrow C$ to form the cube
\[
\xymatrix@R=20pt@C=20pt{
                   & Q' \ar[dl]\ar'[d][dd]\ar[rr] &             & U' \ar[dl]\ar[dd] \\
A' \ar[dd]\ar[rr] &                              & B' \ar[dd] &                   \\
                   & Q  \ar[dl]\ar'[r][rr]        &             & U  \ar[dl]        \\
A   \ar[rr]        &                              & C           &
}
\]
Then all the faces of this cube are pullbacks. Hence $Q'$ is a normalization of $Q$. Also $Q$ is normal since it admits a map to $U$. Now $Q' \longrightarrow Q$ is a trivial fibration of cofibrant objects and thus admits a section. Since the front face of the cube is a pushout, so is the back face (use that $\mathbf{dSets^+}$ is a quasi-topos, so that pullbacks commute with pushouts). Hence the pushout $U' \longrightarrow U$ of $Q' \longrightarrow Q$ also admits a section. The composition
\begin{equation*}
U \longrightarrow U' \longrightarrow B'
\end{equation*}
is the required lift. $\Box$
\end{proof}

Before proving the next proposition, we need the following:

\begin{lemma}
Let $S$ and $T$ be dendroidal sets and let $Z$ be a normal object of $\mathbf{dSets^+}/T$. Then the functor
\begin{equation*}
\mathbf{dSets^+}/S \longrightarrow \mathbf{dSets^+}/(S \otimes T): X \longmapsto X \otimes Z
\end{equation*}
preserves marked equivalences.
\end{lemma}
\begin{proof}
Let $f: X \longrightarrow Y$ be a marked equivalence in $\mathbf{dSets^+}/S$. We wish to show that $f \otimes \mathrm{id}_Z$ is a marked equivalence in $\mathbf{dSets^+}/(S \otimes T)$. Pick a marked anodyne $X \longrightarrow X'$ such that $X'$ is fibrant in $\mathbf{dSets^+}/S$ and subsequently pick a marked anodyne $X' \coprod_X Y \longrightarrow Y'$ such that $Y' \in \mathbf{dSets^+}/S$ is fibrant. The tensor products $X \otimes Z \longrightarrow X' \otimes Z$ and $(X' \coprod_X Y) \otimes Z \longrightarrow Y' \otimes Z$ are marked anodyne by Proposition \ref{prop:smashproduct}, using the fact that $Z$ is normal. Considering the diagram
\[
\xymatrix@C=50pt{
X \otimes Z \ar[r]^{\sim}\ar[d] & X' \otimes Z \ar[d]\ar[dr] & \\
Y \otimes Z \ar[r]^{\sim} & (X' \coprod_X Y) \otimes Z \ar[r]^{\sim} & Y' \otimes Z
}
\] 
we see that $X \otimes Z \longrightarrow Y \otimes Z$ is a marked equivalence if and only if $X' \otimes Z \longrightarrow Y' \otimes Z$ is so. The induced map on normalizations $f'_{(n)}: X'_{(n)} \longrightarrow Y'_{(n)}$ is a marked equivalence of fibrant-cofibrant objects and thus admits a strong homotopy inverse $g'_{(n)}$. The map $g'_{(n)} \otimes \mathrm{id}_Z$ is then a strong homotopy inverse to $f'_{(n)} \otimes \mathrm{id}_Z$, completing the proof. $\Box$ 
\end{proof}

\begin{corollary}
Let $f: A \longrightarrow B$ and $g: C \longrightarrow D$ be cofibrations in respectively $\mathbf{dSets^+}/S$ and $\mathbf{dSets^+}/T$. Then the smash product $f \wedge g$ is a cofibration in $\mathbf{dSets^+}/(S \otimes T)$ which is trivial if either $f$ or $g$ is.
\end{corollary}
\begin{proof}
A map in $\mathbf{dSets^+}/(S \otimes T)$ is a cofibration if and only if the underlying map of dendroidal sets is, so the first statement follows from the corresponding statement for dendroidal sets \cite{dendroidalsets}. Now assume $f$ is trivial (the case where $g$ is trivial follows by symmetry). Since trivial cofibrations between normal objects generate all the trivial cofibrations as a weakly saturated class (cf. the proof of Proposition \ref{prop:fibrantobjects}) we may assume that the domain and codomain of both $f$ and $g$ are normal. Considering the diagram
\[
\xymatrix@C=50pt{
A \otimes C \ar[r]\ar[d]^\sim & A \otimes D \ar[d]^{\sim}\ar[dr]^{\sim} & \\ 
B \otimes C \ar[r] & B \otimes C \coprod_{A \otimes C} A \otimes D \ar[r]^{f \wedge g} & B \otimes D 
}
\]
we see that $f \wedge g$ is trivial. $\Box$ 
\end{proof}

\begin{remark}
There are two ways to make $\mathbf{dSets^+}/S$ into a simplicial category with simplicial tensoring and cotensoring. We can use as mapping objects $\mathrm{Map}_S^{\flat/\sharp}(X, Y)$ and for $K \in \mathbf{sSets}$ the tensoring and cotensoring can be given by
\begin{eqnarray*}
X \otimes K & := & X \otimes i_!(K)^{\flat/\sharp} \\
(X^K)_T & := & \mathbf{dSets}/S(i_!(K)^{\flat/\sharp} \otimes \Omega[T]^\flat, X) \quad, \quad \mathcal{E}_{X^K} := \coprod_{n \in \mathbb{Z}_{\geq 0}} \mathbf{dSets}/S(i_!(K)^{\flat/\sharp} \otimes \Omega[C_n]^\sharp, X)
\end{eqnarray*}
Here the structural map $X \otimes K \longrightarrow S$ is given by first projecting onto $X$ and then using the structural map of $X$. This projection onto the first factor is possible because $K$ is a simplicial set.
\end{remark}

\begin{proposition}
If we regard $\mathbf{dSets^+}/S$ as a simplicial category with mapping objects given by $\mathrm{Map}_S^\sharp(X,Y)$ and corresponding tensors and cotensors over $\mathbf{sSets}$, then $\mathbf{dSets^+}/S$ acquires the structure of a simplicial model category.
\end{proposition}
\begin{proof}
Since the unit of the tensor product is cofibrant, we need only check that given a cofibration $f: X \longrightarrow Y$ of $\mathbf{dSets^+}/S$ and an inclusion of simplicial sets $g: A \subseteq B$ the smash product $f \wedge i_!(g)^\sharp$ is a cofibration which is trivial if either $f$ is a trivial cofibration or $g$ is a homotopy equivalence. The case where $f$ is trivial follows immediately from the corollary above (setting $T = \Omega[\eta]$). The case where $g$ is a homotopy equivalence of simplicial sets will also follow if we can show that $i_!(g)^\sharp$ is a marked equivalence in $\mathbf{dSets^+}/\Omega[\eta]$; in other words, in $\mathbf{sSets^+}$. \par 
Let $Z \longrightarrow \Delta^0$ be a coCartesian fibration of simplicial sets. This is simply an inner fibration, the coCartesian edges being precisely the equivalences in $Z$. We denote by $k(Z)$ the maximal Kan complex of $Z$. We have a natural isomorphism
\begin{equation*}
\mathrm{Map}_{\Delta^0}^\sharp(A^\sharp, Z^\natural) \simeq \mathbf{Hom_{sSets}}(A, k(Z)) 
\end{equation*}
and similarly for $B$. The map $g$ induces a homotopy equivalence from $ \mathbf{Hom_{sSets}}(B, k(Z))$ to $ \mathbf{Hom_{sSets}}(A, k(Z))$ by assumption, so it gives us a homotopy equivalence
\begin{equation*}
\mathrm{Map}_{\Delta^0}^\sharp(B^\sharp, Z^\natural) \longrightarrow \mathrm{Map}_{\Delta^0}^\sharp(A^\sharp, Z^\natural)
\end{equation*}
By Proposition \ref{prop:equivKan} we conclude that $i_!(g)^\sharp$ is a marked equivalence. $\Box$
\end{proof}

\begin{remark}
If we were to use the mapping objects $\mathrm{Map}_S^\flat(X, Y)$, we would obtain a model structure enriched over simplicial sets endowed with the Joyal model structure instead of the Quillen model structure.
\end{remark}

\begin{remark}
In the special case that $S = i_!(K)$ for a simplicial set $K$, we have an identification $\mathbf{dSets^+}/S = \mathbf{sSets^+}/K$. The coCartesian model structure constructed above then reproduces Lurie's coCartesian model structure as in \cite{htt}. In particular, by slicing over $\Delta^0$ we obtain a model structure on $\mathbf{sSets^+}$ in which the fibrant objects are the $\infty$-categories with equivalences marked.
\end{remark}

\begin{remark}
Using the same methods as in this chapter one can in fact also construct a model structure on $\mathbf{dSets^+}/\ast^\flat$, i.e. on marked dendroidal sets in which we are only allowed to mark 1-corollas. It is characterized by the fact that the cofibrations are the normal monomorphisms and the fibrant objects are the $\infty$-operads with equivalences marked. This model structure is simplicial and monoidal. It is easily checked that we have a monoidal Quillen equivalence
\[
\xymatrix@R=40pt@C=40pt{
\mathbf{dSets} \ar@<.5ex>[r]^{(-)^\flat} & \mathbf{dSets^+} \ar@<.5ex>[l]^u
}
\]
Because of the fact that the model structure on $\mathbf{dSets^+}$ is in fact simplicial, it is for some purposes more convenient than the model structure on $\mathbf{dSets}$ itself.
\end{remark}

\newpage

\section{The Grothendieck construction for $\infty$-operads}
\label{section:grothconstinftyoperads}

For a dendroidal set $S$, we established a simplicial combinatorial model structure on the category $\mathbf{dSets}^+/S$. We use this to define an $\infty$-category of coCartesian fibrations. Recall that for a (simplicial) model category $\mathbf{C}$ we denote by $\mathbf{C}^\circ$ its full (simplicial) subcategory on the fibrant-cofibrant objects.

\begin{definition}
The $\infty$-category $\mathbf{coCart}(S)$ of coCartesian fibrations over $S$ is defined by
\begin{equation*}
\mathbf{coCart}(S) := \mathrm{hc}N((\mathbf{dSets}^+/S)^\circ)
\end{equation*}
\end{definition}

We want to compare this category to the category of `$S$-algebras in $\infty$-categories'. Recall from \cite{htt} that the $\infty$-category of (small) $\infty$-categories, which we denote by $\mathbf{Cat}_\infty$, is equivalent (in the Joyal model structure) to the homotopy coherent nerve of the full simplicial subcategory of $\mathbf{sSets}^+$ on the fibrant-cofibrant objects, i.e.
\begin{equation*}
\mathbf{Cat}_\infty \simeq \mathrm{hc}N((\mathbf{sSets}^+)^\circ)
\end{equation*}
We want to describe a model for $S$-algebras in this category. The results of Berger and Moerdijk \cite{bergermoerdijk1}\cite{bergermoerdijk2} can be adapted to this setting to prove:

\begin{theorem}
If $S$ is normal, so that $\mathrm{hc}\tau_d(S)$ is cofibrant, there exists a left proper simplicial model structure on the (simplicial) category $\mathrm{Alg}_{\mathrm{hc}\tau_d(S)}(\mathbf{sSets}^+)$ of $\mathrm{hc}\tau_d(S)$-algebras in $\mathbf{sSets}^+$ in which a map of algebras is a weak equivalence (resp. a fibration) if and only if it is a pointwise weak equivalence (resp. a pointwise fibration).
\end{theorem}

This enables us to do the following:

\begin{definition}
For a normal dendroidal set $S$ we define the $\infty$-category of $S$-algebras in $\mathbf{Cat}_\infty$ by
\begin{equation*}
\mathrm{Alg}_{S}(\mathbf{Cat}_\infty) := \mathrm{hc}N\bigl((\mathrm{Alg}_{\mathrm{hc}\tau_d(S)}(\mathbf{sSets}^+))^\circ\bigr)
\end{equation*}
\end{definition}

Our main result is the following:

\begin{theorem}
\label{thm:Grothendieckconstruction}
For a normal dendroidal set $S$, there is an equivalence of $\infty$-categories
\begin{equation*}
\mathbf{coCart}(S) \simeq \mathrm{Alg}_{S}(\mathbf{Cat}_\infty)
\end{equation*}
\end{theorem}

The rest of this chapter is devoted to the proof of Theorem \ref{thm:Grothendieckconstruction}.

\subsection{The straightening functor}
We wish to construct a marked version of the straightening functor we introduced in section \ref{section:leftfib}. First, define the functor
\begin{equation*}
\mathrm{hc}\tau_d^{\mathrm{op}}: \mathbf{dSets} \longrightarrow \mathbf{sOper}
\end{equation*}
discussed in the prerequisites, which is described by
\begin{equation*}
\mathrm{hc}\tau_d^{\mathrm{op}}(S)(s_1, \ldots, s_n; s) := \bigl(\mathrm{hc}\tau_d(S)(s_1, \ldots, s_n; s)\bigr)^\mathrm{op}
\end{equation*}
Since a simplicial set and its opposite have isomorphic geometric realizations, the simplicial operad $\mathrm{hc}\tau_d^{\mathrm{op}}(S)$ is weakly equivalent to the simplicial operad $\mathrm{hc}\tau_d(S)$ and this change is insubstantial. In particular, it is irrelevant to the statement of Theorem \ref{thm:Grothendieckconstruction}. However, the functor $\mathrm{hc}\tau_d^{\mathrm{op}}$ is more convenient when defining the straightening functor here. \par 
We will first define a functor
\begin{equation*}
St_S: \mathbf{dSets}/S \longrightarrow \mathrm{Alg}_{\mathrm{hc}\tau_d^{\mathrm{op}}(S)}(\mathbf{sSets})
\end{equation*}
and take care of the markings later. In analogy with section \ref{section:leftfib}, we define
\begin{equation*}
St_{\Omega[T]}(\mathrm{id}_{\Omega[T]})(c) := \Delta[T/c]^\mathrm{op}
\end{equation*}
Composition is again given by grafting of trees, assigning length 1 to newly arising inner edges. For a map $p: \Omega[T] \longrightarrow S$ we set
\begin{equation*}
St_S(p) := \mathrm{hc}\tau_d^{\mathrm{op}}(p)_!\bigl(St_{\Omega[T]}(\mathrm{id}_{\Omega[T]})\bigr) 
\end{equation*}
Functoriality in $T$ works the same way it did in section 2. We can left Kan extend $St_S$ to obtain a functor
\begin{equation*}
St_S: \mathbf{dSets}/S \longrightarrow \mathrm{Alg}_{\mathrm{hc}\tau_d^{\mathrm{op}}(S)}(\mathbf{sSets}): X \longmapsto \varinjlim_{\Omega[T] \rightarrow X} St_S(\Omega[T] \rightarrow X)
\end{equation*}
We now proceed to defining the markings on the straightening functor. Suppose we are given a map of marked dendroidal sets
\begin{equation*}
p: (X, \mathcal{E}_X) \longrightarrow S^\sharp
\end{equation*}
An $n$-corolla $\xi$ of $X$ with root $x$ determines an inclusion
\begin{equation*}
((\Delta^1)^{\times n})^{\mathrm{op}} \subseteq St_S(p)(p(x))
\end{equation*}
We denote this $n$-cube by $\tilde \xi$. Also note that any corolla $\sigma$ of $S$ with inputs $s_1, \ldots, s_n$ and output $s$ defines a map
\begin{equation*}
\sigma_!: St_S(p)(s_1) \times \ldots \times St_S(p)(s_n) \longrightarrow St_S(p)(s)
\end{equation*}

\begin{definition}
We define the \emph{marked straightening functor}
\begin{equation*}
St^+_S: \mathbf{dSets^+}/S \longrightarrow \mathrm{Alg}_{\mathrm{hc}\tau_d^{\mathrm{op}}(S)}(\mathbf{sSets^+}) 
\end{equation*}
by
\begin{equation*}
St^+_S(p)(s) = (St_S(p)(s), \mathcal{E}_p(s))
\end{equation*}
Here $\mathcal{E}_p(s)$ is the set of edges of $St_S(p)(s)$ of the form
\begin{equation*}
\sigma_!(\tilde e)
\end{equation*}
where $\tilde e \in (\tilde\xi)_1$ for some marked corolla $\xi$ of $X$ and $\sigma$ is a corolla of $S$.
\end{definition}

\begin{remark}
Note that our choice of markings is the smallest choice containing all 1-simplices contained in cubes induced by coCartesian corollas of $X$ which is functorial in $s$.
\end{remark}

\begin{remark}
With the definition of the straightening functor in place, we will now omit the `op' on $\mathrm{hc}\tau_d^{\mathrm{op}}$ from the notation for the rest of this section to avoid extremely awkward-looking expressions. It should be clear from the context what is meant.
\end{remark}

A morphism $\phi: S \longrightarrow R$ of dendroidal sets gives us an adjoint pair
\[
\xymatrix@R=40pt@C=40pt{
\phi_!: \mathbf{dSets^+}/S \ar@<.5ex>[r] & \mathbf{dSets^+}/R: \phi^* \ar@<.5ex>[l]
}
\]
Similarly, a morphism $\psi: P \longrightarrow Q$ of simplicial operads yields an adjoint pair
\[
\xymatrix@R=40pt@C=40pt{
\psi_!: \mathrm{Alg}_{P}(\mathbf{sSets^+}) \ar@<.5ex>[r] & \mathrm{Alg}_{Q}(\mathbf{sSets^+}): \psi^* \ar@<.5ex>[l]
}
\]

For later reference, we record the following properties of the straightening functor, which are obvious from its construction:

\begin{proposition}
\label{prop:straightening}
\begin{itemize}
\item[(1)] $St_S^+$ preserves colimits
\item[(2)] For a morphism $\phi: S \longrightarrow R$ of dendroidal sets the following diagram commutes:
\[
\xymatrix@C=40pt{
\mathbf{dSets^+}/S \ar[r]^{St_S^+}\ar[d]_{\phi_!} & \mathrm{Alg}_{\mathrm{hc}\tau_d(S)}(\mathbf{sSets^+}) \ar[d]^{\mathrm{hc}\tau_d(\phi)_!} \\
\mathbf{dSets^+}/R \ar[r]_{St_R^+} & \mathrm{Alg}_{\mathrm{hc}\tau_d(R)}(\mathbf{sSets^+})
}
\]
\end{itemize}
\end{proposition}

We deduce that the straightening functor admits a right adjoint, unsurprisingly called the \emph{unstraightening functor} and denoted
\begin{equation*}
Un_S^+: \mathrm{Alg}_{\mathrm{hc}\tau_d(S)}(\mathbf{sSets^+}) \longrightarrow \mathbf{dSets^+}/S
\end{equation*}
It is this functor that implements the $\infty$-operadic version of the Grothendieck construction. The rest of this section is devoted to proving that $(St_S^+, Un_S^+)$ is a Quillen pair. \par 

\begin{lemma}
The straightening functor $St_S^+$ preserves cofibrations.
\end{lemma}
\begin{proof}
Given a cofibration $f: X \longrightarrow Y$ over $S^\sharp$ we may use (2) of Proposition \ref{prop:straightening} to reduce to the case $S = u(Y)$. We can also assume $f$ is a generating cofibration of either the form $\partial \Omega[T]^\flat \subseteq \Omega[T]^\flat$ or $\Omega[C_n]^\flat \subseteq \Omega[C_n]^\sharp$. In the first case, when $T$ contains at least one vertex, a straightforward computation shows that
\begin{equation*}
St_{\Omega[T]}^+(\partial \Omega[T]^\flat)(c) \longrightarrow St_{\Omega[T]}^+(\Omega[T]^\flat)(c)
\end{equation*} 
is an isomorphism at each colour $c$ of $T$, except for $c$ the root of $T$. In this case it is a cofibration. It is now easy to deduce that $St_{\Omega[T]}^+(f)$ has the left lifting property with respect to all trivial fibrations, i.e. is a cofibration. In case $T = \eta$ the map $St_{\Omega[\eta]}^+(f)$ is simply the inclusion $\emptyset \subseteq \Delta^0$ of marked simplicial sets, which is of course a cofibration. In case $f$ equals an inclusion $\Omega[C_n]^\flat \subseteq \Omega[C_n]^\sharp$ then the underlying map $St_{\Omega[C_n]}(f)$ is an isomorphism, which implies it has the left lifting property with respect to trivial fibrations. $\Box$ 
\end{proof}

We first state the proof that $(St_S^+, Un_S^+)$ is a Quillen pair and treat the necessary technical lemmas after that:

\begin{proposition}
The adjunction
\[
\xymatrix@R=40pt@C=40pt{
St^+_S: \mathbf{dSets^+}/S \ar@<.5ex>[r] & \mathrm{Alg}_{\mathrm{hc}\tau_d(S)}(\mathbf{sSets}^+): Un^+_S\ar@<.5ex>[l]
}
\]
is a Quillen adjunction.
\end{proposition}
\begin{proof}
We have already checked that $St^+_S$ preserves cofibrations, so it suffices to check that it preserves trivial cofibrations. We will actually verify that $St^+_S$ preserves weak equivalences. By the same argument used in the proof above, we are free to replace $S$ by $u(Y)$; we will do so throughout the rest of this section. \par 
Let $f: X \longrightarrow Y$ be a marked equivalence in $\mathbf{dSets^+}/S$. We pick a marked anodyne $X \longrightarrow X'$ such that $X'$ is fibrant and a marked anodyne $X' \coprod_X Y \longrightarrow Y'$ such that $Y'$ is fibrant. Since $St^+_S$ maps marked anodynes to trivial cofibrations (Lemma \ref{lemma:Stmarkedanodynes} below) we get a diagram
\[
\xymatrix@C=50pt{
St^+_S(X) \ar[r]^{\sim}\ar[d] & St^+_S(X') \ar[d]\ar[dr] & \\
St^+_S(Y) \ar[r]^{\sim} & St^+_S(X' \coprod_X Y) \ar[r]^{\sim} & St^+_S(Y')
}
\] 
We see that $St^+_S(f)$ is a weak equivalence if and only if the induced map $St^+_S(X') \longrightarrow St^+_S(Y')$ is, i.e. we may reduce to the case where $X$ and $Y$ are fibrant. In this case there exists a strong homotopy inverse $g$ of $f$. We will show that $St^+_S(g)$ is a left homotopy inverse to $St^+_S(f)$; a similar argument will show that it is a right homotopy inverse. \par 
We have a map $h: X \otimes i_!(\Delta^1)^\sharp \longrightarrow X$ such that the restriction of $h$ to $X \otimes \{0\}$ equals $\mathrm{id}_X$ and its restriction to $X \otimes \{1\}$ equals $g \circ f$. Now consider the diagram
\[
\xymatrix@C=50pt{
St^+_S(X) \ar[d]_{St^+_S(\mathrm{id}_X \otimes \{0\})} \ar@{=}[dr] & \\
St^+_S(X \otimes i_!(\Delta^1)^\sharp) \ar[r]_{St^+_S(h)} & St^+_S(X)
}
\] 
The vertical arrow is a weak equivalence by Lemma \ref{lemma:Stproduct} below, so $St^+_S(h)$ yields an isomorphism in the homotopy category of $\mathrm{Alg}_{\mathrm{hc}\tau_d(S)}(\mathbf{sSets}^+)$. Since $\mathrm{id}_X$ and $g \circ f$ can be interpreted as sections of $h$, they must both induce inverses of $St^+_S(h)$ in that homotopy category. Hence they are homotopic. $\Box$
\end{proof}

\begin{lemma}
\label{lemma:Stmarkedanodynes}
The functor $St_S^+$ maps marked anodyne morphisms to trivial cofibrations.
\end{lemma}
\begin{proof}
We are free to check this only on generating marked anodynes. We need to check five cases. \par 
\textbf{(1)} Suppose $f: \Lambda^e[T]^\flat \subseteq \Omega[T]^\flat$ is an inner horn inclusion. If $c$ is a colour of $T$ other than its root then the map $St_{\Omega[T]}^+(f)(c)$ is an isomorphism. Let $r$ be the root of $T$. If we can prove that $St_{\Omega[T]}^+(f)(r)$ is a trivial cofibration of marked simplicial sets then it is straightforward to show that $St_{\Omega[T]}^+(f)$ has the left lifting property with respect to fibrations and is thus a trivial cofibration. Now define $C := \mathrm{col}(T)-\{e,r\}$ and denote by $K$ the cube $(\Delta^1)^C$. Inspection of the Boardman-Vogt $W$-construction shows that we have a map
\begin{equation*}
(K \times \Delta^1)^{\mathrm{op}} \longrightarrow St_{\Omega[T]}(\Omega[T])(r)
\end{equation*}
corresponding to the unique non-degenerate $T$-dendrex of $\Omega[T]$. \par
For the purposes of this proof, we will say that a colour $c$ of $T$ is above $e$ if $c \neq e$ and there exists a subtree of $T$ containing the colour $c$ and having $e$ as its root. Any vertex $v$ of $K$ defines a map $\bar v: C \longrightarrow \{0,1\}$. We let $V \subseteq K_0$ be the set of vertices $v$ of $K$ such that $\bar v^{-1}(\{1\})$ contains a colour of $T$ which is above $e$. We define $\mathcal{E}$ to be the union of the set of 1-simplices of $V \times \Delta^1$ with the set of degenerate 1-simplices of $K \times \Delta^1$. Define
\begin{equation*}
M := K \times \{1\} \coprod_{\partial K \times \{1\}} \partial K \times \Delta^1
\end{equation*} 
and set $\mathcal{E}' := \mathcal{E} \cap M_1$. There is a pushout diagram
\[
\xymatrix{
(M, \mathcal{E}')^{\mathrm{op}} \ar[d]\ar[r] & St^+_{\Omega[T]}(\Lambda^e[T]^\flat)(r) \ar[d]^{St^+_{\Omega[T]}(f)(r)} \\
(K \times \Delta^1, \mathcal{E})^{\mathrm{op}} \ar[r] & St^+_{\Omega[T]}(\Omega[T]^\flat)(r)
}
\]
The left vertical map is marked anodyne by Lemma \ref{lemma:markedanodyne4} below, so we conclude that $St^+_{\Omega[T]}(f)(r)$ is marked anodyne and hence a trivial cofibration. \par 
\textbf{(2)} Let $f$ be the inclusion
\begin{equation*}
(\Lambda^v[T], \mathcal{E}\cap cor(\Lambda^v[T])) \subseteq (\Omega[T], \mathcal{E})
\end{equation*}
where $T$ is a tree with at least two vertices, $v$ is the vertex of a leaf corolla and $\mathcal{E}$ is the set of all degenerate corollas of $T$ together with this leaf corolla. Again $St_{\Omega[T]}^+(f)(c)$ is an isomorphism for any colour $c$ of $T$ other than the root $r$. Denote the leaves of the corolla with vertex $v$ by $l_1, \ldots, l_n$. Set $C = \mathrm{col}(T) - \{l_1, \ldots, l_n, r\}$ and $K = (\Delta^1)^C$. Now let $V$ denote the set of vertices $w$ of $K$ such that $\bar w^{-1}(\{1\})$ contains the output edge of the vertex $v$. Set $\mathcal{E}$ to be the union of the set of 1-simplices of $V \times (\Delta^1)^n$ with the set of degenerate 1-simplices of $K \times (\Delta^1)^n$. Define
\begin{equation*}
M := K \times \{1\}^n \coprod_{\partial K \times \{1\}^n} \partial K \times (\Delta^1)^n
\end{equation*}
and set $\mathcal{E}' = \mathcal{E} \cap M_1$. The map $St_{\Omega[T]}^+(f)(r)$ is now a pushout of the map 
\begin{equation*}
(M, \mathcal{E}')^{\mathrm{op}} \longrightarrow (K \times (\Delta^1)^n, \mathcal{E})^{\mathrm{op}}
\end{equation*}
which is again marked anodyne by Lemma \ref{lemma:markedanodyne4}. \par
\textbf{($\mathbf{2^*}$)} Set $f$ to be the inclusion
\begin{equation*}
\coprod_{c \in \mathrm{in}(C_n)} \eta_c \subseteq \Omega[C_n]^\sharp
\end{equation*}
for some $n \geq 0$. The map $St_{\Omega[C_n]}^+(f)$ is an isomorphism at all the leaves of $C_n$ and can be identified with the map
\begin{equation*}
\{0\} \longrightarrow ((\Delta^1)^{\times n})^\sharp 
\end{equation*}
at the root of $C_n$. This map is easily seen to be marked anodyne. \par
\textbf{(3)} Let $T$ be a tree with two vertices $v$ and $w$, let $\mathcal{E}$ be the set of degenerate corollas of $\Omega[T]$ together with the corollas with vertices $v$ and $w$ and let $f$ be the inclusion $(\Omega[T], \mathcal{E}) \subseteq \Omega[T]^\sharp$. One verifies that $St_{\Omega[T]}^+(f)$ is an isomorphism at every colour of $T$ except for the root $r$. Suppose $v$ has $j$ leaves and $w$ has $k$ leaves. We can identify $St_{\Omega[T]}^+(f)(r)$ with a pushout of the map
\begin{equation*}
((\Delta^1)^j \times (\Delta^1)^k, \mathcal{E})^{\mathrm{op}} \subseteq \bigl(((\Delta^1)^j \times (\Delta^1)^k)^{\mathrm{op}}\bigr)^\sharp
\end{equation*}
where $\mathcal{E}$ is the union of all degenerate 1-simplices together with the 1-simplices of the simplicial set
\begin{equation*}
(\Delta^1)^j \times \{0\}^k \cup (\Delta^1)^j \times \{1\}^k \cup \{1\}^j \times (\Delta^1)^k
\end{equation*}
This map can be obtained by a composition of pushouts of the maps
\begin{equation*}
(\Lambda^2_1)^\sharp \coprod_{(\Lambda^2_1)^\flat} (\Delta^2)^\flat \longrightarrow (\Delta^2)^\sharp
\end{equation*}
and
\begin{equation*}
(\Lambda^2_0)^\sharp \coprod_{(\Lambda^2_0)^\flat} (\Delta^2)^\flat \longrightarrow (\Delta^2)^\sharp
\end{equation*}
The first of these is marked anodyne by definition, the second is marked anodyne by Corollary \ref{cor:markedanodyne3}. \par 
\textbf{(4)} Let $K$ be a Kan complex in simplicial sets and let $f$ be the inclusion $i_!(K)^\flat \subseteq i_!(K)^\sharp$. The story now only involves simplicial sets and one can copy the argument used by Lurie in the proof of Proposition 3.2.1.11 of \cite{htt}. $\Box$
\end{proof}

In the proof above, we used the following result:

\begin{lemma}
\label{lemma:markedanodyne4}
Let $K$ be a simplicial set and let $V$ be a set of vertices of $K$. Define $\mathcal{E}$ to be union of the set of degenerate simplices of $K \times (\Delta^1)^n$ with the set of 1-simplices of $V \times (\Delta^1)^n$. Define
\begin{equation*}
M := K \times \{0\}^n \coprod_{\partial K \times \{0\}^n} \partial K \times (\Delta^1)^n
\end{equation*}
Suppose that for each non-degenerate simplex $\sigma$ of $K$ the initial vertex of $\sigma$ belongs to $V$. Then the inclusion
\begin{equation*}
(M, \mathcal{E} \cap M_1) \subseteq (K \times (\Delta^1)^n, \mathcal{E})
\end{equation*}
is marked anodyne.
\end{lemma}
\begin{proof}
We can filter the inclusion by working one simplex of $K$ at a time. With this method we reduce to proving the following: let $\mathcal{E}$ be the union of degenerate simplices of $\Delta^n \times \Delta^1$ with the 1-simplex $\{0\} \times \Delta^1$. Set
\begin{equation*}
M := \Delta^n \times \{0\} \coprod_{\partial \Delta^n \times \{0\}} \partial \Delta^n \times \Delta^1
\end{equation*}
Then the inclusion
\begin{equation*}
(M, \mathcal{E} \cap M_1) \subseteq (\Delta^n \times \Delta^1, \mathcal{E})
\end{equation*}
is marked anodyne. Indeed, using the standard subdivision of $\Delta^n \times \Delta^1$ into $(n+1)$-simplices, this inclusion can be obtained as a composition of $n$ inner anodynes followed by a pushout of a marked anodyne of type (2). $\Box$
\end{proof}

We now investigate the behaviour of the straightening functor with respect to the simplicial tensoring on $\mathbf{dSets^+}/S$. Recall that for $X \in \mathbf{dSets^+}/S$ and $K \in \mathbf{sSets}$ this tensoring is given by $X \otimes K = X \otimes i_!(K)^\sharp$. We have the following result:

\begin{lemma}
\label{lemma:Stproduct}
Let $S$ be normal. For $p: X \longrightarrow S$ in $\mathbf{dSets^+}/S$ and $K \in \mathbf{sSets}$ the natural map
\begin{equation*}
St^+_S(X \otimes K) \longrightarrow St^+_S(X) \otimes K
\end{equation*}
is a weak equivalence.
\end{lemma}
\begin{proof}
Both sides are compatible with the formation of colimits in $X$ and weak equivalences are closed under filtered colimits, which allows us to reduce to the case where $X$ has only finitely many nondegenerate dendrices. We work by induction on the maximal size of $T$-dendrices of $X$: first observe that the statement is trivial if $X$ has only $\eta$-dendrices (i.e. is just a discrete set). Now suppose the statement is true for dendroidal sets $X$ only having nondegenerate $T$-dendrices for trees $T$ with at most $n$ vertices, where $n \geq 1$ (we will prove the case $n = 1$ later, to establish the induction base). If a dendroidal set $X'$ has nondegenerate $T'$-dendrices for trees $T'$ with up to $n+1$ vertices, the set of which we denote by $\{\alpha_i, T'_i\}_{i \in I}$ with $T'_i$ denoting the shape of the dendrex, there is a pushout square
\[
\xymatrix{
St^+_S(\coprod_{i \in I}\partial \Omega[T'_i]^\flat \otimes K) \ar[r]\ar[d] & St^+_S(\mathrm{sk}_n(X') \otimes K) \ar[d] \\
St^+_S(\coprod_{i \in I} \Omega[T'_i]^\flat \otimes K) \ar[r] & St^+_S(X' \otimes K)
}
\]
which is in fact a homotopy pushout by left properness of $\mathrm{Alg}_{\mathrm{hc}\tau_d(S)}(\mathbf{sSets^+})$. By the inductive hypothesis, the maps 
\begin{eqnarray*}
St^+_S(\coprod_{i \in I}\partial \Omega[T'_i]^\flat \otimes K) & \longrightarrow & St^+_S(\coprod_{i \in I}\partial \Omega[T'_i]^\flat) \otimes K \\
St^+_S(\mathrm{sk}_n(X') \otimes K) & \longrightarrow & St^+_S(\mathrm{sk}_n(X')) \otimes K
\end{eqnarray*}
are weak equivalences. Hence it will suffice to show that the map
\begin{equation*}
St^+_S(\Omega[T'_i]^\flat \otimes K) \longrightarrow St^+_S(\Omega[T'_i]^\flat) \otimes K
\end{equation*}
is a weak equivalence. In other words, we may assume $X$ is of the form $\Omega[T]$. \par 
Now let $E(T)$ denote the set of inner edges of $T$. Define $\mathrm{spine}(T)$ to be a coequalizer as follows:
\begin{equation*}
\coprod_{e \in E(T)} \eta_e \rightrightarrows \coprod_{\alpha \in cor(T)} \Omega[\alpha] \rightarrow \mathrm{spine}(T)
\end{equation*}
where the two arrows come from the inclusion of each inner edge $e$ as the root of one corolla and a leaf of another. Now observe that
\begin{equation*}
\varsigma: \mathrm{spine}(T) \longrightarrow \Omega[T]
\end{equation*}
is inner anodyne, so that $\varsigma^\flat$ is marked anodyne. Applying Lemma \ref{lemma:Stmarkedanodynes} and a similar trick as above, we see that we may reduce to the case where $X$ is a corolla. We also still need to establish the induction base, so we will prove the statement in case $X$ is either a marked or unmarked corolla. Note that we may also apply the procedure just described to reduce to the case where $K = \Delta^1$. We need to check that
\begin{equation*}
St^+_S(\Omega[C_n]^{\flat/\sharp} \otimes \Delta^1) \longrightarrow St^+_S(\Omega[C_n]^{\flat/\sharp}) \otimes \Delta^1
\end{equation*} 
is a weak equivalence. First apply Proposition \ref{prop:straightening} to reduce to the case $S = \Omega[C_n]$. Now observe that the stated map is an isomorphism at the leaves of $C_n$. At the root of $C_n$ it is a matter of direct calculation. For convenience of the reader we describe the case where $n = 1$, the rest is similar. The marked simplicial set $St^+_{i_!(\Delta^1)}(\Omega[C_1]^{\flat/\sharp} \otimes \Delta^1)(1)$ may be described as
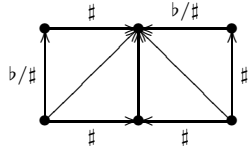
\[
\xymatrix@R=35pt@C=35pt{
*=0{\bullet}\ar[r]^{\sharp}&*=0{\bullet}&*=0{\bullet}\ar[l]_{\flat/\sharp} \\
*=0{\bullet}\ar[u]^{\flat/\sharp}\ar[ur]\ar[r]_{\sharp}&*=0{\bullet}\ar[u]&*=0{\bullet}\ar[u]_{\sharp}\ar[ul]\ar[l]^{\sharp}
}
\]
whereas $\bigl(St^+_{i_!(\Delta^1)}(\Omega[C_1]^{\flat/\sharp}) \otimes \Delta^1\bigr)(1)$ may be described as
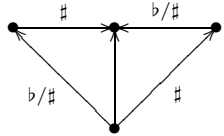
\[
\xymatrix@R=35pt@C=35pt{
*=0{\bullet}\ar[r]^{\sharp}&*=0{\bullet}&*=0{\bullet}\ar[l]_{\flat/\sharp} \\
&*=0{\bullet}\ar[u]\ar[ul]^{\flat/\sharp}\ar[ur]_{\sharp}&
}
\]
The natural map between the two is indeed a marked equivalence in $\mathbf{sSets^+}$. $\Box$
\end{proof}

The straightening functor is not quite a simplicial functor: the preceding result shows that it is so only in a weak sense. However, observe that by playing with the present adjunctions we get the following chain of maps:
\begin{eqnarray*}
\mathbf{dSets^+}/S(X, Un^+_S(A^K)) & \simeq & \mathrm{Alg}_{\mathrm{hc}\tau_d(S)}(\mathbf{sSets^+})(St^+_S(X)\otimes K, A) \\ 
& \rightarrow & \mathrm{Alg}_{\mathrm{hc}\tau_d(S)}(\mathbf{sSets^+})(St^+_S(X \otimes K), A) \\
& \simeq & \mathbf{dSets^+}/S(X, Un^+_S(A)^K)
\end{eqnarray*}
This gives us a natural map $Un^+_S(A^K) \longrightarrow Un^+_S(A)^K$. From this we get the maps
\begin{eqnarray*}
\mathbf{sSets}(K, \mathrm{Map}_{\mathrm{Alg}_{\mathrm{hc}\tau_d(S)}(\mathbf{sSets^+})}(A, B)) & \simeq & \mathrm{Alg}_{\mathrm{hc}\tau_d(S)}(\mathbf{sSets^+})(A, B^K) \\
& \rightarrow & \mathbf{dSets^+}/S(Un^+_S(A),Un^+_S(B^K)) \\
& \rightarrow & \mathbf{dSets^+}/S(Un^+_S(A),Un^+_S(B)^K) \\
& \simeq & \mathbf{sSets}\bigl(K, \mathrm{Map}^\sharp_S(Un^+_S(A),Un^+_S(B))\bigr)
\end{eqnarray*}
and hence a natural map
\begin{equation*}
\mathrm{Map}_{\mathrm{Alg}_{\mathrm{hc}\tau_d(S)}(\mathbf{sSets^+})}(A, B) \longrightarrow \mathrm{Map}^\sharp_S(Un^+_S(A),Un^+_S(B))
\end{equation*}
This makes the unstraightening functor into a simplicial functor.

\subsection{Mapping trees}
Before being able to prove Theorem \ref{thm:Grothendieckconstruction} in the case where $S$ is a representable, i.e. of the form $\Omega[T]$, we will study so-called \emph{mapping trees}. These are dendroidal sets equipped with a map to $\Omega[T]$ such that for any coCartesian fibration over $\Omega[T]$ there exists a mapping tree equivalent to it. These mapping trees are artificial devices which will be useful later on, because it is possible to compute the effect of the straightening functor on these in an explicit fashion. \par  
Any tree $T \in \mathbf{\Omega}$ gives rise to an operad in $\mathbf{Sets}$, namely the free operad generated by that tree, which we denote by $\Omega(T)$. Suppose we are given an $\Omega(T)$-algebra $A$ in $\mathbf{sSets}$. In other words, we are given a simplicial set $A(c)$ for each edge $c$ of $T$ and a map $A(c_1) \times  \cdots \times A(c_n) \longrightarrow A(c)$ for each operation $(c_1, \ldots, c_n) \longrightarrow c$ of $\Omega(T)$ satisfying the necessary composition and symmetry identities. We define a dendroidal set $M(A)$, the \emph{mapping tree} of $A$, as follows. For a tree $R \in \mathbf{\Omega}$, an $R$-dendrex of $M(A)$ is determined by the following data:
\begin{itemize}
\item A map $\delta: R \longrightarrow T$ in $\mathbf{\Omega}$
\item For each leaf $l$ of $R$, a map $\lambda_l: \Delta^{\nu_l} \longrightarrow A(l)$, where $\nu_l$ is the number of vertices between $l$ and the root of $R$. To be more precise, for each leaf $l$ of $R$, there is a unique directed subgraph of $R$ from $l$ to the root of $R$, which can be identified with a linear tree. In other words, it is of the form $i_!(\Delta^{\nu_l})$ for some integer $\nu_l$, which is precisely the number of inner vertices of this graph
\end{itemize}
We need to explain how this defines a dendroidal set. Suppose $Q \longrightarrow R$ is a face map. Composition with $\delta$ gives a map $Q \longrightarrow T$. Also, $Q$ induces subtrees of the directed graphs $i_!(\Delta^{\nu_l})$ we used in the definition. If $Q$ is not a face map chopping off a leaf corolla, these induced subtrees correspond to identities, inner faces or possibly a face chopping of the root of these trees, and we simply take corresponding faces of the $i_!(\Delta^{\nu_l})$. Now suppose $Q$ chops off some leaf corolla of $R$ with inputs $l_1, \ldots, l_n$ and output $k$. Note that
\begin{equation*}
\nu_k + 1 = \nu_{l_1} = \ldots = \nu_{l_n}
\end{equation*}
Our subtree $Q$ defines the faces
\begin{equation*}
\Delta^{\nu_k} \longrightarrow \Delta^{\nu_{l_i}}
\end{equation*}
given by chopping of the initial vertex of $\Delta^{\nu_{l_i}}$. Composing these with the map defined by our $R$-dendrex we get a map
\begin{equation*}
(\Delta^{\nu_k})^n \longrightarrow A(l_1) \times \cdots \times A(l_n)
\end{equation*}
Using the algebra structure of $A$ we get a map
\begin{equation*}
(\Delta^{\nu_k})^n \longrightarrow A(k) 
\end{equation*}
Composing this with the diagonal
\begin{equation*}
\Delta^{\nu_k} \longrightarrow (\Delta^{\nu_k})^n
\end{equation*}
gives the structure map we need to define our $Q$-dendrex. One can now also easily treat the case where $Q$ is a degeneracy of $R$. Note that the mapping tree $M(A)$ comes equipped with a natural map to $\Omega[T]$. Its fiber over a colour $c$ of $T$ is isomorphic to the simplicial set $A(c)$. \par 
\begin{remark}
The definition of the mapping tree $M(A)$ may seem strange at first; the reader might want to convince him- or herself that it is actually quite close to what one would write down if one was to construct the `Grothendieck construction' of $A$ by hand. Close, but not quite it. The mapping tree will in general not be a coCartesian fibration over $\Omega[T]$. The use of these mapping trees comes from the fact that they are relatively easy to straighten and that they are `good enough' for our purposes, cf. Proposition \ref{prop:mappingtrees} below.
\end{remark}

To define a corolla of $M(A)$, we must give a corolla of $\Omega[T]$ and a 1-simplex of $A(l)$ for each leaf $l$ of this corolla. We will say that this corolla is marked if all these 1-simplices are degenerate and denote the set of these marked corolla's by $\mathcal{E}$. We define the marked dendroidal set
\begin{equation*}
M^+(A) := (M(A), \mathcal{E})
\end{equation*} 

The main result of this section is the following:

\begin{proposition}
\label{prop:mappingtrees}
Let $X \longrightarrow \Omega[T]$ be a coCartesian fibration. Then there exists a simplicial $\Omega(T)$-algebra $A$ and a commutative diagram
\[
\xymatrix{
M^+(A) \ar[rr]^\phi\ar[dr] & & X^\natural \ar[dl] \\
& \Omega[T]^\sharp & 
}
\]
such that $\phi$ is a marked equivalence in $\mathbf{dSets^+}/\Omega[T]$. Also, the underlying map of $\phi$ is an operadic equivalence of dendroidal sets.
\end{proposition}
\begin{proof}
Apply Corollary \ref{cor:mappingtree} below to find suitable $A$ and $\phi$ such that $\phi$ induces categorical equivalences $A(c) \longrightarrow X_c$ for each colour $c$ of $\Omega[T]$. Subsequently apply Lemmas \ref{lemma:mappingtreeoperadicequiv} and \ref{lemma:mappingtreemarkedequiv} to establish the proposition. $\Box$
\end{proof}

\begin{lemma}
\label{lemma:mappingtree}
Suppose we are given a coCartesian fibration $X \longrightarrow \Omega[T]$. Let $\{l_1, \ldots, l_N\}$ be the set of leaves of $T$ and suppose we are given a simplicial set $K_{l_j}$ for each leaf $l_j$ of $T$. Define the map $\kappa$ to be the composition
\begin{equation*}
\coprod_{j=1}^N i_!(K_{l_j})^\flat \longrightarrow \coprod_{j=1}^N \eta_{l_j}^\sharp \longrightarrow \Omega[T]^\sharp
\end{equation*}
Finally, suppose we are given a commutative diagram
\[
\xymatrix{
\coprod_{j=1}^N i_!(K_{l_j})^\flat \ar_{\kappa}[dr]\ar[rr] & & X^\natural \ar[dl] \\
& \Omega[T]^\sharp &
}
\]
Then there exists a simplicial $\Omega(T)$-algebra $A$ and an extension of the above diagram to a diagram
\[
\xymatrix{
\coprod_{j=1}^N i_!(K_{l_j})^\flat \ar_{\kappa}[dr]\ar[r] & M^+(A) \ar[r]^\phi & X^\natural \ar[dl] \\
& \Omega[T]^\sharp &
}
\]
such that $\phi$ induces categorical equivalences on the (underlying simplicial sets of the) fibers over the colours of $\Omega[T]$.
\end{lemma}
\begin{proof}
In case $T = \eta$ the result is trivial; indeed, just fix a factorization
\begin{equation*}
K_{\eta} \longrightarrow A \longrightarrow X 
\end{equation*}
where the first map is a cofibration and the second one a trivial fibration. We proceed by induction on the number of vertices of $T$. Fix a leaf corolla $\xi$ of $T$; we denote its vertex by $v$ and its root by $r$. Reindex the leaves such that $\{l_1, \ldots, l_n\}$ are the leaves of this corolla and $\{l_{n+1}, \ldots, l_N\}$ is the set of remaining leaves. For each $1 \leq j \leq n$ we fix a factorization
\begin{equation*}
K_{l_j} \longrightarrow  A(l_j) \longrightarrow X_{l_j}
\end{equation*}
into a cofibration followed by a trivial fibration. Labelling the colours of $C_n$ by the colours of $\alpha$, we define a simplicial $\Omega(C_n)$-algebra $\Xi$ by
\begin{eqnarray*}
\Xi(l_j) & = & A(l_j) \quad \text{for $1 \leq j \leq n$} \\
\Xi(r) & = & \prod_{j=1}^n A(l_j)
\end{eqnarray*}
with the obvious algebra structure. The map
\begin{equation*}
\coprod_{j=1}^n A(l_j)^\flat \longrightarrow M^+(\Xi)
\end{equation*}
is marked anodyne [...], so that we can find a map as indicated by the dotted arrow in the following diagram:
\[
\xymatrix{
\coprod_{j=1}^n i_!(A(l_j))^\flat \ar[d]\ar[r] & X^\natural \ar[d] \\
M^+(\Xi) \ar@{-->}[ur]\ar[r] & \Omega[T]^\sharp
}
\]
By the inductive hypothesis there exists a simplicial $\Omega(\partial_v T)$-algebra $A'$ and a diagram
\[
\xymatrix@C=50pt{
\Xi(r)^\flat \coprod \bigl(\coprod_{j=n+1}^N i_!(K_{l_j})^\flat\bigr) \ar[r]\ar[dr] & M^+(A') \ar[r]^{\phi '} & (X \times_{\Omega[T]} \Omega[\partial_v T])^\natural \ar[dl] \\
& \Omega[\partial_v T]^\sharp & 
}
\]
such that $\phi '$ induces categorical equivalences on the underlying simplicial sets of the fibers over $\Omega[\partial_v T]$. Considering the fiber over $r$ we get a map
\begin{equation*}
\Xi(r) = \prod_{j=1}^n A(l_j) \longrightarrow A'(r)
\end{equation*}
We join this map onto $A'$ to obtain a simplicial $\Omega(T)$-algebra $A$. The map $\phi '$ and the maps $A(l_j) \longrightarrow X_{l_j}$ together define a map
$\phi: M^+(A) \longrightarrow X^\natural$ which induces categorical equivalences on the fibers over $S$ and fits into a diagram as in the statement of the lemma. $\Box$
\end{proof}

By setting $K_{l_j} = \emptyset$ we obtain the following:

\begin{corollary}
\label{cor:mappingtree}
For any coCartesian fibration $X \longrightarrow \Omega[T]$ there exists a simplicial $\Omega(T)$-algebra $A$ and a commutative diagram
\[
\xymatrix{
M^+(A) \ar[rr]^\phi\ar[dr] & & X^\natural \ar[dl] \\
& \Omega[T]^\sharp & 
}
\]
such that the underlying map of $\phi$ induces categorical equivalences on the fibers over the colours of $\Omega[T]$.
\end{corollary}

The combinatorial situation occurring in the previous proof is one we'll have to consider several more times to carry out inductive arguments, so let us analyze it a little closer. We let $T$ be a tree with at least one vertex, we let $\alpha$ be a leaf corolla of this tree, whose leaves, root and vertex we label by $l_1, \ldots, l_n$, $r$ and $v$ respectively. Suppose we are given a simplicial $\Omega(T)$-algebra $A$. We denote the restriction of this algebra to $\Omega(\partial_v T)$ by $A'$ (if $T$ has only one vertex we take $\partial_v T$ to be the root of $T$ here). Now (identifying $\alpha$ with $C_n$) define an algebra $\Xi$ over $\Omega(\alpha)$ by
\begin{eqnarray*}
\Xi(l_j) & = & A(l_j) \quad \text{for $1 \leq j \leq n$} \\
\Xi(r) & = & \prod_{j=1}^n A(l_j)
\end{eqnarray*}
with the obvious algebra structure. One easily proves the following:

\begin{lemma}
\label{lemma:mappingtreetrick}
The map
\begin{equation*}
M(A') \coprod_{\Xi(r)} M(\Xi) \longrightarrow M(A)
\end{equation*}
is inner anodyne.
\end{lemma}

In the following lemmas we will use simplicial operads to prove several results. This is not strictly necessary; the proofs can also be carried out within the language of dendroidal sets. However, the use of simplicial operads is slightly more convenient here, since the spaces of operations of these carry a well-defined composition operation.

\begin{lemma}
\label{lemma:mappingtreetrick2}
Assume we are given the situation as described right above Lemma \ref{lemma:mappingtreetrick} and assume we are given a vertex $x_i \in A(l_i)$ for each $1 \leq i \leq n$. The tuple $(x_1, \ldots, x_n)$ together with the corolla $\alpha$ determine a corolla $\xi$ of $M(A)$, whose root induces a vertex $x$ of $A(r)$. Now let $(y_1, \ldots, y_m, z)$ be a tuple of vertices of $M(A)$ such that all these vertices are not contained in any of the $A(l_i)$. Then precomposing with $\xi$ induces a homotopy equivalence of simplicial sets
\begin{equation*}
\mathrm{hc}\tau_d(M(A))(x, y_1, \ldots, y_m; z) \longrightarrow \mathrm{hc}\tau_d(M(A))(x_1, \ldots, x_n, y_1, \ldots, y_m; z) 
\end{equation*}
\end{lemma}
\begin{proof}
By Lemma \ref{lemma:mappingtreetrick2} we may replace $M(A)$ by $M(A') \coprod_{\Xi(r)} M(\Xi)$. Observe that in this case the stated map is an isomorphism of simplicial sets. $\Box$
\end{proof}

\begin{lemma}
\label{lemma:mappingtreeoperadicequiv}
Let $X \longrightarrow \Omega[T]$ be a coCartesian fibration, let $A$ be a simplicial $\Omega(T)$-algebra and let
\begin{equation*}
\phi: M^+(A) \longrightarrow X^\natural
\end{equation*}
be a map compatible with the projections to $\Omega[T]$ which induces categorical equivalences on the fibers over $\Omega[T]$. Then the underlying map of $\phi$ is an operadic equivalence.
\end{lemma}
\begin{proof}
We work by induction on the size of $T$. If $T$ has no vertices the result is trivial. Therefore suppose $T$ has at least one vertex and fix a leaf corolla $\alpha$. We employ the notation described immediately above Lemma \ref{lemma:mappingtreetrick}, so that the leaves of $\alpha$ are $l_1, \ldots, l_n$. The inductive assumption is that the statement is valid for $\partial_v T$. \par 
We will check that for any tuple $(x_1, \ldots, x_N, y)$ of vertices of $M(A)$ the induced map
\begin{equation*}
\mathrm{hc}\tau_d(M(A))(x_1, \ldots, x_N; y) \longrightarrow \mathrm{hc}\tau_d(X)(\phi(x_1), \ldots, \phi(x_N); \phi(y))
\end{equation*}
is a homotopy equivalence of simplicial sets. Applying the inductive hypothesis, we see that we only have to check this in case the tuple $(x_1, \ldots, x_N)$ contains exactly one vertex in the fiber over each $l_i$ for $1 \leq i \leq n$ and $y$ is not contained in any such fiber. Reindexing if necessary, suppose that $x_i$ is a vertex over $l_i$ for $1 \leq i \leq n$. \par 
The degenerate 1-simplices at $x_1, \ldots, x_n$ and the corolla $\alpha$ determine a marked corolla $\xi$ of $M^+(A)$, whose root gives us a colour $x$ in the fiber over $r$. Now observe that we have a commutative diagram
\[
\xymatrix{
\mathrm{hc}\tau_d(M(A))(x, x_{n+1}, \ldots, x_N; y) \ar[r]\ar[d] & \mathrm{hc}\tau_d(M(A))(x_1, \ldots, x_N; y) \ar[d] \\
\mathrm{hc}\tau_d(X)(\phi(x), \phi(x_{n+1}), \ldots, \phi(x_N); \phi(y)) \ar[r] & \mathrm{hc}\tau_d(X)(\phi(x_1), \ldots, \phi(x_N); \phi(y))
}
\]
where the horizontal arrows are given by precomposing with $\xi$ and $\phi(\xi)$ respectively. Since $\phi(\xi)$ is a marked corolla of $X^\natural$, it is coCartesian and the bottom arrow is a homotopy equivalence. The left vertical arrow is a homotopy equivalence by the inductive hypothesis and the top horizontal arrow is a homotopy equivalence by Lemma \ref{lemma:mappingtreetrick2} above. We conclude that the right vertical map is a homotopy equivalence. $\Box$
\end{proof}

\begin{lemma}
\label{lemma:mappingtreemarkedequiv}
Let $p: X \longrightarrow \Omega[T]$ be a coCartesian fibration, let $A$ be a simplicial $\Omega(T)$-algebra and let
\begin{equation*}
\phi: M^+(A) \longrightarrow X^\natural
\end{equation*}
be a map compatible with the projections to $\Omega[T]$ such that the underlying map of $\phi$ is an operadic equivalence. Then $\phi$ is a marked equivalence in $\mathbf{dSets^+}/\Omega[T]$.
\end{lemma}
\begin{proof}
First observe that any map from a dendroidal set into $\Omega[T]$ is uniquely determined by what it does on colours. In particular, suppose we are given a map $Z \longrightarrow \Omega[T]$ and a map $f: X \otimes i_!(\Delta^n) \longrightarrow Z$. If the restrictions $X \otimes i_!(\{k\}) \longrightarrow Z$ are compatible with the projections to $\Omega[T]$ for all $0 \leq k \leq n$ then the map $f$ itself will be compatible with the projections to $\Omega[T]$. \par 
Now let $q: Z \longrightarrow \Omega[T]$ be any coCartesian fibration. We have a commutative diagram
\[
\xymatrix@C=40pt{
\mathrm{Map}^\flat_{\Omega[T]}(X^\natural, Z^\natural) \ar@{^(->}[r]\ar[d] & i^*\mathbf{Hom_{dSets}}(X, Z) \ar[d] \\
\mathrm{Map}^\flat_{\Omega[T]}(M^+(A), Z^\natural) \ar@{^(->}[r] & i^*\mathbf{Hom_{dSets}}(M(A), Z)
}
\]
The observation just made implies that the horizontal arrows are inclusions of full subcategories. Observe that, since $Z$ is an  $\infty$-operad, the right vertical arrow is a categorical equivalence by assumption. If we can show that any map of dendroidal sets $f: X \longrightarrow Z$ such that $f \circ \phi$ is in $\mathrm{Map}^\flat_{\Omega[T]}(M^+(A), Z^\natural)$ is itself contained in $\mathrm{Map}^\flat_{\Omega[T]}(X^\natural, Z^\natural)$, we can conclude that the left vertical arrow is also a categorical equivalence, establishing the lemma. \par 
Indeed, suppose $f$ is as described. Since $\phi$ is an operadic equivalence, any dendrex $\alpha$ of $X$ is in the image of $\phi$ up to an equivalence in $X$, i.e. there is a dendrex $\tilde \alpha$ of $M(A)$ such that $\phi(\tilde \alpha)$ is homotopic to $\alpha$. But any equivalence in $X$ projects to an equivalence in $\Omega[T]$, which is necessarily an identity. Also, $(f \circ \phi)(\tilde \alpha)$ will be homotopic to $f(\alpha)$. We obtain
\begin{equation*}
p(\alpha) = (p \circ \phi)(\tilde \alpha) = (q \circ f \circ \phi)(\tilde \alpha) = (q \circ f)(\alpha) 
\end{equation*}
which is what we were after. Compatibility of $f$ with markings can be shown in similar fashion, using the fact that coCartesian lifts are unique up to equivalence. $\Box$
\end{proof}

\subsection{Straightening over a tree}
Since any dendroidal set is a colimit of representables, it seems sensible to first prove Theorem \ref{thm:Grothendieckconstruction} in the case where $S$ is such a representable. In the next section we will deduce the full result from this by general arguments. \par 
Our first step is to consider the case where $S = \eta$. In this case we can identify both $\mathbf{dSets^+}/S$ and $\mathrm{Alg}_{\mathrm{hc}\tau_d(S)}(\mathbf{sSets^+})$ with $\mathbf{sSets^+}$. In this special case we use the abbreviated notation
\[
\xymatrix@R=40pt@C=40pt{
\mathcal{S}: \mathbf{sSets^+} \ar@<.5ex>[r] & \mathbf{sSets^+}: \mathcal{U} \ar@<.5ex>[l]
}
\]
for the Quillen adjunction induced by the straightening and unstraightening functors. Lurie proves the following result as Lemma 3.2.3.1 in \cite{htt}:

\begin{lemma}
\label{lemma:strovereta}
The Quillen pair $(\mathcal{S}, \mathcal{U})$ is a Quillen equivalence.
\end{lemma}

We will first use a standard trick to reduce proving that $(St^+_S, Un^+_S)$ is a Quillen equivalence to proving that the derived unit
\begin{equation*}
\mathrm{id} \longrightarrow RUn^+_S \circ LSt^+_S 
\end{equation*}
is a weak equivalence. Here $LSt^+_S$ (resp. $RUn^+_S$) denotes the left (resp. right) derived functor of the straightening (resp. unstraightening) functor. Indeed, a priori we need to prove that both the derived unit and the derived counit
\begin{equation*}
LSt^+_S \circ RUn^+_S \longrightarrow \mathrm{id}
\end{equation*}
are weak equivalences. Now assume that $RUn^+_S$ detects weak equivalences. Then in order to prove that the counit is a weak equivalence it is sufficient to prove that
\begin{equation*}
RUn^+_S \circ LSt^+_S \circ RUn^+_S \longrightarrow RUn^+_S 
\end{equation*}
is a weak equivalence. But this will follow if the unit is a weak equivalence. We therefore need to prove the following fact:

\begin{lemma}
The functor $RUn^+_S$ detects weak equivalences.
\end{lemma}
\begin{proof}
Suppose we are given a map $f: A \longrightarrow B$ between fibrant objects of $\mathrm{Alg}_{\mathrm{hc}\tau_d(S)}(\mathbf{sSets^+})$ and suppose that $Un^+_S(f)$ is a marked equivalence in $\mathbf{dSets^+}/S$.  We need to show that for each colour $s$ of $S$ the induced map $f_s: A(s) \longrightarrow B(s)$ is a marked equivalence in $\mathbf{sSets^+}$. Note that $A(s)$ and $B(s)$ are fibrant. Lemma \ref{lemma:strovereta} above tells us that $\mathcal{U}$ detects weak equivalences between fibrant objects of $\mathbf{sSets^+}$, so we might as well prove that $\mathcal{U}(f_s)$ is a marked equivalence. If we denote the inclusion $\{s\} \longrightarrow S$ by $\iota_s$ then 
\begin{equation*}
\mathcal{U}(f_s) = (\mathcal{U} \circ \mathrm{hc}\tau_d(\iota_s)^\ast)(f) 
\end{equation*}
The adjoint of Proposition \ref{prop:straightening} tells us that we have a natural isomorphism of functors
\begin{equation*}
\mathcal{U} \circ \mathrm{hc}\tau_d(\iota_s)^\ast \simeq \iota_s^\ast \circ Un^+_S
\end{equation*}
Now observe that $(\iota_s^\ast \circ Un^+_S)(f)$ is a weak equivalence since $Un^+_S(f)$ is a weak equivalence between fibrant objects and $\iota_s^\ast$ is a right Quillen functor. $\Box$
\end{proof}

For any $X \in \mathbf{dSets^+}/S$ and a colour $s$ of $S$, the counit of the adjunction $((\iota_s)_!, \iota_s^\ast)$ gives us a map
\begin{equation*}
(\iota_s)_! X_s \longrightarrow X
\end{equation*}
Applying the straightening functor and then using Proposition \ref{prop:straightening} we get a map
\begin{equation*}
(\mathrm{hc}\tau_d(\iota_s)_! \circ \mathcal{S})(X_s) \longrightarrow St^+_S(X)
\end{equation*}
which by adjunction yields a natural map
\begin{equation*}
\psi^X_s: \mathcal{S}(X_s) \longrightarrow St^+_S(X)(s)
\end{equation*}
Before stating the next lemma, recall that for $A$ a simplicial $\Omega(T)$-algebra, the fiber of $M^+(A)$ over a colour $c$ of $\Omega(T)$ is isomorphic to $A(c)^\flat$. The following will be very convenient:

\begin{lemma}
\label{lemma:stmappingtree}
Let $A$ be a simplicial $\Omega(T)$-algebra. Then for any colour $c$ of $T$, the map
\begin{equation*}
\psi^{M^+(A)}_c: \mathcal{S}(A(c)^\flat) \longrightarrow St^+_S(M^+(A))(c)
\end{equation*}
is a marked equivalence in $\mathbf{sSets^+}$.
\end{lemma}
\begin{proof}
The statement is trivial in case $T = \eta$. Hence assume $T$ has at least one vertex. We proceed by induction on the number of vertices of $T$. We employ the notation described immediately above Lemma \ref{lemma:mappingtreetrick}, i.e. assume $\alpha$ is a leaf corolla of $T$ etc. The map $\psi^{M^+(A)}_{l_i}$ is an isomorphism for $1 \leq i \leq n$, so from now on assume $c$ is a colour of $T$ other than any of the leaves $l_i$. We have a commutative diagram
\[
\xymatrix{
\mathcal{S}(A(c)^\flat) \ar[rr]\ar[dr] & & St^+_{\Omega[T]}(M^+(A))(c) \\
& St^+_{\Omega[T]}(M^+(A'))(c) \ar[ur] &
}
\]
Letting $\iota: \partial_v T \longrightarrow T$ denote the face inclusion, Proposition \ref{prop:straightening} tells us that
\begin{equation*}
St^+_{\Omega[T]}(M^+(A'))(c) \simeq St^+_{\partial_v \Omega[T]}(M^+(A'))(c)
\end{equation*}
so that the map
\begin{equation*}
\mathcal{S}(A(c)^\flat) \longrightarrow St^+_{\Omega[T]}(M^+(A))(c)
\end{equation*}
is a marked equivalence in $\mathbf{sSets^+}$ by the inductive hypothesis. Therefore it suffices to show that the map 
\begin{equation*}
St^+_{\Omega[T]}(M^+(A'))(c) \longrightarrow St^+_{\Omega[T]}(M^+(A))(c)
\end{equation*}
is a marked equivalence. Applying Lemma \ref{lemma:mappingtreetrick} we see that we have a homotopy pushout
\[
\xymatrix{
St^+_{\Omega[T]}(\Xi(r)^\flat)(c) \ar[r]\ar[d] & St^+_{\Omega[T]}(M^+(\Xi))(c) \ar[d] \\
St^+_{\Omega[T]}(M^+(A'))(c) \ar[r] & St^+_{\Omega[T]}(M^+(A))(c)
}
\]
Replacing $A$ by an equivalent mapping tree if necessary, we may in fact assume the left vertical map is a cofibration. (Indeed one checks, using left properness of the Moerdijk-Cisinski model structure, that an equivalence of algebras induces an equivalence of mapping trees.) In order to show that the bottom horizontal map is a marked equivalence it now suffices to show that the top horizontal map is so. If we let $\kappa: \alpha \longrightarrow T$ denote the inclusion of our leaf corolla, we actually have
\begin{eqnarray*}
St^+_{\Omega[T]}(\Xi(r)^\flat) & \simeq & \mathrm{hc}\tau_d(\kappa)_!\bigl(St^+_{\Omega[\alpha]}(\Xi(r)^\flat)\bigr) \\
St^+_{\Omega[T]}(M^+(\Xi)) & \simeq & \mathrm{hc}\tau_d(\kappa)_!\bigl(St^+_{\Omega[\alpha]}(M^+(\Xi))\bigr)
\end{eqnarray*}
It will suffice to show that
\begin{equation*}
St^+_{\Omega[\alpha]}(\Xi(r)^\flat)(r) \longrightarrow St^+_{\Omega[\alpha]}(M^+(\Xi))(r)
\end{equation*}
is a marked equivalence. This is a quick computation: the left-hand side is simply
\begin{equation*}
\mathcal{S}(\prod_{i=1}^n A(l_i)^\flat)
\end{equation*}
and the right-hand side can be seen to be equivalent to
\begin{equation*}
\mathcal{S}\bigl(\prod_{i=1}^n A(l_i)^\flat \bigr) \times ((\Delta^1)^{\times n})^\sharp
\end{equation*}
This concludes the proof. $\Box$  
\end{proof}

We are now ready to prove the following result:

\begin{proposition}
\label{prop:strovertree}
Set $S = \Omega[T]$ for some tree $T \in \Omega$. Then the adjunction $(St^+_S, Un^+_S)$ is a Quillen equivalence.
\end{proposition}
\begin{proof}
As reasoned above, it suffices to show that the derived unit
\begin{equation*}
\mathrm{id} \longrightarrow RUn^+_S \circ LSt^+_S 
\end{equation*}
is a weak equivalence for each $X \in \mathbf{dSets^+}/S$. Note that every object of $\mathbf{dSets^+}/S$ is cofibrant, since in our case $S$ is normal. Hence we can simply identify $LSt^+_S$ with $St^+_S$. Since the composite $RUn^+_S \circ St^+_S$ preserves weak equivalences, we can without loss of generality take a fibrant replacement for $X$. Therefore assume $X$ to be fibrant. \par 
Since $X$ is fibrant, we can apply Proposition \ref{prop:mappingtrees} to deduce the existence of a simplicial $\Omega(T)$-algebra $A$ and a marked equivalence
\begin{equation*}
M^+(A) \longrightarrow X
\end{equation*}
Hence it suffices to prove that
\begin{equation*}
M^+(A) \longrightarrow (RUn^+_S \circ St^+_S)(M^+(A)) 
\end{equation*}
is a marked equivalence. Since the object on the right is fibrant, Lemmas \ref{lemma:mappingtreeoperadicequiv} and \ref{lemma:mappingtreemarkedequiv} tell us that it is sufficient to prove that the stated map induces categorical equivalences on the fibers. This is an easy computation; indeed, first note that using Proposition \ref{prop:straightening} we get (for a colour $s$ of $S$)
\begin{eqnarray*}
\bigl((RUn^+_S \circ St^+_S)(M^+(A))\bigr)_s & \simeq & (\iota_s^\ast \circ RUn^+_S \circ St^+_S)(M^+(A)) \\
& \simeq & (\mathcal{U} \circ \mathrm{hc}\tau_d(\iota_s)^\ast \circ R \circ St^+_S)(M^+(A)) 
\end{eqnarray*}
where $R$ denotes fibrant replacement in $\mathrm{Alg}_{\mathrm{hc}\tau_d(S)}(\mathbf{sSets^+})$. By Lemma \ref{lemma:strovereta} the map
\begin{eqnarray*}
A(s)^\flat \simeq M^+(A)_s & \longrightarrow & (\mathcal{U} \circ \mathrm{hc}\tau_d(\iota_s)^\ast \circ R \circ St^+_S)(M^+(A)) 
\end{eqnarray*}
is a weak equivalence if and only if the adjoint map
\begin{eqnarray*}
\mathcal{S}(A(s)^\flat) & \longrightarrow &  (\mathrm{hc}\tau_d(\iota_s)^\ast \circ R \circ St^+_S)(M^+(A)) \\
& \simeq & (R \circ St^+_S)(M^+(A))(s)
\end{eqnarray*}
is a weak equivalence. This map factors as
\begin{equation*}
\mathcal{S}(A(s)^\flat) \longrightarrow St^+_S(M^+(A))(s) \longrightarrow (R \circ St^+_S)(M^+(A))(s)
\end{equation*}
The first map is a marked equivalence by Lemma \ref{lemma:stmappingtree}, the second by definition. We have now shown that the map of fibers 
\begin{equation*}
\phi_s: A(s)^\flat \longrightarrow (RUn^+_S \circ St^+_S)(M^+(A))_s 
\end{equation*} 
is a marked equivalence in $\mathbf{sSets^+}$. For any $\infty$-category $Z$ we have isomorphisms
\begin{eqnarray*}
\mathrm{Map}^\flat(A(s)^\flat, Z^\natural) & \simeq & A(s)^Z \\
\mathrm{Map}^\flat((RUn^+_S \circ St^+_S)(M^+(A))_s, Z^\natural) & \simeq & u((RUn^+_S \circ St^+_S)(M^+(A))_s)^Z
\end{eqnarray*}
In the first line we use the adjunction between $(-)^\flat$ and $u$, in the second line the fact that $(RUn^+_S \circ St^+_S)(M^+(A))_s$ is an $\infty$-category with equivalences marked, so that preservation of markings when mapping into $Z^\natural$ is automatic. We now conclude that the underlying map of $\phi_s$ is a categorical equivalence. $\Box$
\end{proof}

\subsection{Straightening in general}
In this section we will finally prove Theorem \ref{thm:Grothendieckconstruction} using Proposition \ref{prop:strovertree} and some formal arguments. The first we need is the following easy lemma:

\begin{lemma}
\label{lemma:normaldsets}
Suppose $\mathcal{C}$ is a collection of dendroidal sets satisfying the following conditions:
\begin{itemize}
\item[(1)] Every representable $\Omega[T]$ is contained in $\mathcal{C}$
\item[(2)] $\mathcal{C}$ is stable under coproducts
\item[(3)] If we are given a pushout square
\[
\xymatrix{
X \ar[d]_j \ar[r] & Y \ar[d]\\
X' \ar[r] & Y'\\
}
\]
such that $X, X', Y \in \mathcal{C}$ and $j$ is a normal monomorphism, then $Y' \in C$
\item[(4)] If we are given a sequence of normal monomorphisms
\begin{equation*}
X(0) \longrightarrow X(1) \longrightarrow \cdots
\end{equation*}
then the colimit $\varinjlim X(i)$ belongs to $\mathcal{C}$
\end{itemize}
Then $\mathcal{C}$ contains all normal dendroidal sets.
\end{lemma}
\begin{proof}
Let $X$ be a normal dendroidal set. We wish to show that $X \in \mathcal{C}$. First use that $X$ admits a normal skeletal filtration (see \cite{dendroidalsets})
\begin{equation*}
\mathrm{sk}_0(X) \longrightarrow \mathrm{sk}_1(X) \longrightarrow \cdots
\end{equation*}
such that $X \simeq \varinjlim \mathrm{sk}_i(X)$. Using (4) we reduce to showing that $\mathrm{sk}_n(X) \in \mathcal{C}$ for all $n$. We now proceed by induction on $n$. We have a pushout diagram
\[
\xymatrix{
\coprod_{(T,t)} \partial \Omega[T] \ar[d]\ar[r] & \mathrm{sk}_{n-1}(X) \ar[d] \\ 
\coprod_{(T,t)} \Omega[T] \ar[r] & \mathrm{sk}_n(X)
}
\]
where the coproduct is taken over all pairs $(T,t)$ where $T$ is a tree with $n$ vertices and $t$ is a non-degenerate $T$-dendrex of $X$. The two dendroidal sets in the top row are in $\mathcal{C}$ by the inductive hypothesis, the bottom left one is in $\mathcal{C}$ by assumptions (1) and (2). Assumption (3) now guarantees that $\mathrm{sk}_n(X) \in \mathcal{C}$. $\Box$
\end{proof}

Recall that the unstraightening functor $Un^+_S$ is a simplicial functor. In order to prove Theorem \ref{thm:Grothendieckconstruction}, it suffices (by Lemma \ref{lemma:equivsimpcat}) to prove that $Un^+_S$ induces a weak equivalence of simplicial categories
\begin{equation*}
(\mathrm{Alg}_{\mathrm{hc}\tau_d(S)}(\mathbf{sSets^+}))^\circ \longrightarrow (\mathbf{dSets^+}/S)^\circ
\end{equation*}
This is what we will do.

\begin{proof}[Proof of Theorem \ref{thm:Grothendieckconstruction}]
Let $\mathcal{C}$ be the collection of all normal dendroidal sets for which the adjunction $(St^+_S, Un^+_S)$ is a Quillen equivalence. We wish to show that $\mathcal{C}$ satisfies the hypotheses of Lemma \ref{lemma:normaldsets} in order to conclude that $\mathcal{C}$ is precisely the collection of all normal dendroidal sets. \par 
The collection $\mathcal{C}$ satisfies (1) by Proposition \ref{prop:strovertree} and Lemma \ref{lemma:equivsimpcat}. It is easy to verify that it satisfies (2). Now let $(\mathrm{Alg}_{\mathrm{hc}\tau_d(S)}(\mathbf{sSets^+}))_f$ denote the full simplicial subcategory on fibrant objects of the category in parentheses and let $W_S$ denote the class of weak equivalences in this category. Also, let $\tilde{W}_S$ denote the class of weak equivalences in $(\mathbf{dSets^+}/S)^\circ$ which, by Proposition \ref{prop:markedequivalence}, is the class of morphisms which induce categorical equivalences on the fibers over $S$. We have a commutative square of simplicial categories as follows:
\[
\xymatrix@C=60pt{
(\mathrm{Alg}_{\mathrm{hc}\tau_d(S)}(\mathbf{sSets^+}))^\circ \ar[r]^{Un^+_S}\ar[d] & (\mathbf{dSets^+}/S)^\circ \ar[d]^{\psi_S} \\
(\mathrm{Alg}_{\mathrm{hc}\tau_d(S)}(\mathbf{sSets^+}))_f[W_S^{-1}] \ar[r]_{\varphi_S} & (\mathbf{dSets^+}/S)^\circ [\tilde{W}_S^{-1}]
}
\]
The left vertical functor is a weak equivalence of simplicial categories by Lemma \ref{lemma:combinlocalization}. Also, the functor $\psi_S$ is a weak equivalence of simplicial categories by Proposition \ref{prop:dwyerkan} below. Hence it suffices to show that $\varphi_S$ is an equivalence of simplicial categories. \par 
Now suppose we are given a pushout diagram of normal dendroidal sets
\[
\xymatrix{
X \ar[d]_f\ar[r]^g & Y \ar[d]^k \\
X' \ar[r]_l & Y'
}
\] 
such that $X, X', Y \in \mathcal{C}$ and such that the left vertical map is a normal monomorphism. Lemma \ref{lemma:combinlocalization2} and the assumption on $X$, $X'$ and $Y$ tell us that the following diagram is a homotopy pullback:
\[
\xymatrix{
(\mathrm{Alg}_{\mathrm{hc}\tau_d(Y')}(\mathbf{sSets^+}))_f[W_{Y'}^{-1}] \ar[r]\ar[d] & (\mathbf{dSets^+}/Y)^\circ [\tilde{W}_Y^{-1}] \ar[d] \\
(\mathbf{dSets^+}/X')^\circ [\tilde{W}_{X'}^{-1}] \ar[r] & (\mathbf{dSets^+}/X)^\circ [\tilde{W}_X^{-1}]
}
\]
The top and left arrows in this diagram factor through the map
\begin{equation*}
\varphi_{Y'}: (\mathrm{Alg}_{\mathrm{hc}\tau_d(Y')}(\mathbf{sSets^+}))_f[W_{Y'}^{-1}] \longrightarrow (\mathbf{dSets^+}/Y')^\circ [\tilde{W}_{Y'}^{-1}]
\end{equation*}
Using Lemma \ref{lemma:homlimsimplcat} we now deduce that $\varphi_{Y'}$ is a weak equivalence if and only if for any pair of objects $A$, $B$ of $(\mathbf{dSets^+}/Y')^\circ [\tilde{W}_{Y'}^{-1}]$ the induced diagram of simplicial mapping objects
\[
\xymatrix{
\mathrm{Map}_{(\mathbf{dSets^+}/Y')^\circ [\tilde{W}_{Y'}^{-1}]}(A,B) \ar[r]\ar[d] & \mathrm{Map}_{(\mathbf{dSets^+}/Y)^\circ [\tilde{W}_{Y}^{-1}]}(k^*A,k^*B) \ar[d] \\
\mathrm{Map}_{(\mathbf{dSets^+}/X')^\circ [\tilde{W}_{X'}^{-1}]}(l^*A,l^*B) \ar[r] & \mathrm{Map}_{(\mathbf{dSets^+}/X)^\circ [\tilde{W}_{X}^{-1}]}(f^*l^*A,f^*l^*B)
}
\]
is a homotopy pullback in the Quillen model structure on simplicial sets. Since $\psi_{X}$, $\psi_{X'}$, $\psi_{Y}$ and $\psi_{Y'}$ are equivalences of simplicial categories, we may replace this diagram with the equivalent diagram
\[
\xymatrix{
\mathrm{Map}^\sharp_{Y'}(A,B) \ar[r]\ar[d] & \mathrm{Map}^\sharp_{Y}(k^*A,k^*B) \ar[d] \\
\mathrm{Map}^\sharp_{X'}(l^*A,l^*B) \ar[r] & \mathrm{Map}^\sharp_{X}(f^*l^*A,f^*l^*B) 
}
\]
which is a pullback square. In this diagram all objects are fibrant and the bottom horizontal arrow is a Kan fibration by Lemma \ref{lemma:restrKanfibration}. Hence it is a homotopy pullback square. We have now established that $\mathcal{C}$ satisfies (3). Verifying that it satisfies (4) is done in a similar fashion, now applying Lemmas \ref{lemma:combinlocalization2} and \ref{lemma:homlimsimplcat} in the case of sequential colimits instead of pushouts. $\Box$
\end{proof}

In the proof above we used the following two results. The first goes back to Dwyer and Kan \cite{dwyerkan}:

\begin{proposition}
\label{prop:dwyerkan}
Let $\mathbf{C}$ be a simplicial category and let $W$ be a class of equivalences in $\mathbf{C}$ (equivalence here in the simplicial sense). Then the localization
\begin{equation*}
\mathbf{C} \longrightarrow \mathbf{C}[W^{-1}]
\end{equation*}
is a weak equivalence of simplicial categories.
\end{proposition}

\begin{lemma}
\label{lemma:restrKanfibration}
Suppose we are given a normal monomorphism $\tilde{S} \longrightarrow S$ of dendroidal sets and coCartesian fibrations $X \longrightarrow S$ and $Y \longrightarrow S$. Assume $X$ is normal. Define  $\tilde X = X \times_{S} \tilde S$ and $\tilde Y = Y \times_{S} \tilde S$. Then the induced map
\begin{equation*}
\mathrm{Map}^\sharp_S(X^\natural, Y^\natural) \longrightarrow \mathrm{Map}^\sharp_{\tilde S}(\tilde{X}^\natural, \tilde{Y}^\natural)
\end{equation*}
is a Kan fibration.
\end{lemma}
\begin{proof}
Since both mapping objects are in fact Kan complexes, it suffices to show that the stated map is a left fibration. So, let $A \subseteq B$ be a left anodyne inclusion of simplicial sets. We want to show that there exists a lift in the diagram
\[
\xymatrix{
A \ar[r]\ar[d] & \mathrm{Map}^\sharp_S(X^\natural, Y^\natural) \ar[d] \\
B \ar[r] & \mathrm{Map}^\sharp_{\tilde S}(\tilde{X}^\natural, \tilde{Y}^\natural)
}
\]
For this it suffices to show (by adjunction) that there exists a lift in the diagram
\[
\xymatrix{
i_!(A)^\sharp \otimes X^\natural \coprod_{i_!(A)^\sharp \otimes \tilde{X}^\natural} i_!(B)^\sharp \otimes \tilde{X}^\natural \ar[r]\ar[d] & Y^\natural \ar[d] \\
i_!(B)^\sharp \otimes X^\natural \ar[r] & S^\sharp
}
\]
Observe that the left vertical map is the smash product of the marked anodyne $i_!(A)^\sharp \subseteq i_!(B)^\sharp$ with the cofibration $\tilde{X}^\natural \subseteq X^\natural$ and is hence marked anodyne by Proposition \ref{prop:smashproduct}. By Proposition \ref{prop:lifting} we conclude that the lift exists. $\Box$
\end{proof}

\newpage

\section{Applications}
\subsection{Naturality of the coCartesian model structure}
We will investigate the behaviour of the coCartesian model structure on $\mathbf{dSets}^+/S$ with respect to a map $\phi: S \longrightarrow T$. Observe that any such map provides a simplicial adjunction
\[
\xymatrix@R=40pt@C=40pt{
\phi_!: \mathbf{dSets}^+/S \ar@<.5ex>[r] & \mathbf{dSets}^+/T: \phi^* \ar@<.5ex>[l]
}
\]
where the left adjoint $\phi_!$ is given by composition with $\phi$ and the right adjoint $\phi^*$ by pulling back along $\phi$.

\begin{proposition}
The adjunction $(\phi_!, \phi^*)$ is a simplicial Quillen adjunction.
\end{proposition}
\begin{proof}
Clearly $\phi_!$ preserves cofibrations. It rests us to check that $\phi^*$ preserves fibrant objects. This follows from the fact that the pullback of a map having the right lifting property with respect to marked anodynes will again have the right lifting property with respect to marked anodynes. $\Box$ 
\end{proof}

\begin{theorem}
\label{thm:naturalitycoCmodstruct}
If $\phi: S \longrightarrow T$ is an operadic equivalence of dendroidal sets then the adjunction $(\phi_!, \phi^*)$ is a simplicial Quillen equivalence.
\end{theorem}
\begin{proof}
First assume $S$ and $T$ are cofibrant. Then we have a diagram of left Quillen functors
\[
\xymatrix{
\mathbf{dSets}^+/S \ar[d]\ar[r] & \mathrm{Alg}_{\mathrm{hc}\tau_d(S)}(\mathbf{sSets^+}) \ar[d]\\
\mathbf{dSets}^+/T \ar[r] & \mathrm{Alg}_{\mathrm{hc}\tau_d(T)}(\mathbf{sSets^+})
}
\]
in which the horizontal functors give Quillen equivalences by Theorem \ref{thm:Grothendieckconstruction} and the right vertical functor gives a Quillen equivalence by the results of Berger and Moerdijk \cite{bergermoerdijk1}\cite{bergermoerdijk2}. We conclude that the left vertical functor must also be part of a Quillen equivalence. \par 
For general $S$ and $T$ we have a commutative square of left Quillen functors as follows:
\[
\xymatrix{
\mathbf{dSets}^+/(S \times E_\infty) \ar[r]\ar[d] & \mathbf{dSets}^+/S \ar[d] \\
\mathbf{dSets}^+/(T \times E_\infty) \ar[r] & \mathbf{dSets}^+/T
}
\]  
We know the horizontal arrows are part of Quillen equivalences and we just proved the left vertical arrow is part of a Quillen equivalence. This shows the right vertical functor also gives a Quillen equivalence. $\Box$
\end{proof}

\begin{corollary}
An operadic equivalence $\phi: S \longrightarrow T$ induces an equivalence of $\infty$-categories
\begin{equation*}
\mathrm{hc}N(\phi_!): \mathbf{coCart}(S) \longrightarrow \mathbf{coCart}(T)
\end{equation*}
\end{corollary}

\begin{remark}
Note that this result provides us with a technical advantage when working with algebras over an $\infty$-operad $S$ as opposed to homotopy algebras over simplicial operads. Indeed, suppose $P$ is a simplicial operad. By definition a homotopy algebra over $P$ in $\mathbf{sSets^+}$ is an algebra in $\mathbf{sSets^+}$ over a cofibrant resolution of $P$. By Theorem \ref{thm:Grothendieckconstruction} and the previous corollary the $\infty$-category of such algebras is equivalent to the $\infty$-category $\mathbf{coCart}(\mathrm{hc}N_d(P))$. In other words, when working with homotopy algebras over a simplicial operad we first have to pass to a cofibrant resolution of our operad; for coCartesian fibrations over a dendroidal set it is no longer necessary to pass to such a resolution.
\end{remark}

The previous corollary also justifies the following definition.

\begin{definition}
A \emph{symmetric monoidal $\infty$-category} is a coCartesian fibration over $N_d(\mathrm{Comm})$. We define the $\infty$-category of symmetric monoidal $\infty$-categories by
\begin{equation*}
\mathbf{Cat}^\otimes_\infty := \mathbf{coCart}(N_d(\mathrm{Comm}))
\end{equation*}
\end{definition}

\begin{remark}
By our results, the $\infty$-category $\mathbf{Cat}^\otimes_\infty$ is equivalent to the category of $E_\infty$-algebras in $\infty$-categories.
\end{remark}

\subsection{Topological algebras and left fibrations}
\label{subsection:leftfibrations}
We will now discuss how our results specialize to the setting of algebras over an $\infty$-operad taking values in spaces or, equivalently, $\infty$-groupoids and prove the results claimed in section \ref{section:leftfib}. Observe that there is a simplicial adjunction
\[
\xymatrix@R=40pt@C=40pt{
u: \mathbf{dSets}^+/S \ar@<.5ex>[r] & \mathbf{dSets}/S: (-)^\sharp \ar@<.5ex>[l]
}
\]
We wish to use this adjunction to transfer the coCartesian model structure on $\mathbf{dSets}^+/S$ to a model structure on $\mathbf{dSets}/S$, which will turn out to be the \emph{covariant model structure} described in section \ref{section:leftfib}. Since the functor $(-)^\sharp$ is full and faithful, one might also interpret this procedure as restricting the coCartesian model structure along $(-)^\sharp$. \par 
For later use, we observe that the functor $(-)^\sharp$ actually also admits a right adjoint of its own. Indeed, define
\begin{equation*}
\mathpzc{k}: \mathbf{dSets}^+/S \longrightarrow \mathbf{dSets}/S
\end{equation*}
by letting $\mathpzc{k}(X, \mathcal{E}_X)$ be the dendroidal subset of $X$ consisting of the dendrices all of whose corollas are in $\mathcal{E}_X$. Then $\mathpzc{k}$ is right adjoint to $(-)^\sharp$. \par 
We will need the following fact:
\begin{lemma}
\label{lemma:maxKanofcoCfib}
If $p: X \longrightarrow S$ is a coCartesian fibration, then $\tilde p: \mathpzc{k}(X^\natural) \longrightarrow S$ is a left fibration.
\end{lemma}
\begin{proof}
First we wish to show that $\mathpzc{k}(X^\natural) \longrightarrow S$ is an inner fibration. Let $T$ be a tree with at least two vertices and let $e$ be an inner edge of $T$. Suppose we are given a lifting problem
\[
\xymatrix{
\Lambda^e[T] \ar[d]\ar[r] & \mathpzc{k}(X^\natural) \ar[d] \\
\Omega[T] \ar[r]\ar@{-->}[ur] & S
}
\]
We can extend this diagram and find a map as indicated by the dotted arrow below:
\[
\xymatrix{
\Lambda^e[T] \ar[d]\ar[r] & \mathpzc{k}(X^\natural) \ar[d]\ar[r] & X \ar[dl] \\
\Omega[T] \ar[r]\ar@{-->}[urr] & S & 
}
\]
If $T$ has 3 or more vertices, then all corollas of $\Omega[T]$ are already contained in $\Lambda^e[T]$, so the dotted arrow must factor over $\mathpzc{k}(X^\natural)$. If $T$ has 2 vertices, the dotted arrow must factor over $\mathpzc{k}(X^\natural)$ since the map $X^\natural \longrightarrow S^\sharp$ has the right lifting property with respect to marked anodynes of type (3), or said more informally, the set of marked corollas of $X^\natural$ is closed under composition. It is now easy to see that $\tilde p$ is a coCartesian fibration such that every corolla of $\mathpzc{k}(X^\natural)$ is $\tilde p$-coCartesian. In other words, $\tilde p$ is a left fibration. $\Box$
\end{proof}

\begin{definition}
We call a map $f$ in $\mathbf{dSets}/S$ a \emph{covariant equivalence} (resp. a \emph{covariant fibration}) if $f^\sharp$ is a coCartesian equivalence (resp. a fibration in the coCartesian model structure). 
\end{definition}
\begin{remark}
Be aware that a fibration in the coCartesian model structure is not the same thing as a coCartesian fibration.
\end{remark}

\begin{theorem}
\label{thm:covmodelstruct2}
There exists a left proper, simplicial, combinatorial model structure on $\mathbf{dSets}/S$ in which the weak equivalences (resp. fibrations) are the covariant equivalences (resp. covariant fibrations). Furthermore, the cofibrations in this model structure are the normal monomorphisms.
\end{theorem}
\begin{proof}
To show that the covariant equivalences and fibrations indeed define a model structure $\mathbf{dSets}/S$, we have to show that the adjunction between $u$ and $(-)^\sharp$ satisfies the requirements for transfer \cite{bergermoerdijk1}\cite{crans}. It suffices to show the following two things:
\begin{itemize}
\item[(i)] $(-)^\sharp$ preserves filtered colimits
\item[(ii)] For any trivial cofibration $f$ in $\mathbf{dSets}^+/S$, the map $u(f)$ is a covariant equivalence
\end{itemize}
Property (i) is obvious: $(-)^\sharp$ admits a right adjoint, so in fact it preserves all colimits. Now suppose $f: A \longrightarrow B$ is a trivial cofibration in $\mathbf{dSets}^+/S$. For convenience of notation, assume $A$ and $B$ are normal. Otherwise, we would first just take normalizations. We need to check that for any coCartesian fibration $Z \longrightarrow S$, the induced map
\begin{equation*}
\mathrm{Map}^\sharp_S(u(B)^\sharp, Z^\natural) \longrightarrow \mathrm{Map}^\sharp_S(u(A)^\sharp, Z^\natural)
\end{equation*}
is a weak homotopy equivalence of simplicial sets. Taking adjoints, this is the same as showing that
\begin{equation*}
\mathrm{Map}^\sharp_S(u(B), \mathpzc{k}(Z^\natural)) \longrightarrow \mathrm{Map}^\sharp_S(u(A), \mathpzc{k}(Z^\natural))
\end{equation*}
is a weak equivalence, which in turn is the same as showing that
\begin{equation*}
\mathrm{Map}^\sharp_S(B, \mathpzc{k}(Z^\natural)^\sharp) \longrightarrow \mathrm{Map}^\sharp_S(A, \mathpzc{k}(Z^\natural)^\sharp)
\end{equation*}
is a weak equivalence. This follows from the fact that $\mathpzc{k}(Z) \longrightarrow S$ is a left fibration (so in particular a coCartesian fibration) by Lemma \ref{lemma:maxKanofcoCfib} and the assumption that $A \longrightarrow B$ is a coCartesian equivalence. \par 
Now let us show that the cofibrations of this model structure are precisely the normal monomorphisms. Since every normal monomorphism is in the image of $u$ and $u$ is left Quillen by construction of the model structure on $\mathbf{dSets}/S$ we are considering, we conclude that every normal monomorphism is a cofibration in $\mathbf{dSets}/S$. For the converse, suppose $f: A \longrightarrow B$ is a cofibration in $\mathbf{dSets}/S$. If we can show $f^\sharp$ is a cofibration in the coCartesian model structure we can conclude that $f$ is a normal monomorphism. Suppose $X \longrightarrow Y$ is a trivial fibration in $\mathbf{dSets}^+/S$. We want to solve lifting problems of the form
\[
\xymatrix{
A^\sharp \ar[d]\ar[r] & X \ar[d] \\
B^\sharp \ar[r]\ar@{-->}[ur] & Y
}
\]
Note that the horizontal maps admit factorizations as follows:
\[
\xymatrix{
A^\sharp \ar[d]\ar[r] & \mathpzc{k}(Y)^\sharp \times_Y X \ar[d]\ar[r]& X \ar[d] \\
B^\sharp \ar[r] & \mathpzc{k}(Y)^\sharp \ar[r] & Y
}
\]
The middle vertical map is the pullback of a trivial fibration and hence itself a trivial fibration. In particular, it is surjective. Using that it also has the right lifting property with respect to cofibrations of the form $\Omega[C_n]^\flat \longrightarrow \Omega[C_n]^\sharp$ we conclude that
\begin{equation*}
\mathpzc{k}(Y)^\sharp \times_Y X = \bigl(u(\mathpzc{k}(Y)^\sharp \times_Y X)\bigr)^\sharp
\end{equation*} 
or in other words, that every corolla of $\mathpzc{k}(Y)^\sharp \times_Y X$ is marked. We see that the middle vertical map is in fact in the image of the functor $(-)^\sharp$. It follows that
\begin{equation*}
u(\mathpzc{k}(Y)^\sharp \times_Y X) \longrightarrow u(\mathpzc{k}(Y))
\end{equation*} 
is a trivial fibration, so we obtain a lift as follows:
\[
\xymatrix{
A \ar[r]\ar[d] & u(\mathpzc{k}(Y)^\sharp \times_Y X) \ar[d] \\
B \ar[r]\ar[ur] & \mathpzc{k}(Y)
}
\]
The composition
\begin{equation*}
B^\sharp \longrightarrow \mathpzc{k}(Y)^\sharp \times_Y X \longrightarrow X
\end{equation*}
now gives a lift in the original lifting problem, completing the proof that $f$ is a normal monomorphism. \par 
We conclude by showing left properness. Suppose we are given a pushout square in $\mathbf{dSets}/S$
\[
\xymatrix{
A \ar[d]\ar[r] & C \ar[d] \\
B \ar[r] & D
}
\]
in which the left vertical map is a cofibration and the top horizontal map is a covariant equivalence. The diagram
\[
\xymatrix{
A^\sharp \ar[d]\ar[r] & C^\sharp \ar[d] \\
B^\sharp \ar[r] & D^\sharp
}
\]
is still a pushout square, by definition the top horizontal map is a coCartesian equivalence and by what we've just shown the left vertical map is a cofibration in the coCartesian model structure as well. By left properness of the coCartesian model structure, the bottom horizontal map is a coCartesian equivalence and so, by definition, the bottom horizontal map in the first square is a covariant equivalence. $\Box$
\end{proof}

From now on we will refer to the model structure of Theorem \ref{thm:covmodelstruct2} as the covariant model structure. 

\begin{remark}
Observe that a map $f$ in $\mathbf{dSets}/S$ is a cofibration (resp. weak equivalence, resp. fibration) in the covariant model structure if and only if $f^\sharp$ is a cofibration (resp. weak equivalence, resp. fibration) in the coCartesian model structure on $\mathbf{dSets}^+/S$. This justifies the idea that we are `restricting' the coCartesian model structure along $(-)^\sharp$. Note that both the pairs $(u, (-)^\sharp)$ and $((-)^\sharp, \mathpzc{k})$ are Quillen pairs; the first by construction, the second by the characterization of cofibrations given above.
\end{remark}

\begin{remark}
The fibrant objects of $\mathbf{dSets}/S$ in the covariant model structure are the maps $p: X \longrightarrow S$ such that $p^\sharp: X^\sharp \longrightarrow S^\sharp$ is a fibrant object of $\mathbf{dSets}^+/S$. In this case $p$ is a coCartesian fibration and every corolla of $X$ is coCartesian; in other words, $p$ is a left fibration. Thus the fibrant objects of $\mathbf{dSets}/S$ are precisely the left fibrations. 
\end{remark}

As in the coCartesian model structure, there is a convenient characterization of weak equivalences between fibrant objects:

\begin{proposition}
Suppose we are given a diagram
\[
\xymatrix{
X \ar[rr]^f\ar[dr]_p & & Y \ar[dl]^q \\
 & S & 
}
\]
such that both $p$ and $q$ are left fibrations. Then the following are equivalent:
\begin{itemize}
\item[(1)] The map $f$ induces a covariant equivalence in $\mathbf{dSets}/S$
\item[(2)] The map $f$ is an operadic equivalence
\item[(3)] The map $f$ induces a weak homotopy equivalence of Kan complexes  $X_s \longrightarrow Y_s$ for each colour $s$ of $S$
\end{itemize}
\end{proposition}
\begin{proof}
By definition, (1) is equivalent to $f^\sharp$ being a coCartesian equivalence in $\mathbf{dSets}^+/S$. Since $X^\sharp \longrightarrow S^\sharp$ and $Y^\sharp \longrightarrow S^\sharp$ are fibrant objects of the latter category, we may apply Proposition \ref{prop:markedequivalence} to deduce that (1) is in fact equivalent to any of the following statements:
\begin{itemize}
\item[(2)] The map $f$ is an operadic equivalence
\item[(3')] The map $f$ induces a categorical equivalence $X_s \longrightarrow Y_s$ for each colour $s$ of $S$
\end{itemize}
Since the fibers $X_s$ and $Y_s$ are in fact Kan complexes by the fact that $p$ and $q$ are left fibrations, it follows that (3') is equivalent to (3): indeed, a map of Kan complexes is a categorical equivalence if and only if it is a weak homotopy equivalence. $\Box$ 
\end{proof}

\begin{definition}
Endow $\mathbf{dSets}/S$ with the covariant model structure. We define the $\infty$-category of left fibrations over $S$ by
\begin{equation*}
\mathbf{LFib}(S) := \mathrm{hc}N((\mathbf{dSets}/S)^\circ)
\end{equation*}
In case $S$ is normal, we define the $\infty$-category of \emph{$S$-algebras in spaces} by
\begin{equation*}
\mathrm{Alg}_S(\mathcal{S}) := \mathrm{hc}N((\mathrm{Alg}_{\mathrm{hc}\tau_d(S)}(\mathbf{sSets}))^\circ)
\end{equation*}
where $\mathbf{sSets}$ is endowed with the Quillen model structure.
\end{definition}

Inspecting our constructions, we may derive the following special case of Theorem \ref{thm:Grothendieckconstruction}:

\begin{theorem}
For a normal dendroidal set $S$, the straightening functor $St_S$ induces an equivalence of $\infty$-categories
\begin{equation*}
\mathbf{LFib}(S) \simeq \mathrm{Alg}_S(\mathcal{S})
\end{equation*}
\end{theorem}

We also get a naturality result analogous to the previous section:

\begin{theorem}
An operadic equivalence of dendroidal sets $\phi: S \longrightarrow T$ induces an equivalence of $\infty$-categories $\mathbf{LFib}(S) \simeq \mathbf{LFib}(T)$.
\end{theorem}

\begin{remark}
Note that by setting $S = N_d(\mathrm{Comm})$ we get an equivalence of $\infty$-categories
\begin{equation*}
\mathrm{hc}N(\mathbf{dSets}^\circ) \simeq \mathrm{Alg}_{E_\infty}(\mathcal{S})
\end{equation*}
where $\mathbf{dSets}$ is endowed with the covariant model structure. In this way we can obtain an infinite loop space machine for $\infty$-operads. We will discuss this situation in some more detail in \cite{heuts2}.
\end{remark}

\begin{remark}
An obvious but important warning: beware that the covariant model structure on $\mathbf{dSets}$ is definitely \emph{not} Quillen equivalent to the Cisinski-Moerdijk model structure on $\mathbf{dSets}$. In fact, it is a left Bousfield localization of it. More generally, if $S$ is an $\infty$-operad, the covariant model structure on $\mathbf{dSets}/S$ is a localization of the model structure induced from the Cisinski-Moerdijk model structure.
\end{remark}

\newpage

\appendix
\section{Appendix}
\subsection{An inner anodyne extension}
In the proof of Proposition \ref{prop:composition}, we used the following slight generalization of Lemma 7.2.4 of \cite{dendroidalsets}:

\begin{lemma}
\label{lemma:innerextension}
Suppose we are given trees $S, T \in \Omega$ and a leaf $l$ of $T$. Suppose furthermore that $R$ is a tree equipped with a monomorphism $R \longrightarrow S$ mapping the root of $R$ to the root of $S$. Denote the tree obtained by grafting $S$ onto $l$ by $T \circ_l S$ and similarly define $T \circ_l R$. Then the induced map
\begin{equation*}
\Omega[S] \coprod_{\Omega[R]} \Omega[T \circ_l R] \longrightarrow \Omega[T \circ_l S]
\end{equation*}
is inner anodyne.
\end{lemma}
\begin{proof}
We will proceed by induction on the number of vertices of $T$ and $S$. If either $S$ or $T$ equals $\eta$, the result is trivial.
Therefore, suppose both have one vertex, and denote the vertex of $T$ by $t$. If $R = S$, the result is again trivial, so suppose $R$ is the root of $S$. The stated map is now (isomorphic to) the inner horn inclusion
\begin{equation*}
\Lambda^l[T \circ_l S] \subseteq \Omega[T \circ_l S]
\end{equation*}
and hence inner anodyne. Now suppose $S$ has $n \geq 2$ vertices and the monomorphism $R \longrightarrow S$ is a face of $S$. Let $p: X \longrightarrow Y$ be an inner fibration of dendroidal sets. We need to show we can find a lift in the following diagram:
\[
\xymatrix{
\Omega[S] \coprod_{\Omega[R]} \Omega[T \circ_l R] \ar[d]\ar[r] & X \ar[d]^p \\
\Omega[T \circ_l S] \ar[r] & Y
}
\]
Consider the horn $\Lambda^l[T \circ_l S]$. It consists of the face $\partial_t \Omega[T \circ_l S] = \Omega[S]$, on which our map to $X$ has already been specified, and a bunch of faces of the form $\Omega[T] \circ_l \partial_\alpha \Omega[S]$ induced from the faces of $\Omega[S]$. By applying the inductive hypothesis to the maps
\begin{equation*}
\partial_\alpha \Omega[S] \coprod_{\Omega[R \cap \partial_\alpha S]} \Omega[T \circ_l (R \cap \partial_\alpha S)] \longrightarrow \Omega[T \circ_l \partial_\alpha S]
\end{equation*}
we see that our map to $X$ can be extended to these faces as well. (We use here that $\Omega[R \cap \partial_\alpha S]$ is either a face of $\partial_\alpha \Omega[S]$ or equal to it.) We have now built a diagram
\[
\xymatrix{
\Lambda^l[T \circ_l S] \ar[d]\ar[r] & X \ar[d]^p \\
\Omega[T \circ_l S] \ar[r] & Y
}
\]
in which a lift exists by the fact that $p$ is an inner fibration. This lift will also serve as a lift in our original diagram. \\
Now suppose $T$ has $n \geq 2$ vertices. The horn $\Lambda^l[T \circ_l S]$ now consists entirely of faces induced from faces of $S$ and $T$, so the inductive hypothesis combined with an argument similar to the one above yields the desired lift. \\
Finally, we wish to show the result holds for $R$ as in the proposition. We can factor the map $R \longrightarrow S$ as a composition of face maps. Denote this composition by
\begin{equation*}
R = A_0 \longrightarrow A_1 \longrightarrow \ldots \longrightarrow A_k = S
\end{equation*}
We already know that the induced map
\begin{equation*}
\Omega[A_1] \coprod_{\Omega[A_0]} \Omega[T \circ_l A_0] \longrightarrow \Omega[T \circ_l A_1]
\end{equation*}
is inner anodyne. Now consider, for $1 \leq i \leq k-1$, the diagram
\[
\xymatrix{
\Omega[A_i] \coprod_{\Omega[A_0]} \Omega[T \circ_l A_0] \ar[r]\ar[d] & \Omega[A_{i+1}] \coprod_{\Omega[A_0]} \Omega[T \circ_l A_0] \ar[d]\ar[ddr] & \\
\Omega[T \circ_l A_i] \ar[r]\ar[rrd] & P \ar[dr] & \\
&& \Omega[T \circ_l A_{i+1}]
}
\]
in which the square is a pushout. If we assume the left vertical map is inner anodyne, the right vertical map will be as well. The pushout $P$ may be identified with the pushout
\begin{equation*}
\Omega[A_{i+1}] \coprod_{\Omega[A_i]} \Omega[T \circ_l A_i]
\end{equation*}
By what we have already proven, we see that the map $P \longrightarrow \Omega[T \circ_l A_{i+1}]$ is inner anodyne. Hence
\begin{equation*}
\Omega[A_{i+1}] \coprod_{\Omega[A_0]} \Omega[T \circ_l A_0] \longrightarrow \Omega[T \circ_l A_{i+1}]
\end{equation*}
will be inner anodyne. By setting $i+1 = n$ we conclude that
\begin{equation*}
\Omega[S] \coprod_{\Omega[R]} \Omega[T \circ_l R] \longrightarrow \Omega[T \circ_l S]
\end{equation*}
is inner anodyne. $\Box$
\end{proof}

\subsection{Subdivision of foliage}
In this section we prove two lemma's needed in the proof of Proposition \ref{prop:smashproduct}. Recall the classes (1), (2), ($2^*$), (3) and (4) from the definition of marked anodyne morphisms. We will refer to morphisms in the weakly saturated class generated by (1) as \emph{inner anodynes} and morphisms in the weakly saturated class generated by (2) and ($2^*$) as \emph{leaf anodynes}. \par
Recall the notation
\begin{equation*}
\eta_{c_1, \ldots, c_n} := \coprod_{1\leq i \leq n} \eta_{c_i}
\end{equation*}
Then we have the following:

\begin{lemma}
\label{lemma:cylindersubdivision}
Let $S$ be an arbitrary tree and let $c_1, \ldots, c_n$ denote the leaves of $C_n$. The inclusion
\begin{equation*}
\eta_{c_1, \ldots, c_n}^\sharp \otimes \Omega[S]^\flat \coprod_{\eta_{c_1, \ldots, c_n} \otimes \partial \Omega[S]^\flat} \Omega[C_n]^\sharp \otimes \partial \Omega[S]^\flat \longrightarrow \Omega[C_n]^\sharp \otimes \Omega[S]^\flat
\end{equation*}
is marked anodyne.
\end{lemma}
\begin{proof}
In case $S = \eta$ the statement is trivial, so we may suppose $S$ has at least one vertex. Throughout this proof we will use several pictures to illustrate what is going on: a black vertex will represent the vertex of $C_n$ (let's call it $v$), a white vertex will represent a vertex of $S$. Denote the root of $C_n$ by $d$. We will also fix planar representations of our trees in order to avoid unnecessary discussions of automorphisms in $\Omega$. \par
The tensor product $\Omega[C_n]^\sharp \otimes \Omega[S]^\flat$ can be written as a union of all the shuffles of $C_n$ and $S$, say
\begin{equation*}
\Omega[C_n]^\sharp \otimes \Omega[S]^\flat = \bigl(\bigcup_{k=1}^N \Omega[R_k], \mathcal{E}\bigr)
\end{equation*}
Recall that the set of these shuffles has a partial order in which there is a minimal element given by grafting a copy of $S$ onto each leaf of $C_n$ and a maximal element obtained by grafting $n$-corollas onto all the leaves of $S$. \par
First, we assume for simplicity that $n \geq 1$. We set
\begin{equation*}
A_0 := \eta_{c_1, \ldots, c_n}^\sharp \otimes \Omega[S]^\flat \coprod_{\eta_{c_1, \ldots, c_n} \otimes \partial \Omega[S]^\flat} \Omega[C_n]^\sharp \otimes \partial \Omega[S]^\flat
\end{equation*}
Now define a filtration
\begin{equation*}
A_0 \subseteq A_1 \subseteq \cdots \subseteq A_N = \Omega[C_n]^\sharp \otimes \Omega[S]^\flat
\end{equation*}
of the given inclusion by successively adjoining all the shuffles of the tensor product in a way that respects the partial ordering on these shuffles. In other words, for $i > 0 $ each $A_{i+1}$ is obtained from $A_i$ by taking the union with a shuffle $R_k$ which can be obtained from a shuffle $R_l$ already contained in $A_i$ by replacing a configuration in $R_l$ of the form
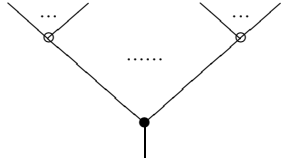
\[
\xymatrix@R=10pt@C=12pt{
&&&&&& \\
&*=0{\circ}\ar@{-}[ul]\ar@{-}[ur]\ar@{}[u]|{\cdots}&&&&*=0{\circ}\ar@{-}[ul]\ar@{-}[ur]\ar@{}[u]|{\cdots}&\\
&&&\ar@{}[u]|{\cdots\cdots}&&&\\
&&&*=0{\bullet}\ar@{-}\ar@{-}[uull]\ar@{-}[uurr]&&&\\
&&&\ar@{-}[u]&&&\\
&&&&&&
}
\]
by the configuration
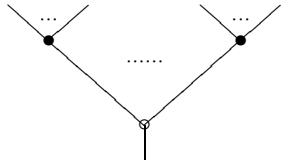
\[
\xymatrix@R=10pt@C=12pt{
&&&&&& \\
&*=0{\bullet}\ar@{-}[ul]\ar@{-}[ur]\ar@{}[u]|{\cdots}&&&&*=0{\bullet}\ar@{-}[ul]\ar@{-}[ur]\ar@{}[u]|{\cdots}&\\
&&&\ar@{}[u]|{\cdots\cdots}&&&\\
&&&*=0{\circ}\ar@{-}\ar@{-}[uull]\ar@{-}[uurr]&&&\\
&&&\ar@{-}[u]&&&\\
&&&&&&
}
\]
We will first examine the inclusion $A_0 \subseteq A_1$, which is the adjoining of the initial shuffle of our tensor product. Let $F_1(S)$ denote the set of all subtrees of $S$ containing the root $r$ of $S$. For $0 \leq j \leq M$ with $M := n \cdot |\mathrm{vert}(S)|$ we define $K_0^j$ to be the set of dendrices of $A_1$ of the form
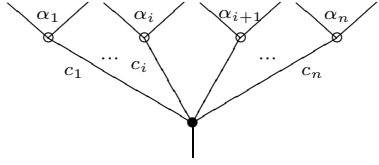
\[
\xymatrix@R=10pt@C=12pt{
&&&&&&&& \\
&*=0{\circ}\ar@{-}[ul]\ar@{-}[ur]\ar@{}[u]|{\alpha_1}&&*=0{\circ}\ar@{-}[ul]\ar@{-}[ur]\ar@{}[u]|{\alpha_i}&&*=0{\circ}\ar@{-}[ul]\ar@{-}[ur]\ar@{}[u]|{\alpha_{i+1}}&&*=0{\circ}\ar@{-}[ul]\ar@{-}[ur]\ar@{}[u]|{\alpha_n}&\\
&&\ar@{}[u]|{\quad\cdots}&&&\ar@{}[u]|{\quad\quad\cdots}&&&\\
&&&&*=0{\bullet}\ar@{-}\ar@{-}[uul]^(.75){c_i}\ar@{-}[uur]\ar@{-}[uulll]^(.75){c_1}\ar@{-}[uurrr]_(.75){c_n}&&&&\\
&&&&\ar@{-}[u]&&&&\\
&&&&&&&&
}
\]
where $\alpha_1, \ldots, \alpha_n$ are elements of $F_1(S)$ and the total number of white vertices equals $j$. The $c_i$ in the picture should more accurately be denoted by $c_i \otimes r$ with $r$ the root of $S$, but this is omitted for graphical reasons. Note that $K_0^0$ is contained in $A_0$ and $K_0^M$ is the initial shuffle of our tensor product. By setting $A_0^j = A_0 \cup K_0^j$ we obtain a filtration
\begin{equation*}
A_0 \subseteq A_0^1 \subseteq \cdots \subseteq A_0^M = A_1
\end{equation*}
Each inclusion $A_0^j \subseteq A_0^{j+1}$ is a composition of inclusions obtained by adjoining the elements of $K_0^{j+1}$ one at a time. Consider first the inclusion $A_0 \subseteq A_0^1$ and write it as such a composition of inclusions. Say we are adjoining an element $f$ of $K_0^1$. It might already factor through $\Omega[C_n]^\sharp \otimes \partial\Omega[S]^\flat$, in which case there is nothing to prove. If it doesn't, we consider the three faces of $f$: the face obtained by chopping of the white vertex, which is contained in $\Omega[C_n]^\sharp \otimes \partial\Omega[S]^\flat$, the face obtained by chopping of the black vertex, which is contained in $\eta_{c_1, \ldots, c_n}^\sharp \otimes \Omega[S]^\flat$, and an inner face which is not contained in any previous stage of our composition. Hence the map adjoining $f$ is a pushout of the map $\Lambda^e[f]^\flat \subseteq \Omega[f]^\flat$ (with $e$ denoting the inner edge of $f$) and therefore inner anodyne. We deduce that $A_0 \subseteq A_0^1$ is a composition of inner anodynes and thus inner anodyne. \par
Now consider an inclusion $A_0^j \subseteq A_0^{j+1}$, which is filtered by adjoining elements of $K_0^{j+1}$ one by one. If we adjoin an element $f \in K_0^{j+1}$ that factors through $\Omega[C_n]^\sharp \otimes \partial\Omega[S]^\flat$ there is again nothing to prove. If it doesn't, we consider its faces: 
\begin{itemize}
\item[(i)] We have faces chopping of a white vertex or contracting an inner edge connecting two white vertices; these are already contained in $A_0^j$.
\item[(ii)] Possibly there is a face chopping of the black vertex, which factors through  $\eta_{c_1, \ldots, c_n}^\sharp \otimes \Omega[S]^\flat$.
\item[(iii)] There are inner faces of the form $\partial_{c_i\otimes r} f$ for $1 \leq i \leq n$, which cannot be contained in any earlier stage of our filtration. Denote the set of these faces by $E$.
\end{itemize}
We see that the map adjoining $f$ is a pushout of the map $\Lambda^E[f]^\flat \subseteq \Omega[f]^\flat$. This is easily seen to be inner anodyne (cf. Lemma 7.2.3 of \cite{dendroidalsets}). We have shown that the map $A_0 \subseteq A_1$ is inner anodyne. \par 
Let now $A_i \subseteq A_{i+1}$ be any other map from our filtration. Say $A_{i+1}$ is obtained from $A_i$ by adjoining a shuffle $R_k$ and suppose $R_k$ is obtained from a shuffle $R_l$ which is contained in $A_i$ by percolating the vertex of $C_n$ up through a vertex of $S$, say $w$. If $w$ is a vertex with no inputs, then $R_k$ is actually a (composition of) face(s) of $R_l$ and there is nothing to prove. Hence we will assume $w$ has at least one input edge. We introduce some terminology: we will say a vertex of $R_k$ of the form $c_i \otimes s$ for $s$ a vertex of $S$ and $1 \leq i \leq n$ is \emph{above a black vertex} and a vertex of the form $d \otimes s$ is \emph{below the black vertices}. Recall that $d$ is the root of $C_n$. \par 
First, let $U$ denote the set of colours of $S$ such that the corolla $C_n \otimes s$ appears in $R_k$. For a subset $V \subseteq U$, we denote by $R_k^{(V)}$ the subtree of $R_k$ obtained by contracting all edges of the form $d \otimes u$ for $u \in U \backslash V$. Note that if $V$ does not contain any input edges of $w$ then $R_k^{(V)}$ is already contained in $A_i$ by the Boardman-Vogt relation. Therefore we will only consider $V$ containing at least one input edge of $V$. Also remark that if $V = U$ we have $R_k^{(V)} = R_k$. Fix a linear order on all $V \subseteq U$ containing an input edge of $V$ extending the partial order of inclusion. By adjoining the trees $R_k^{(V)}$ to $A_i$ in this order we obtain a filtration
\begin{equation*}
A_i = A_i^0 \subseteq A_i^1 \subseteq \cdots \subseteq A_i^l = A_{i+1}
\end{equation*}
We will refine this filtration even further. Given a map $A_i^j \subseteq A_i^{j+1}$ which is given by adjoining a tree $R_k^{(V)}$, we define $K_i^{j,m}$ to be the set of subtrees of $R_k^{(V)}$ which have a total of exactly $m$ white vertices above black vertices and which are equal to $R_k^{(V)}$ below (and at) the black vertices. In other words, the elements of $K_i^{j,m}$ are obtained from $R_k^{(V)}$ by chopping of white vertices which are above black vertices and by contracting inner edges connecting white vertices above black vertices. The integer $m$ ranges from 0 to the number of white vertices in $R_k$ which are above black vertices, which we will denote by $M$. By defining $A_i^{j,m} = A_i^j \cup K_i^{j,m}$ we obtain a filtration
\begin{equation*}
A_i^j \subseteq A_i^{j,0} \subseteq A_i^{j,1} \subseteq \cdots \subseteq A_i^{j,M} = A_i^{j+1} 
\end{equation*}
Note the analogy between this filtration and the one we considered for the inclusion $A_0 \subseteq A_1$. We first consider this filtration in case the map $A_i^j \subseteq A_i^{j+1}$ is given by adjoining a tree $R_k^{(V)}$ where $V$ is a singleton $\{e\}$, with $e$ necessarily being an input edge of $w$. The set $K_i^{j,0}$ contains only one tree, say $\tilde R_k^{(V)}$, which is obtained from $R_k^{(V)}$ by chopping off everything above the black vertices. If this tree is already contained in $\Omega[C_n]^\sharp \otimes \partial\Omega[S]^\flat$ then the inclusion $A_i^j \subseteq A_i^{j,0}$ is an equality. (Actually, this is always the case, unless $e$ is a leaf of $S$.) If not, we observe the following. The faces of $\tilde R_k^{(V)}$ are:
\begin{itemize}
\item[(i)] The face obtained by chopping of the corolla $C_n$ attached to $e$. This face cannot factor through any earlier stage of our filtration.
\item[(ii)] The inner face contracting $e$. This is contained in $A_i^j$ by the Boardman-Vogt relation.
\item[(iii)] Faces obtained by chopping off corollas other than the one considered in (i) and inner faces contracting an edge between two white corollas below the black vertices. All of these are already contained in $\Omega[C_n]^\sharp \otimes \partial\Omega[S]^\flat$.   
\end{itemize}
We conclude that the inclusion $\Omega[\tilde R_k^{(V)}] \cap A_i^j \subseteq \Omega[\tilde R_k^{(V)}]$ equals the horn inclusion 
\begin{equation*}
(\Lambda^v[\tilde R_k^{(V)}], \mathcal{E}_V') \subseteq (\Omega[\tilde R_k^{(V)}], \mathcal{E}_V)
\end{equation*}
where $\mathcal{E}_V$ denotes the set of all degenerate 1-corollas of $\Omega[\tilde R_k^{(V)]}$ and the corolla $C_n$ which is attached to the edge $e$. Also, $\mathcal{E}_V'$ denotes the intersection of this set with the corollas of the horn on the left. Note that by a slight abuse of notation we have again denoted the vertex of this corolla by $v$. Now remark that the map $A_i^j \subseteq A_i^{j,0}$ is a pushout of this horn inclusion, which is leaf anodyne, and hence this map is itself leaf anodyne. \par 
Consider now a map $A_i^{j,m} \subseteq A_i^{j,m+1}$. We can again write it as a composition of maps by adjoining the elements of $K_i^{j,m+1}$ one by one. Suppose we are adjoining such an element $f$. As before, if it factors through $\Omega[C_n]^\sharp \otimes \partial\Omega[S]^\flat$ there is nothing to prove. If not, consider its faces:
\begin{itemize}
\item[(i)] We have the faces obtained by chopping off a white vertex above a black vertex or contracting an inner edge between two white vertices above a black vertex. These are contained in $A_i^{j,m}$.
\item[(ii)] The inner face contracting the edge $e$. This is contained in $A_i^j$ by the Boardman-Vogt relation.
\item[(iii)]  Faces obtained by chopping off a corolla other than the $C_n$ attached to $e$, having a vertex which is not a white one above a black one. Also, inner faces contracting an edge between two white corollas below the black vertices. All of these are already contained in $\Omega[C_n]^\sharp \otimes \partial\Omega[S]^\flat$.
\item[(iv)] Inner faces contracting an edge between a white vertex directly above a black vertex and the vertex below it, which is the vertex of a leaf corolla of $\tilde R_k^{(V)}$. These faces cannot be contained in any earlier stage of our filtration. Denote the set of these by $E$.
\item[(v)] Possibly, a face obtained by chopping of the corolla $C_n$ attached to $e$. This can only happen if $f$ has no white vertices above this corolla.
\end{itemize} 
In case (v) does not occur, we are done: the map adjoining $f$ is inner anodyne. When (v) does occur, we can first adjoin the elements of $E$ one at a time: the face of each such element obtained by chopping off the corolla $C_n$ attached to $e$ cannot be contained in any earlier stage of our filtration either, so the resulting map adjoining this element will be a pushout of a leaf anodyne and hence leaf anodyne itself. After adjoining all of $E$, the map adjoining $f$ will then again be leaf anodyne, since the only face we're missing is (v). \par 
We have now shown that each map $A_i^j \subseteq A_i^{j+1}$ in our filtration given by adjoining a tree $R_k^{(V)}$ with $V$ a singleton is marked anodyne. We proceed by induction on the size of $V$. So, suppose $V$ contains at least two edges. Let $A_i^j \subseteq A_i^{j+1}$ be the corresponding map in our filtration. The analysis that follows is very similar to what was done above in case $V$ is a singleton, but we'll have to make some slight adaptations. The set $K_i^{j,0}$ contains only one tree, again denoted $\tilde R_k^{(V)}$, which is obtained from $R_k^{(V)}$ by chopping off everything above the black vertices. If this tree is already contained in $\Omega[C_n]^\sharp \otimes \partial\Omega[S]^\flat$ the inclusion $A_i^j \subseteq A_i^{j,0}$ is an equality. If not, we observe the following. The faces of $\tilde R_k^{(V)}$ are:
\begin{itemize}
\item[(i)] The faces obtained by chopping of a corolla $C_n$ attached to an element of $V$. These faces cannot factor through any earlier stage of our filtration. Denote the set of them by $L$.
\item[(ii)] The inner faces contracting edges in $V$. These are contained in $A_i^j$ by the inductive hypothesis on the size of $V$.
\item[(iii)] Faces obtained by chopping off corollas other than the ones considered in (i) and inner faces contracting an edge between two white corollas below the black vertices. All of these are already contained in $\Omega[C_n]^\sharp \otimes \partial\Omega[S]^\flat$.   
\end{itemize}
We see that the map adjoining $\tilde R_k^{(V)}$ is a pushout of the map 
\begin{equation*}
(\Lambda^L[\tilde R_k^{(V)}], \mathcal{E}_V') \subseteq (\Omega[\tilde R_k^{(V)}], \mathcal{E}_V)
\end{equation*}
where $\mathcal{E}_V$ now denotes the set of degenerate 1-corollas of $\Omega[\tilde R_k^{(V)}]$ and the corollas $C_n$ attached to the edges in $V$. The set $\mathcal{E}_V'$ is the intersecion of this with the horn on the left. Similar in spirit to Lemma 7.2.3 of \cite{dendroidalsets} one easily sees that this is leaf anodyne. \par  
We now consider a map $A_i^{j,m} \subseteq A_i^{j,m+1}$ from our filtration. The argument given above when we were considering $V$ a singleton is again easily adapted to the case where $V$ has more elements. This completes our analysis of the filtration of the map
\begin{equation*}
\eta_{c_1, \ldots, c_n}^\sharp \otimes \Omega[S]^\flat \coprod_{\eta_{c_1, \ldots, c_n} \otimes \partial \Omega[S]^\flat} \Omega[C_n]^\sharp \otimes \partial \Omega[S]^\flat \longrightarrow \Omega[C_n]^\sharp \otimes \Omega[S]^\flat
\end{equation*} 
which we have now shown to be marked anodyne. $\Box$
\end{proof}

Let $T$ be a tree with at least two vertices and let $v$ be the vertex of a leaf corolla of $T$. Let $\mathcal{E}$ be the union of all degenerate 1-corollas of $T$ together with the leaf corolla with vertex $v$. We introduce the notation
\begin{eqnarray*}
\Omega[T]^\diamondsuit & := & (\Omega[T], \mathcal{E}) \\
\Lambda^v[T]^\diamondsuit & := & (\Lambda^v[T], \mathcal{E} \cap \Lambda^v[T])
\end{eqnarray*}
We find the following result:

\begin{lemma}
\label{lemma:cylindersubdivision2}
Let $S$ be an arbitrary tree. The smash product
\begin{equation*}
(\Lambda^v[T]^\diamondsuit \otimes \Omega[S]^\flat) \coprod_{\Lambda^v[T]^\diamondsuit \otimes \partial\Omega[S]^\flat} (\Omega[T]^\diamondsuit \otimes \partial\Omega[S]^\flat) \longrightarrow \Omega[T]^\diamondsuit \otimes \Omega[S]^\flat
\end{equation*}
is marked anodyne.
\end{lemma}
\begin{proof}
The method is completely analogous to the proof of \ref{lemma:cylindersubdivision2}. (Expand) $\Box$
\end{proof}

\subsection{Combinatorial model categories}
For the reader's convenience we recall some results from Jeff Smith's treatment of combinatorial model categories. We just state what we need; a comprehensive treatment can be found in Appendix A.2.6 to \cite{htt}.

\begin{definition}
A model category $\mathbf{C}$ is said to be \emph{combinatorial} if it satisfies the following two conditions:
\begin{itemize}
\item[(1)] $\mathbf{C}$ is presentable
\item[(2)] The model structure on $\mathbf{C}$ is cofibrantly generated, i.e. there exists sets $I$ and $J$ such that the class of cofibrations (resp. trivial cofibrations) is the smallest weakly saturated class containing $I$ (resp. $J$)
\end{itemize}
\end{definition}

\begin{definition}
Let $\mathbf{C}$ be a presentable category. A class $W$ of morphisms of $\mathbf{C}$ is said to be \emph{perfect} if it satisfies the following conditions:
\begin{itemize}
\item[(1)] $W$ contains all isomorphisms
\item[(2)] $W$ satisfies the two-out-of-three property: given composable morphisms $f$ and $g$, if any two of the three morphisms $f$, $g$ and $g \circ f$ belong to $W$, they all do
\item[(3)] $W$ is stable under filtered colimits: given a filtered family of morphisms $\{X_\alpha \rightarrow Y_\alpha\}$ in $W$, the induced map $\varinjlim X_\alpha \rightarrow \varinjlim Y_\alpha$ is in $W$
\item[(4)] There exists a set (i.e. not a proper class) $W_0 \subseteq W$ such that $W_0$ generates $W$ under filtered colimits
\end{itemize}
\end{definition}

\begin{lemma}
\label{lemma:perfectclass}
If $F: \mathbf{C} \longrightarrow \mathbf{C}'$ is a functor between presentable categories which preserves filtered colimits and $W'$ is a perfect class of morphisms in $\mathbf{C}'$, then $F^{-1}(W')$ is a perfect class of morphisms in $\mathbf{C}$.
\end{lemma}

\begin{proposition}
\label{prop:combinmodelcat}
Let $\mathbf{C}$ be a presentable category. Suppose we are given classes $C$ and $W$ of morphisms of $\mathbf{C}$ such that $C$ is weakly saturated and generated by a set and $W$ is perfect. Suppose furthermore that $W$ is stable under pushouts by elements of $C$ and that any morphism having the right lifting property with respect to all morphisms in $C$ belongs to $W$. Then there exists a left proper combinatorial model structure on $\mathbf{C}$ in which the cofibrations are the elements of $C$, the weak equivalences are the elements of $W$ and the fibrations are the morphisms having the right lifting property with respect to every morphism in $C \cap W$.
\end{proposition}

\begin{lemma}
\label{lemma:combinlocalization}
Let $\mathbf{C}$ be a combinatorial simplicial model category and let $\mathbf{C}_f$ be the full simplicial subcategory of $\mathbf{C}$ on fibrant objects. Let $W$ denote the class of weak equivalences in $\mathbf{C}_f$. Let $\mathbf{C}_f[W^{-1}]$ be the localization of $\mathbf{C}_f$ obtained by formally inverting all elements of $W$. Then the functor
\begin{equation*}
\mathbf{C}^\circ \longrightarrow \mathbf{C}_f[W^{-1}]
\end{equation*} 
is a weak equivalence in the Bergner model structure on simplicial categories \cite{bergner}.
\end{lemma}

The following lemma has not been stated in the literature before, but is a version of Corollary A.3.6.18 of \cite{htt}. All of the necessary arguments given by Lurie involving projective model structures on simplicial functor categories carry through, mutatis mutandis, to the setting of projective model structures on categories of algebras over simplicial operads, and we obtain:   
\begin{lemma}
\label{lemma:combinlocalization2}
Suppose we are given a partially ordered set $I$ and define $I^+ = I \cup \{\infty\}$ by adjoining a largest element $\infty$. Suppose we are given a diagram
\begin{equation*}
\mathcal{D}: I^+ \longrightarrow \mathbf{sOper} 
\end{equation*} 
taking values in cofibrant simplicial operads, which exhibits $\mathcal{D}(\infty)$ as the colimit of the diagram $\mathcal{D}|_I$. Let $\mathbf{C}$ be a combinatorial simplicial model category an endow the categories $\mathrm{Alg}_{\mathcal{D}(i)}(\mathbf{C})$ with the projective model structure (cf. \cite{bergermoerdijk2}). Denote by $\mathrm{Alg}_{\mathcal{D}(i)}(\mathbf{C})_f$ their full simplicial subcategories on fibrant objects and let $W_i$ be the classes of weak equivalences in the respective categories. Then the homotopy limit (in the Bergner model structure) of the diagram $\{\mathrm{Alg}_{\mathcal{D}(i)}(\mathbf{C})_f[W_i^{-1}]\}_{i \in I}$ is the simplicial category
\begin{equation*}
\mathrm{Alg}_{\mathcal{D}(\infty)}(\mathbf{C})_f[W_\infty^{-1}]
\end{equation*}
\end{lemma}

Finally, we state a couple of lemmas which do not involve combinatorial model categories, but which could not find another place in this text. They can be found in sections A.2.3 and A.3.1 and A.3.2 of \cite{htt}.

\begin{lemma}
\label{lemma:extensionuptohomotopy}
Suppose we are given a diagram
\[
\xymatrix{
A' \ar[r]\ar[dd] & A \ar[dd]^i\ar[dr] & \\
 & & X \\
B' \ar[r] & B \ar@{-->}[ur] &
}
\]
in any model category $\mathbf{C}$, where $X$ is fibrant, $i$ is a cofibration between cofibrant objects and the horizontal arrows are weak equivalences. If we can find an extension $B' \longrightarrow X$ rendering the diagram commutative then the dotted extension also exists.
\end{lemma} 

\begin{lemma}
\label{lemma:equivsimpcat}
Let $\mathbf{C}$ and $\mathbf{D}$ be simplicial model categories and suppose that every object of $\mathbf{C}$ is cofibrant. Let
\[
\xymatrix@R=40pt@C=40pt{
F: \mathbf{C} \ar@<1ex>[r] & \mathbf{D}: G \ar@<1ex>[l]
}
\]
be a Quillen adjunction between the underlying model categories and suppose $G$ is a simplicial functor. Then the following are equivalent:
\begin{itemize}
\item[(1)] The Quillen pair $(F,G)$ is a Quillen equivalence
\item[(2)] By restriction $G$ induces a weak equivalence $\mathbf{D}^\circ \longrightarrow \mathbf{C}^\circ$ of simpicial categories in the Bergner model structure
\end{itemize} 
\end{lemma}

\begin{lemma}
\label{lemma:homlimsimplcat}
Let $I$ be a small category and let $\{\mathbf{C}_i\}_{i \in I}$ be a diagram of simplicial categories. Suppose we are given simplicial functors
\[
\xymatrix@C=50pt{
\mathbf{D} \ar[r]^{\mathcal{F}} & \mathbf{C} \ar[r]^{\mathcal{G}} & \varprojlim \{\mathbf{C}_i\}_{i \in I}
}
\]
such that $\mathcal{G} \circ \mathcal{F}$ exhibits $\mathbf{D}$ as a homotopy limit of the diagram $\{\mathbf{C}_i\}_{i \in I}$. Then the following are equivalent:
\begin{itemize}
\item[(1)] $\mathcal{G}$ exhibits $\mathbf{D}$ as a homotopy limit of $\{\mathbf{C}_i\}_{i \in I}$
\item[(2)] For every pair of objects $x, y \in \mathbf{C}$, the functor $\mathcal{G}$ exhibits $\mathrm{Map}_{\mathbf{C}}(A,B)$ as a homotopy limit of the diagram $\{\mathrm{Map}_{\mathbf{C_i}}(\mathcal{G}_i A, \mathcal{G}_i B)\}_{i \in I}$
\end{itemize}
\end{lemma}

\newpage


\bibliographystyle{plain}
\bibliography{biblio}
\par 

\textsc{Gijs Heuts} \\
Harvard University \\
gheuts@math.harvard.edu

\end{document}